\documentclass{article}

\usepackage{rotating}
\usepackage{tikz-cd}
\usepackage[all]{xy}
\usepackage{amssymb}
\usepackage{amsmath}
\usepackage{amsthm}
\newtheorem{thm}{Theorem}
\newtheorem{prop}{Proposition}
\newtheorem{lem}{Lemma}
\newtheorem{rem}{Remark}
\newtheorem{cor}{Corollary}
\newtheorem{defi}{Definition}

\def\Top{\mathop{\rm Top}\nolimits}
\def\RTop{\mathop{\rm RTop}\nolimits}
\def\Sch{\mathop{\rm Sch}\nolimits}

\def\Pic{\mathop{\rm Pic}\nolimits}

\def\codim{\mathop{\rm codim}\nolimits}
\def\dim{\mathop{\rm dim}\nolimits}
\def\deg{\mathop{\rm deg}\nolimits}
\def\Tr{\mathop{\rm Tr}\nolimits}
\def\Cor{\mathop{\rm Cor}\nolimits}

\def\SmVar{\mathop{\rm SmVar}\nolimits}

\def\PSmVar{\mathop{\rm PSmVar}\nolimits}
\def\Var{\mathop{\rm Var}\nolimits}
\def\Hom{\mathop{\rm Hom}\nolimits}
\def\Spec{\mathop{\rm Spec}\nolimits}
\def\Coh{\mathop{\rm Coh}\nolimits}
\def\supp{\mathop{\rm supp}\nolimits}
\def\QPVar{\mathop{\rm QPVar}\nolimits}
\def\AnSp{\mathop{\rm AnSp}\nolimits}
\def\CW{\mathop{\rm CW}\nolimits}

\def\PVar{\mathop{\rm PVar}\nolimits}

\def\sing{\mathop{\rm sing}\nolimits}

\def\Im{\mathop{\rm Im}\nolimits}

\def\Cone{\mathop{\rm Cone}\nolimits}
\def\ad{\mathop{\rm ad}\nolimits}

\def\log{\mathop{\rm log}\nolimits}
\def\Diff{\mathop{\rm Diff}\nolimits}
\def\An{\mathop{\rm An}\nolimits}

\def\PSh{\mathop{\rm PSh}\nolimits}
\def\AnSm{\mathop{\rm AnSm}\nolimits}
\def\Tr{\mathop{\rm Tr}\nolimits}

\def\DA{\mathop{\rm DA}\nolimits}

\def\Fun{\mathop{\rm Fun}\nolimits}

\def\coker{\mathop{\rm coker}\nolimits}

\def\Cat{\mathop{\rm Cat}\nolimits}
\def\RCat{\mathop{\rm RCat}\nolimits}

\def\an{\mathop{\rm an}\nolimits}

\def\Mod{\mathop{\rm Mod}\nolimits}

\def\Shv{\mathop{\rm Shv}\nolimits}

\def\Vect{\mathop{\rm Vect}\nolimits}

\def\PSch{\mathop{\rm PSch}\nolimits}
\def\RigVar{\mathop{\rm RigVar}\nolimits}
\def\AbVar{\mathop{\rm AbVar}\nolimits}
\def\Br{\mathop{\rm Br}\nolimits}
\def\Pro{\mathop{\rm Pro}\nolimits}

\def\dar[#1]{\ar@<2pt>[#1]\ar@<-2pt>[#1]}

\title{De Rham logarithmic classes and Tate conjecture}

\author{Johann Bouali}

\begin{document}

\maketitle

\begin{abstract}
We introduce the notion of De Rham logarithmic classes. 
We show that the De Rham class of an algebraic cycle of a smooth algebraic variety over a field of characteristic zero
is logarithmic and conversely that a logarithmic class of bidegree $(d,d)$ 
is the De Rham class of an algebraic cycle (of codimension $d$). Moreover for smooth projective varieties, 
we show that there are no non trivial logarithmic classes of bidegree $(p,q)$ for $p\neq q$.
For smooth algebraic varieties over a $p$-adic field, we also give an analytic version of this result.
We deduce, from the analytic version, the Tate conjecture for smooth projective varieties over fields of finite type over $\mathbb Q$.
\end{abstract}

\section{Introduction}

\subsection{De Rham logarithmic classes}

In this work, we introduce in a first part for $X$ a noetherian scheme, the notion of logarithmic De Rham cohomology classes
which are for each $j\in\mathbb Z$, the elements of the subgroup 
$H^jOL_X(\mathbb H^j_{et}(X,\Omega^{\bullet}_{X,\log}))\subset\mathbb H^j_{et}(X,\Omega_X^{\bullet})=:H^j_{DR}(X)$
of the De Rham cohomology abelian group, where $OL_X:\Omega^{\bullet}_{X,\log}\hookrightarrow\Omega_X^{\bullet}$
is the subcomplex of abelian sheaves on $X$ consisting of logarithmic forms introduced in definition \ref{wlogdef}. 
To our knowledge, the notion of logarithmic forms was introduced
in the seventies by S.Bloch for varieties over perfect fields of charactersitic $p$ 
(\cite{Illusie}, §3F) in order to compute the Frobenius fixed part of the De Rham-Witt complex. 
Note that this notion is different from the notion of logarithmic forms along a divisor with normal crossing 
as introduced by Deligne for exemple in Th\'eorie de Hodge II.
The presheaves $\Omega^{\bullet}_{X,\log}$, as for $O_X^*$ are not $O_X$ modules, 
nor locally constant but they have good purity properties.
Our main result is that if $X$ is a smooth algebraic variety over a field of characteristic zero, a de Rham cohomology class
which is logarithmic of type $(d,d)$ is the class of an algebraic cycle of codimension $d$.
More precisely, let $X$ be a smooth algebraic variety over a field $k$ of characteristic zero. 
To each algebraic cycle $Z\in\mathcal Z^d(X)$, we associate its De Rham cohomology class $[Z]\in H^{2d}_{DR}(X)$,
which is by definition, as for any Weil cohomology theory, the image of the fundamental class $[Z]\in H^{2d}_{DR,Z}(X)$
by the canonical morphism $H^{2d}_{DR,Z}(X)\to H^{2d}(X)$.
In section 3, we show (c.f. theorem \ref{drlogZ}(i)) that 
\begin{itemize}
\item for $Z\in\mathcal Z^d(X)$, $[Z]\in H^{2d}_{DR}(X)$ is logarithmic of bidegree $(d,d)$, that is 
\begin{equation*}
[Z]=H^{2d}OL_X([Z]_L)\in H^{2d}_{DR}(X), \; [Z]_L\in H^d_{et}(X,\Omega^d_{X,\log})
\end{equation*}
where $H^d_{et}(X,\Omega^d_{X,\log})\subset\mathbb H^{2d}_{et}(X,\Omega^{\bullet}_{X,\log})$ is the canonical subspace,
as all the differentials of $\Omega^{\bullet}_{X,\log}$ vanishes since by definition a logarithmic form is closed.
This fact is a consequence of the fact that
motivic isomorphisms applied to De Rham cohomology preserve logarithmic classes (c.f. proposition \ref{smXZ}). 
The key point is that logarithmic De Rham forms 
on algebraic varieties over $k$ are closed, (trivially) $\mathbb A_k^1$ invariant
and compatible with the transfers maps induced by finite morphisms
of algebraic varieties and in particular finite correspondences. 
Since $\Omega^l_{\log}\subset\Omega^l$ are (trivially) $\mathbb A_k^1$ invariant
presheaves with transfers on the category of smooth algebraic varieties over $k$, by a theorem of Voevodsky
the cohomology presheaves of $\mathbb A_k^1$ invariant presheaves with transfers are $\mathbb A_k^1$ invariant.
\item Conversely, we show (c.f. theorem \ref{drlogZ}(ii)), 
that a logarithmic class $w\in H^{2d}_{DR}(X)$ of bidegree $(d,d)$ is the De Rham class of an algebraic cycle 
(of codimension $d$).
In the case $X$ projective, we also get a vanishing result (c.f. theorem \ref{drlogZ}(iii)') : 
$H^{2d}OL_X(H^{d-k}_{et}(X,\Omega^{d+k}_{X,\log}))=0$ for $k>0$. \\
The proof works as follows :
A logarithmic class of bidegree $(p,q)$ is locally acyclic for the Zariski topology of $X$ since
it is the etale cohomology of a single sheaf. This allows us to proceed by a finite induction 
using the crucial fact that the purity isomorphism for De Rham cohomology preserve logarithmic classes
(c.f. proposition \ref{drlogZcor}). The proof of proposition \ref{drlogZcor} follows from the fact 
that the purity isomorphism is motivic (c.f. \cite{CD}, see proposition \ref{smXZ}), that
the Euler class of a vector bundle of rank $d$ 
over an algebraic variety is logarithmic of bidegree $(d,d)$ (c.f. proposition \ref{drlogZprop})
and that the motivic isomorphisms applied to De Rham cohomology preserve logarithmic classes 
(c.f. proposition \ref{smXZlog}). At the final step, we use, for simplicity, 
the fact that for a scheme $Y$, $H^1_{et}(Y,\Omega^1_{Y,\log})=H^1(Y,O^*_Y)$ is the Picard group of $Y$.
\end{itemize}

For $X$ an algebraic variety over a $p$-adic field $K$, we also introduce the notion of 
logarithmic analytic (more precisely formal) De Rham cohomology classes. In proposition \ref{GAGAlog}
we give an analytic analogue of theorem \ref{drlogZ}, 
that is if $X$ is a smooth projective variety over a $p$-adic field $K$ with good reduction and $X^{\mathcal O}$ is a smooth integral model of $X$, 
denoting $c:\hat X^{\mathcal O}\to X^{\mathcal O}$ the morphism of ringed spaces given by the completion along the ideal generated by $p$,
\begin{itemize}
\item a logarithmic analytic De Rham cohomology class of type $(d,d)$ is the class of a codimension $d$ algebraic formal cycle,
\item there are no non trivial logarithmic analytic De Rham cohomology classes of type $(p,q)$ for $p\neq q$, i.e.
$H^{2d}OL_{\hat X}(H^{d-k}_{pet}(X,\Omega^{d+k}_{\hat X^{(p)},\log,\mathcal O}))=0$ for $k\neq 0$.
\end{itemize}
The proof of proposition \ref{GAGAlog} has two part. 
First, by constructing a (non canonical) functorial splitting of the sub groups of formal logarithmic forms inside the De Rham formal forms,
we note that a logarithmic formal class of bidegree $(p,q)$ is acyclic for the pro-etale topology on each etale map of integral models
$e:U^{\mathcal O}\to X^{\mathcal O}$, $U^{\mathcal O}$ affine such that such that the tangent bundle is trivial 
and $e^*\Br(X^{\mathcal O}/p)[p]=0$ (proposition \ref{UXpet}) and we prove 
\begin{equation*}
H^q_{pet}(X^{\mathcal O},\Omega^p_{\hat X^{(p)},\log,\mathcal O})=H^q_{et}(X^{\mathcal O},\Omega^p_{\hat X^{(p)},\log,\mathcal O}), \;
H^q_{pet}(X,r^*\Omega^p_{\hat X^{(p)},\log,\mathcal O})=H^q_{et}(X,r^*\Omega^p_{\hat X^{(p)},\log,\mathcal O}), \; p,q\in\mathbb Z,
\end{equation*}
where $X^{\mathcal O,et}:=(X^{\mathcal O})^{et}\cap((\Sch^{int}/O_K)/X^{\mathcal O})$ 
is the subcategory of $(X^{\mathcal O})^{et}$ consisting of integral models, 
$X^{\mathcal O,pet}:=\Pro(X^{\mathcal O,et})$ is the category of pro-objects of $X^{\mathcal O,et}$ 
and $r:X^{et}\to X^{\mathcal O,et}$ and $r:X^{pet}\to X^{\mathcal O,pet}$ are the canonical morphism of sites.
The second part is motivic and inspired on the proof of theorem \ref{drlogZ}, considering the big site of smooth integral models,
a logarithmic analytic class of bidegree $(p,q)$ is locally acyclic for the Zariski topology of a closed subset $D\subset X$ since
it is the etale cohomology of a single sheaf. Note that $D$ has bad reduction in general.
This allows us to proceed by a finite induction 
using the crucial fact that the purity isomorphism for De Rham cohomology preserve logarithmic formal classes
since the purity isomorphism is motivic (c.f. \cite{CD}, see proposition \ref{smXZ}).

\subsection{Tate conjecture}

Let $X$ be a smooth projective variety over a field $k$ of finite type over $\mathbb Q$. 
Up to split $X$ by its connected components, we may assume that $X$ is connected, hence irreducible, of dimension $d_X$.
Let $p$ be a prime number unramified over $k$ such that $X$ has good reduction at $p$.
We fix an embedding $\sigma_p:k\hookrightarrow\mathbb C_p$ and we denote $\bar k$ the algebraic closure of $k$ inside $\mathbb C_p$. 
We prove the Tate conjecture for $X$ with $\mathbb Q_p$ coefficient (c.f. theorem \ref{Tate}).
The proof works as follows :
let $\hat k_{\sigma_p}\subset\mathbb C_p$ be the $p$-adic completion of $k$ with respect to $\sigma_p$. 
We choose a smooth proper scheme $X^{\mathcal O}_{\hat k_{\sigma_p}}$ over $O_{\hat k_{\sigma_p}}$ which is an integral model of $X_{\hat k_{\sigma_p}}$, 
i.e. such that $X_{\hat k_{\sigma_p}}^{\mathcal O}\times_{O_{\hat k_{\sigma_p}}}\hat k_{\sigma_p}=X_{\hat k_{\sigma_p}}$ 
with $X_{\hat k_{\sigma_p}}^{\mathcal O}$ irreducible and surjective over $O_{\hat k_{\sigma_p}}$.
We will consider $c:\hat X_{\hat k_{\sigma_p}}^{\mathcal O}\to X_{\hat k_{\sigma_p}}^{\mathcal O}$ 
the morphism of ringed spaces which is the formal completion of $X_{\hat k_{\sigma_p}}^{\mathcal O}$ along the ideal generated by $p$.
Let $\alpha\in H_{et}^{2d}(X_{\bar k},\mathbb Q_p)(d)^G$ be a Tate class of $X$, where $G:=Gal(\bar k/k)$.
Denote by $G_p:=Gal(\mathbb C_p/\hat k_{\sigma_p})\subset G$ the local Galois group. In proposition \ref{LatticeLog}, we show, 
using a result of \cite{Illusie} together with results of \cite{Sch} on the pro-etale topology,
that $\alpha$ gives by the $p$-adic crystalline comparison isomorphism
\begin{equation*}
H^{2d}R\alpha(X):H_{et}^{2d}(X_{\mathbb C_p},\mathbb Z_p)\otimes_{\mathbb Z_p}\mathbb B_{cris,\hat k_{\sigma_p}}
\xrightarrow{\sim}H^{2d}_{DR}(X)\otimes_k\mathbb B_{cris,\hat k_{\sigma_p}},
\end{equation*}
an analytic logarithmic de Rham class 
\begin{equation*}
w(\alpha):=H^{2d}R\alpha(X)(\alpha)\in 
H^{2d}OL_{X_{\hat k_{\sigma_p}}}(\mathbb H_{pet}^{2d}(X_{\hat k_{\sigma_p}},\Omega^{\bullet\geq d}_{\hat X^{(p)}_{\hat k_{\sigma_p}},\log,\mathcal O}))
\subset H^{2d}_{DR}(\hat X^{\mathcal O}_{\hat k_{\sigma_p}})\otimes_{O_{\hat k_{\sigma_p}}}\hat k_{\sigma_p}=H_{DR}^{2d}(X_{\hat k_{\sigma_p}}).
\end{equation*}
Proposition \ref{GAGAlog} then implies that 
\begin{equation*}
w(\alpha)\in\Im(\Omega(\gamma^{\vee}_{Z}):H^{2d}_{DR,\hat Z^{\mathcal O}}(\hat X^{\mathcal O}_{\hat k_{\sigma_p}})\to 
H^{2d}_{DR}(\hat X^{\mathcal O}_{\hat k_{\sigma_p}})\otimes_{O_{\hat k_{\sigma_p}}}\hat k_{\sigma_p}=H_{DR}^{2d}(X_{\hat k_{\sigma_p}})), 
\end{equation*}
where $Z^{\mathcal O}=\cup_{i=1}^rZ_i^{\mathcal O}\subset X^{\mathcal O}_{\hat k_{\sigma_p}}$ is a zariski closed subset 
such that all irreducible components $(Z_i^{\mathcal O})_{1\leq i\leq r}$ are of codimension $d$ 
and (proper) surjective over the integral ring $O_{\hat k_{\sigma_p}}$ of $\hat k_{\sigma_p}$, 
and $Z:=Z^{\mathcal O}\times_{O_{\hat k_{\sigma_p}}}\hat k_{\sigma_p}$. 
Here for $Y$ an algebraic variety over $\hat k_{\sigma_p}$ and $Y^{\mathcal O}$ is an integral model of $Y$,
$H^i_{DR}(\hat Y^{\mathcal O}):=\mathbb H^i_{et}(Y^{\mathcal O},\Omega_{\hat Y^{\mathcal O}}^{\bullet})$
is the formal de Rham cohomology of $Y^{\mathcal O}$ (see definition \ref{wlogdef}). 
Note that it does depends on the integral model since it is the limit of the de Rham complexes but NOT the de Rham complex of the limit. 
Also note that if $Y$ is smooth projective and admits a smooth projective integral model then the integral model is unique in codimension one of the special fiber. 
But if $Y$ is smooth but non proper and admits a smooth integral model then it has infinitely many special fiber birational classes of smooth integral models.
We get
\begin{equation*}
w(\alpha)=\sum_{i=1}^r\sum_{j=1}^{s_i}n_{ij}[(Z_i^{\mathcal O}/p)_j]\in 
H^{2d}_{DR}(\hat X^{\mathcal O}_{\hat k_{\sigma_p}})\otimes_{O_{\hat k_{\sigma_p}}}\hat k_{\sigma_p}=H_{DR}^{2d}(X_{\hat k_{\sigma_p}}), 
\end{equation*}
where $n_{ij}\in\mathbb Q_p$ and $((Z_i^{\mathcal O}/p)_j)_{1\leq j\leq s_i}$ are the irreducible components of $Z_i^{\mathcal O}/p$ 
which are all of codimension $d$ in $X^{\mathcal O}_{\hat k_{\sigma_p}}/p$ since the reduction mod $p$ consists on one equation. 
Note that $Z^{\mathcal O}_i$ has bad reduction mod $p$ in general, in particular $Z_i^{\mathcal O}/p$ may be reducible while $Z_i^{\mathcal O}$ is irreducible. 
By lemma \ref{Aut}, there exists an etale morphism $r_e:X^{\mathcal O}_e\to X^{\mathcal O}_{\hat k_{\sigma_p}}$ of integral models over $O_{\hat k_{\sigma_p}}$ 
such that for each $1\leq i\leq r$ we have $r_e(X^{\mathcal O}_e)\cap Z_i\neq\emptyset$ and $Aut_{O_{\hat k_{\sigma_p}}}(X^{\mathcal O}_e,Z^{\mathcal O})$ acts transitively
on $(T_{ij}:=r_e^{-1}((Z_i^{\mathcal O}/p)_j))_{1\leq j\leq s_i}$, where $Aut_{O_{\hat k_{\sigma_p}}}(X^{\mathcal O}_e,Z^{\mathcal O})$ denote the automorphisms $g$ of 
$X^{\mathcal O}_e$ over $O_{\hat k_{\sigma_p}}$ such that $g(r_e^{-1}(Z^{\mathcal O}))=r_e^{-1}(Z^{\mathcal O})$.
Take a compactification $\bar r_e:\bar X^{\mathcal O}_e\to X^{\mathcal O}_{\hat k_{\sigma_p}}$ of $r_e$ which is a finite surjective morphism.
Consider then the finite surjective morphism of integral models over $O_{\hat k_{\sigma_p}}$ which is the composite
\begin{equation*}
r':X^{'\mathcal O}\xrightarrow{\epsilon^{\mathcal O}}\bar X^{\mathcal O}_e\xrightarrow{\bar r_e}X^{\mathcal O}_{\hat k_{\sigma_p}},
\end{equation*}
where $\epsilon:(X',E)\to(\bar X_e,\mathcal R)$ is a desingularization of the pair $(\bar X_e,\mathcal R)$ of projective varieties over $\hat k_{\sigma_p}$,
where $\mathcal R$ is the ramification locus of $\bar r_e\otimes_{O_{\hat k_{\sigma_p}}}\hat k_{\sigma_p}$.
We have then, considering $Z^{\mathcal O}/p=\cup_{i=1}^r\cup_{j=1}^{s_i}(Z_i^{\mathcal O}/p)_j$,
\begin{equation*}
Z^{'\mathcal O}:=r^{'-1}(Z^{\mathcal O}), \; Z^{'\mathcal O}/p=\cup_{i=1}^r\cup_{j=1}^{s_i}T_{ij}, \; T_{ij}:=r^{'-1}((Z_i^{\mathcal O}/p)_j).
\end{equation*}
Note that $X'$ has bad reduction in general. In particular $X^{'\mathcal O}$ is NOT smooth. 
Consider the factorization $r':X^{'\mathcal O}\xrightarrow{i}\mathbb P^N\times X^{\mathcal O}\xrightarrow{p}X^{\mathcal O}$ 
where $i$ is the graph closed embedding and $p$ is the projection. 
Denote by $C^{\mathcal O}:=C_{X^{'\mathcal O}/\mathbb P^N\times X^{\mathcal O}}\xrightarrow{p'}X^{'\mathcal O}$ the normal cone of $i$.
Since $X'$ is smooth $C:=C^{\mathcal O}\otimes_{O_{\hat k_{\sigma_p}}}\hat k_{\sigma_p}=N_{X'/\mathbb P^N\times X}\xrightarrow{p'}X'$ is a vector bundle.
We denote again by $c:\widehat{\mathbb P^N\times X^{\mathcal O}}\to \mathbb P^N\times X^{\mathcal O}$, $c:\hat C^{\mathcal O}\to C^{\mathcal O}$ the formal completion maps.
In particular $C^{\mathcal O}$ is not smooth but $C$ is smooth. 
Let $k\subset k'\subset\mathbb C_p$ be a subfield of finite type over $\mathbb Q$ over which 
$(Z_i\subset X_{\hat k_{\sigma_p}})_{1\leq i\leq r}$ and $r':X'\to X_{\hat k_{\sigma_p}}$ are defined.
Take an (algebraic) embedding of field $k'(X')\hookrightarrow\mathbb C_p\simeq\mathbb C$. Then 
\begin{equation*}
A:=Aut_{O_{\hat k_{\sigma_p}}}(X^{'\mathcal O},Z^{\mathcal O})\subset Aut_{k'}(k'(X'))\subset Aut_{k'}(\mathbb C_p)\subset G=Aut_{k}(\mathbb C_p)
\end{equation*}
Note that $\hat k_{\sigma_p}(X')$ is NOT a p-adic field and that $\hat k_{\sigma_p}(X')\hookrightarrow\mathbb C_p$.
We have then the following commutative diagram
\begin{equation*}
\xymatrix{H_{et}^{2d}(C_{\bar k},\mathbb Q_p)(d)^{G_p}\ar[rr]^{\alpha\mapsto w(\alpha)} & \, &
H_{DR}^{2d}(C)\ar[r]^{c^*} & H_{DR}^{2d}(\hat C^{\mathcal O})\otimes_{O_{\hat k_{\sigma_p}}}\hat k_{\sigma_p}  \\
H_{et}^{2d}(X_{\bar k},\mathbb Q_p)(d)^{G_p}\ar[rr]^{\alpha\mapsto w(\alpha)}\ar[u]^{p^{'*}r^{'*}} & \, &
H_{DR}^{2d}(X_{\hat k_{\sigma_p}})\ar[r]^{c^*}\ar[u]^{p^{'*}r^{'*}} & H_{DR}^{2d}(\hat X^{\mathcal O})\otimes_{O_{\hat k_{\sigma_p}}}\hat k_{\sigma_p}
\ar[u]^{p^{'*}r^{'*}}}
\end{equation*}
whose arrows of the upper row commute with the action of $Aut_{O_{\hat k_{\sigma_p}}}(C^{\mathcal O})\subset G=Aut_{k}(\mathbb C_p)$, in particular with the action of
\begin{equation*}
A:=Aut_{O_{\hat k_{\sigma_p}}}(X^{'\mathcal O},Z^{\mathcal O})\subset Aut_{O_{\hat k_{\sigma_p}}}(C^{\mathcal O})\subset G=Aut_{k}(\mathbb C_p). 
\end{equation*}
The first arrows of the rows are the maps of the p-adic De Rham comparison theorem for $X_{\mathbb C_p}$ and $C_{\mathbb C_p}$ respectively ($C$ and $X$ are smooth).
Since $\alpha\in H_{et}^{2d}(X_{\bar k},\mathbb Q_p)(d)^G$, we have 
$p^{'*}r^{'*}\alpha\in H_{et}^{2d}(C_{\bar k},\mathbb Q_p) (d)^A$ for each $1\leq i\leq r$.
We have then
\begin{itemize}
\item $p^{'*}r^{'*}w(\alpha)=c^*w(p^{'*}r^{'*}\alpha)\in(H_{DR}^{2d}(\hat C^{\mathcal O})\otimes_{O_{\hat k_{\sigma_p}}}\hat k_{\sigma_p})^A$
\item and
\begin{equation*}
p^{'*}r^{'*}w(\alpha)=\sum_{i=1}^r\sum_{j=1}^{s_i}n_{ij}[T_{ij}]\in
\Im(H_{DR,\widehat{p^{'-1}(Z^{'\mathcal O})}}^{2d}(\hat C^{\mathcal O})\to H_{DR}^{2d}(\hat C^{\mathcal O})\otimes_{O_{\hat k_{\sigma_p}}}\hat k_{\sigma_p}), 
\end{equation*}
since 
\begin{equation*}
w(\alpha)=\sum_{i=1}^r\sum_{j=1}^{s_i}n_{ij}[(Z_i^{\mathcal O}/p)_j]\in 
H^{2d}_{DR}(\hat X^{\mathcal O}_{\hat k_{\sigma_p}})\otimes_{O_{\hat k_{\sigma_p}}}\hat k_{\sigma_p}=H_{DR}^{2d}(X_{\hat k_{\sigma_p}}), 
\end{equation*}
\end{itemize}
Take for each $1\leq i\leq r$, $1\leq s'_i\leq s_i$ such that
\begin{equation*}
([(Z_i^{\mathcal O}/p)_j])_{1\leq i\leq r,1\leq j\leq s'_i}\in H^{2d}_{DR}(\hat X^{\mathcal O}_{\hat k_{\sigma_p}})
\otimes_{O_{\hat k_{\sigma_p}}}\hat k_{\sigma_p}=H_{DR}^{2d}(X_{\hat k_{\sigma_p}})
\end{equation*}
are linearly independent. Hence 
\begin{equation*}
([T_{ij}])_{1\leq i\leq r,1\leq j\leq s'_i}\in H_{DR}^{2d}(\hat C^{\mathcal O})\otimes_{O_{\hat k_{\sigma_p}}}\hat k_{\sigma_p}
\end{equation*}
are linearly independent since
\begin{eqnarray*}
r'_*([T_{ij}])=[(Z_i^{\mathcal O}/p)_j], \; 
r'_*:H_{DR,\hat Z^{'\mathcal O}}^{2d}(\hat C^{\mathcal O})\to H_{DR,\hat Z^{\mathcal O}}^{2d}(\hat X^{\mathcal O})\to H_{DR}^{2d}(\hat X^{\mathcal O}).
\end{eqnarray*}
We get, since $A$ acts transitively on $(T_{ij}:=r^{'-1}((Z_i^{\mathcal O}/p)_j))_{1\leq j\leq s_i}$ for each $1\leq i\leq r$,
\begin{equation*}
n_{ij}=n_i, \; \mbox{for each} \; 1\leq i\leq r \; \mbox{and all} \; 1\leq j\leq s_i.
\end{equation*}
This gives 
\begin{eqnarray*}
w(\alpha)=\sum_{i=1}^rn_i[Z_i^{\mathcal O}/p]=\sum_{i=1}^rn_i[Z_i^{\mathcal O}]\in 
H^{2d}_{DR}(\hat X^{\mathcal O}_{\hat k_{\sigma_p}})\otimes_{O_{\hat k_{\sigma_p}}}\hat k_{\sigma_p}=H_{DR}^{2d}(X_{\hat k_{\sigma_p}}), 
\end{eqnarray*}
that is
\begin{equation*}
w(\alpha)=[Z]\in H_{DR}^{2d}(X_{\hat k_{\sigma_p}}), \; Z:=\sum_{i=1}^rn_i[Z_i]\in\mathcal Z^d(X_{\hat k_{\sigma_p}})\otimes\mathbb Q_p.
\end{equation*} 
Note that the local Galois group $G_p:=Gal(\mathbb C_p/\hat k_{\sigma_p})\subset G$ fixes the components $T_{ij}$ of $r^{'-1}(Z^{\mathcal O}_i/p)$ for each $1\leq i\leq r$,
since classes of formal algebraic cycles are $G_p$ invariant. Here is where we need that the class $\alpha$ is invariant by the absolute Galois group $G$ of $k$,
not just invariant by $G_p$.
Since the Hilbert schemes of $X$ are defined over $k$, we get 
\begin{equation*}
w(\alpha)=[Z^N]\in H_{DR}^{2d}(X_{\hat k_{\sigma_p}}), \; Z^N\in\mathcal Z^d(X_{\bar k})\otimes\mathbb Q_p.
\end{equation*} 
By the $p$-adic crystalline or de Rham comparison isomorphism, we get $\alpha=[Z^N_{\mathbb C_p}]\in H_{et}^{2d}(X_{\mathbb C_p},\mathbb Q_p)$.
Since $\alpha$ is $G$ invariant and since $Z_N$ is defined over $\bar k$, we get
\begin{equation*}
\alpha=[Z^N]=[Z^{N'}]\in H_{et}^{2d}(X_{\bar k},\mathbb Q_p), 
Z^{N'}:=(1/\# gZ^N)\sum_{g\in G\cdot Z}gZ^N\in\mathcal Z^d(X)\otimes\mathbb Q_p.
\end{equation*} 
As pointed by (\cite{Raskind},\cite{Liedtke}) the $p$-adic Tate conjecture, that is for $p$-adic field with $\mathbb Q_p$ coefficients is false in general.
This is due to the fact that for $p\in\Spec(O_k)$ such that $X^{\mathcal O}_{\hat k_{\sigma_p}}$ has good reduction and for 
$Z^{\mathcal O}\subset X^{\mathcal O}_{\hat k_{\sigma_p}}$ a closed subset with each irreducible components surjective over $O_{\hat k_{\sigma_p}}$,  
\begin{equation*}
c^*:H_{DR,Z}^{2d}(X_{\hat k_{\sigma_p}})=\oplus_{i=1}^r\hat k_{\sigma_p}[Z_i]
\to H^{2d}_{DR,\hat Z^{\mathcal O}}(\hat X^{\mathcal O}_{\hat k_{\sigma_p}})\otimes_{O_{\hat k_{\sigma_p}}}\hat k_{\sigma_p}
=\oplus_{i=1}^r\oplus_{j=1}^{s_i}\hat k_{\sigma_p}[(Z^{\mathcal O}_i/p)_j]
\end{equation*}
is NOT surjective in general since $Z^{\mathcal O}/p$ has more irreducible components then $Z$ in general ($s:=\sum_{i=1}^rs_i\geq r$), and 
\begin{equation*}
c^*:H_{DR}^{2d}(X_{\hat k_{\sigma_p}}\backslash Z)\to 
H^{2d}_{DR}(\widehat{(X_{\hat k_{\sigma_p}}\backslash Z})^{\mathcal O})\otimes_{O_{\hat k_{\sigma_p}}}\hat k_{\sigma_p}
\end{equation*}
is NOT injective in general ($X_{\hat k_{\sigma_p}}\backslash Z$ is not proper) : see remark \ref{RasEx}.
On the other side, for a $p$-adic field $K\subset\mathbb C_p$ and $X$ a smooth projective variety over $K$, 
\begin{equation*}
\dim H_{et}^j(X_{\mathbb C_p},\mathbb Q_\ell)(k)^G>\dim H_{et}^j(X_{\mathbb C_p},\mathbb Q_p)(k)^G 
\end{equation*}
for $\ell\neq p$ in general and it is known that the $\ell$-adic Tate conjecture is not true.
Hence there is no relation with the classical Tate conjecture over finite fields.

Let $X$ be a smooth projective variety over $\mathbb C$. 
Then $X$ is defined over a subfield $k\subset\mathbb C$ of finite type over $\mathbb Q$, that is $X=X_k\otimes_k\mathbb C$.
Take an isomorphism $\mathbb C\simeq\mathbb C_p$
with $p\in\mathbb N$ a prime number such that $X_{\mathbb C_p}$ has good reduction at $p$.
As the Tate conjecture holds for $X_k$ with $\mathbb Q_p$ coefficients, the standard conjectures holds for $X$
and any absolute Hodge class of $X$ is the class of an algebraic cycle (corollary \ref{TateCor}).

I am grateful for professor F.Mokrane for help and support during this work.

\section{Preliminaries and Notations}

\subsection{Notations}

\begin{itemize}

\item Denote by $\Top$ the category of topological spaces and $\RTop$ the category of ringed spaces.
\item Denote by $\Cat$ the category of small categories and $\RCat$ the category of ringed topos.
\item For $\mathcal S\in\Cat$ and $X\in\mathcal S$, we denote $\mathcal S/X\in\Cat$ the category whose
objects are $Y/X:=(Y,f)$ with $Y\in\mathcal S$ and $f:Y\to X$ is a morphism in $\mathcal S$, and whose morphisms
$\Hom((Y',f'),(Y,f))$ consists of $g:Y'\to Y$ in $\mathcal S$ such that $f\circ g=f'$.

\item For $(\mathcal S,O_S)\in\RCat$ a ringed topos, we denote by 
\begin{itemize}
\item $\PSh(\mathcal S)$ the category of presheaves of abelian group on $\mathcal S$ and
$\PSh_{O_S}(\mathcal S)$ the category of presheaves of $O_S$ modules on $\mathcal S$, 
whose objects are 
\begin{equation*}
\PSh_{O_S}(\mathcal S)^0:=\left\{(M,m),M\in\PSh(\mathcal S),m:M\otimes O_S\to M\right\},
\end{equation*}
together with the forgetful functor $o:\PSh(\mathcal S)\to \PSh_{O_S}(\mathcal S)$,
for $F\in\PSh(\mathcal S)$ and $X\in\mathcal S$, we denote $F(X):=\Gamma(X,F)$ the abelian group of section over $X$,
\item $C(\mathcal S)=C(\PSh(\mathcal S))$ and $C_{O_S}(\mathcal S)=C(\PSh_{O_S}(\mathcal S))$ 
the big abelian category of complexes of presheaves of $O_S$ modules on $\mathcal S$,
\item $C_{O_S(2)fil}(\mathcal S):=C_{(2)fil}(\PSh_{O_S}(\mathcal S))\subset C(\PSh_{O_S}(\mathcal S),F,W)$,
the big abelian category of (bi)filtered complexes of presheaves of $O_S$ modules on $\mathcal S$ 
such that the filtration is biregular and $\PSh_{O_S(2)fil}(\mathcal S):=(\PSh_{O_S}(\mathcal S),F,W)$.
\end{itemize}

\item Let $(\mathcal S,O_S)\in\RCat$ a ringed topos with topology $\tau$. For $F\in C_{O_S}(\mathcal S)$,
we denote by $k:F\to E_{\tau}(F)$ the canonical flasque resolution in $C_{O_S}(\mathcal S)$ (see \cite{B5}).
In particular for $X\in\mathcal S$, $H^*(X,E_{\tau}(F))\xrightarrow{\sim}\mathbb H_{\tau}^*(X,F)$.

\item For $f:\mathcal S'\to\mathcal S$ a morphism with $\mathcal S,\mathcal S'\in\RCat$,
endowed with topology $\tau$ and $\tau'$ respectively, we denote for $F\in C_{O_S}(\mathcal S)$ and each $j\in\mathbb Z$,
\begin{itemize}
\item $f^*:=H^j\Gamma(\mathcal S,k\circ\ad(f^*,f_*)(F)):\mathbb H^j(\mathcal S,F)\to\mathbb H^j(\mathcal S',f^*F)$,
\item $f^*:=H^j\Gamma(\mathcal S,k\circ\ad(f^{*mod},f_*)(F)):\mathbb H^j(\mathcal S,F)\to\mathbb H^j(\mathcal S',f^{*mod}F)$, 
\end{itemize}
the canonical maps.

\item For $\mathcal X\in\Cat$ a (pre)site and $p$ a prime number, we consider the full subcategory
\begin{equation*}
\PSh_{\mathbb Z_p}(\mathcal X)\subset\PSh(\mathbb N\times\mathcal X), \; \; 
F=(F_n)_{n\in\mathbb N}, \; p^nF_n=0, \; F_{n+1}/p^n\xrightarrow{\sim}F_n
\end{equation*}
$C_{\mathbb Z_p}(\mathcal X):=C(\PSh_{\mathbb Z_p}(\mathcal X))\subset C(\mathbb N\times\mathcal X)$ and
\begin{equation*}
\mathbb Z_p:=\mathbb Z_{p,\mathcal X}:=((\mathbb Z/p^*\mathbb Z)_{\mathcal X})\in\PSh_{\mathbb Z_p}(\mathcal X)
\end{equation*}
the diagram of constant presheaves on $\mathcal X$.

\item For $X\in\Top$ and $Z\subset X$ a closed subset, denoting $j:X\backslash Z\hookrightarrow X$ the open complementary, 
we will consider
\begin{equation*}
\Gamma^{\vee}_Z\mathbb Z_X:=\Cone(\ad(j_!,j^*)(\mathbb Z_X):j_!j^*\mathbb Z_X\hookrightarrow\mathbb Z_X)\in C(X)
\end{equation*}
and denote for short $\gamma^{\vee}_Z:=\gamma^{\vee}_X(\mathbb Z_X):\mathbb Z_X\to\Gamma^{\vee}_Z\mathbb Z_X$
the canonical map in $C(X)$.

\item Denote by $\Sch\subset\RTop$ the subcategory of schemes (the morphisms are the morphisms of locally ringed spaces).
We denote by $\PSch\subset\Sch$ the full subcategory of proper schemes.
For a commutative ring $A$, we consider $\Sch/A:=\Sch/\Spec A$ the category of schemes over $\Spec A$,
that is whose object are $X:=(X,a_X)$ with $X\in\Sch$ and $a_X:X\to\Spec A$ a morphism
and whose objects are morphism of schemes $f:X'\to X$ such that $f\circ a_{X'}=a_X$. 
We denote by $\PSch/A\subset\Sch/A$ the full subcategory of projective schemes over $A$.
We then denote by
\begin{itemize}
\item $\Var(k)=\Sch^{ft}/k\subset\Sch/k$ the full subcategory consisting of algebraic varieties over $k$, 
i.e. schemes of finite type over $k$,
\item $\PVar(k)\subset\QPVar(k)\subset\Var(k)$ 
the full subcategories consisting of quasi-projective varieties and projective varieties respectively, 
\item $\PSmVar(k)\subset\SmVar(k)\subset\Var(k)$,  $\PSmVar(k):=\PVar(k)\cap\SmVar(k)$,
the full subcategories consisting of smooth varieties and smooth projective varieties respectively.
\end{itemize}
For a morphism of commutative rings $\sigma:A\hookrightarrow B$, we have the extention of scalar functor
\begin{eqnarray*}
\otimes_AB:\Sch/A\to\Sch/B, \; X\mapsto X_B:=X_{B,\sigma}:=X\otimes_AB, \; (f:X'\to X)\mapsto (f_B:=f\otimes I:X'_B\to X_B).
\end{eqnarray*}
which is left ajoint to the restriction of scalar 
\begin{eqnarray*}
Res_{A/B}:\Sch/B\to\Sch/A, \; X=(X,a_X)\mapsto X=(X,\sigma\circ a_X), \; (f:X'\to X)\mapsto (f:X'\to X)
\end{eqnarray*}
For a morphism of fields $\sigma:k\hookrightarrow K$, the extention of scalar functor restricts to a functor
\begin{eqnarray*}
\otimes_kK:\Var(k)\to\Var(K), \; X\mapsto X_K:=X_{K,\sigma}:=X\otimes_kK, \; (f:X'\to X)\mapsto (f_K:=f\otimes I:X'_K\to X_K).
\end{eqnarray*}
and for $X\in\Var(k)$ we have $\pi_{k/K}(X):X_K\to X$ the projection in $\Sch/k$.

\item For $X\in\Sch$ a noetherian scheme and $p\in\mathbb N$, we denote by $\mathcal Z^p(X)$ the free abelian group
generated by closed subset of codimension $p$.

\item For $X\in\Sch$ and $p$ a prime number, 
we denote by $c:\hat X:=\hat X^{(p)}\to X$ the morphism in $\RTop$ which is the completion along the ideal generated by $p$.

\item For $K$ a field which is complete with respect to a $p$-adic norm, we denote by $\RigVar(K)\subset\RTop$
the subcategory of rigid analytic space (i.e. locally given by affinoid which are Tate algebra spectrum). 
We denote by $\An_p:\RigVar(K)\to\Var(K)$ the analytic functor and for $X\in\Var(K)$,
$\an_{X,p}:=\An_{p,|X}:\hat X^{(p)}:=\hat X^{\mathcal O,(p)}\times_{O_K}K\to X$ the corresponding morphism in $\RTop$.

\item For $k$ a field, we denote by $\AbVar(k)$ the category of abelian varieties over $k$, 
i.e. an algebraic variety over $k$ endowed with a structure of (abelian) group which is a morphism of algebraic varieties.

\item For $X\in\Sch$ a scheme, we denote by $\Pic(X)$ its Picard group of line bundle (or equivalently rational classes of Cartier divisors)
and $\Br(X):=H^2_{et}(X,O^*_X)$ its Brauer group. For $M$ an abelian group and $n\in\mathbb N$, 
we denote $M[n]\subset M$ the subgroup whose elements are of $n$ torsion.

\item Denote $\Sch^2\subset\RTop^2$ the subcategory
whose objects are couples $(X,Z)$ with $X=(X,O_X)\in\Sch$ and $Z\subset X$ a closed subset
and whose set of morphisms $\Hom((X',Z'),(X,Z))$ consists of $f:X'\to X$ of locally ringed spaces
such that $f^{-1}(Z)\subset Z'$.

\item Let $k$ be a field of characteristic zero. 
Denote $\SmVar^2(k)\subset\Var^2(k)\subset\Sch^2/k$ the full subcategories
whose objects are $(X,Z)$ with $X\in\Var(k)$, resp. $X\in\SmVar(k)$, and $Z\subset X$ is a closed subset,
and whose morphisms $\Hom((X',Z')\to(X,Z))$ consists of $f:X'\to X$ of schemes over $k$ 
such that $f^{-1}(Z)\subset Z'$.

\item Denote by $\AnSp(\mathbb C)\subset\RTop$ the full subcategory of analytic spaces over $\mathbb C$,
and by $\AnSm(\mathbb C)\subset\AnSp(\mathbb C)$ the full subcategory of smooth analytic spaces (i.e. complex analytic manifold).
Denote by $\CW\subset\Top$ the full subcategory of $CW$ complexes.
Denote by $\Diff(\mathbb R)\subset\RTop$ the full subcategory of differentiable (real) manifold.  

\item For $K$ a field which is complete with respect to a $p$-adic norm, we denote $O_K\subset K$ the ring the integers, i.e elements with norm lower or equal to $1$.

\end{itemize}

\subsection{The pro-etale site of schemes}

For $X\in\Sch$, we denote $X^{et}\subset\Sch/X$ the etale site and
$X^{pet}:=\Pro(X^{et})$ the pro etale site (see \cite{BSch}) and $X^{pet}\subset\Sch/X$ its image by taking the projective limit in the category of schemes
which is the full subcategory of $\Sch/X$ whose object consists of weakly etale maps $U\to X$ (that is flat maps
$U\to X$ such that $\Delta_U:U\to U\times_XU$ is also flat) and whose topology is generated by fpqc covers.
We then have the canonical morphism of site
\begin{equation*}
\nu_X:X^{pet}\to X^{et}, (U\to X)\mapsto (U\to X)
\end{equation*}  
For $F\in C(X^{et})$, 
\begin{equation*}
\ad(\nu_X^*,R\nu_{X*})(F):F\to R\nu_{X*}\nu_X^*F 
\end{equation*}
is an isomorphism in $D(X^{et})$, in particular, for each $n\in\mathbb Z$
\begin{equation*}
\nu_X^*:\mathbb H^n_{et}(X,F)\xrightarrow{\sim}\mathbb H^n_{pet}(X,\nu_X^*F) 
\end{equation*}
are isomorphisms,
For $X\in\Sch$, we denote
\begin{itemize}
\item $\underline{\mathbb Z_p}_X:=\varprojlim_{n\in\mathbb N}\nu_X^*(\mathbb Z/p^n\mathbb Z)_{X^{et}}\in\PSh(X^{pet})$ 
the constant presheaf on $\mathcal X$,
\item $l_{p,\mathcal X}:=(p(*)):\underline{\mathbb Z_p}_{X}\to\nu_X^*(\mathbb Z/p\mathbb Z)_{X^{et}}$ 
the projection map in $\PSh(\mathbb N\times\mathcal X^{pet})$.
\end{itemize}
An affine scheme $U\in\Sch$ is said to be $w$-contractible if any faithfully flat weakly etale map $V\to U$, 
$V\in\Sch$, admits a section. We will use the facts that (see \cite{BSch}):
\begin{itemize}
\item Any affine scheme $X\in\Sch$ admits a faithfully flat pro-etale map $r:U\to X$ with $U$ w-contractile.
\item Any scheme $X\in\Sch$ admits a pro-etale affine cover $(r_i:X_i\to X)_{i\in I}$ 
with for each $i\in I$, $X_i$ a $w$-contractile affine scheme and $r_i:X_i\to X$ a weakly etale map.
For $X\in\Var(k)$ with $k$ a field, we may assume $I$ finite since the topological space $X$ is then quasi-compact.
\item If $U\in\Sch$ is a $w$-contractible affine scheme, then for any sheaf $F\in\Shv(U^{pet})$, 
$H_{pet}^i(U,F)=0$ for $i\neq 0$ since $\Gamma(U,-)$ is an exact functor.
\end{itemize} 

\subsection{Integral models and $p$-adic completion of algebraic varieties over a $p$-adic field}

For $K$ a field which is complete with respect to a $p$-adic norm,
we consider $O_K\subset K$ the subring of $K$ consisting of integral elements, that is $x\in K$ such that $|x|\leq 1$.
\begin{itemize}
\item For $X\in\PVar(K)$ irreducible, we will consider $X^{\mathcal O}\in\PSch/O_K$ irreducible a (non canonical) integral model of $X$, 
i.e. $X^{\mathcal O}\otimes_{O_K}K=X$ and the structural morphism $a_{X^{\mathcal O}}:X^{\mathcal O}\to\Spec(O_K)$ is surjective.
\item For $X\in\PVar(K)$, let $X=\cup_{i=1}^rX_i$ where $X_i$ are the irreducible components of $X$,
we will consider $X^{\mathcal O}:=\cup_{i=1}^rX_i^{\mathcal O}\in\PSch/O_K$ a (non canonical) integral model of $X$, with for each $1\leq i\leq r$,
$X_i^{\mathcal O}\in\PSch/O_K$ is a (non canonical) integral model of $X$.
\item For $X\in\Var(K)$, we will consider $X^{\mathcal O}\in\Sch/O_K$ a (non canonical) integral model of $X$, 
i.e. $X^{\mathcal O}=\bar X^{\mathcal O}\backslash Z^{\mathcal O}$ for 
$\bar X\in\PVar(K)$ a compactification of $X$, $Z:=\bar X\backslash X$,
where $\bar X^{\mathcal O}\in\PSch/O_K$ is an integral model of $\bar X$ and 
$Z^{\mathcal O}:=V(I_Z^{\mathcal O})\subset\bar X^{\mathcal O}$ is an integral model of $Z$. 
\end{itemize}
We consider $\Sch^{int}/O_K:=O(\PSch^2/O_K)\subset\Sch/O_K$ the full subcategory
consisting of integral models of algebraic varieties over $K$,
where $O:\PSch^2/O_K\to\Sch^{ft}/O_K, \; O(X,Z)=X\backslash Z$ is the canonical functor, and 
\begin{equation*}
\Sch^{int,sm}/O_K:=\Sch^{int}/O_K\cap\Sch^{sm}/O_K\subset\Sch/O_K
\end{equation*}
the full subcategory consisting of smooth integral models i.e. the integral models of (smooth) algebraic varieties over $K$
which are smooth over $O_K$. 
For $X\in\Var(K)$, we will consider $X^{\mathcal O}\in\Sch^{int}/O_K$ a (non canonical) integral model of $X$,
we then conisder $X^{\mathcal O,et}:=((\Sch^{int}/O_K)/X^{\mathcal O})\cap(X^{\mathcal O})^{et}$ 
the full subcategory $(\Sch^{int}/O_K)/X^{\mathcal O}$ consisting of $e:U=U^{\mathcal O}\to X^{\mathcal O}$ such that $e$ is an etale morphism of schemes, 
and $X^{\mathcal O,pet}:=\Pro(X^{\mathcal O,et})$ and $X^{\mathcal O,pet}\subset(\Sch^{int}/O_K)/X^{\mathcal O}$ 
its image by taking the projective limit in the category of schemes.
We have then the commutative diagram of sites
\begin{equation*}
\xymatrix{X^{pet}\ar[r]^{r}\ar[d]^{\nu_X} & X^{\mathcal O,pet}\ar[d]^{\nu_{X^{\mathcal O}}} \\
X^{et}\ar[r]^{r} & X^{\mathcal O,et}}, \; \; 
r(t:U=U^{\mathcal O}\to X^{\mathcal O}):=(t\otimes_{O_K}K:U\otimes_{O_K}K\to X^{\mathcal O}\otimes_{O_K}K=X)
\end{equation*}
Note that the inclusion $X^{\mathcal O,et}\subset(X^{\mathcal O})^{et}$ is strict by definition, where we recall
$(X^{\mathcal O})^{et}\subset(\Sch/O_K)/X^{\mathcal O}$ is the full subcategory consisting of etale structural morphisms.
We denote $i:X^{\mathcal O}/p\hookrightarrow X^{\mathcal O}$ the closed embedding of the special fiber and 
$i:(X^{\mathcal O}/p)^{et}\hookrightarrow X^{\mathcal O,et}$,
$i(h:W^{\mathcal O}\to X^{\mathcal O}):=(h\times_{X^{\mathcal O}}X^{\mathcal O}/p:W^{\mathcal O}_{/p}\to X^{\mathcal O}_{/p})$, the associated morphism of site.

We will use the following proposition

\begin{prop}\label{oetetprop}
Let $K$ be a field which is complete with respect to a $p$-adic norm. Let $X\in\Var(K)$ and $X^{\mathcal O}\in\Sch^{int}/O_K$ an integral model of $X$.
Consider the morphism of site $r:X^{et}\to X^{\mathcal O,et}$. For $F=i_*F'\in\PSh(X^{\mathcal O,et})$ with $F'\in\PSh((X^{\mathcal O}/p)^{et})$, 
we have $R^qr_*r^*F=0$ for $q\in\mathbb Z$, $q\neq 0$, that is the morphism in $D(X^{\mathcal O,et})$ 
\begin{equation*}
\ad(r^*,Rr_*):F\to Rr_*r^*F
\end{equation*}
is an isomorphism.
\end{prop}

\begin{proof}
For $S\in\Sch$ a scheme, $S^{pet}\subset\Sch/S$ denote the pro etale site consisting of morphism $h:V\to S$ with $h$ flat and locally of finite presentation, 
and $\epsilon_S:S^{pet}\to S^{et}$ the morphism of site given by the inclusion $S^{et}\subset S^{pet}$.
For $G\in\PSh(S^{et})$, we have $R^q\epsilon_{S*}\epsilon_S^*G=0$ for $q\in\mathbb Z$, $q\neq 0$ (see \cite{BSch}).
We have also for $T^{\mathcal O}\in\Sch^{int}/O_K$, $T^{\mathcal O,pet}\subset(\Sch^{int}/O_K)/T^{\mathcal O}$ 
the full subcategory consisting of morphism $h:W^{\mathcal O}\to T^{\mathcal O}$ with $h$ pro etale and locally of finite presentation,
and $\epsilon_T^{\mathcal O}:T^{\mathcal O,pet}\to T^{\mathcal O,et}$ the morphism of site given by the inclusion $T^{\mathcal O,et}\subset T^{\mathcal O,pet}$.
Consider then the commutative diagram of sites,
\begin{equation*}
\xymatrix{X^{pet}\ar[d]^{\epsilon:=\epsilon_X}\ar[rr]^{r} & \, & X^{\mathcal O,pet}\ar[d]^{\epsilon^{\mathcal O}:=\epsilon_X^{\mathcal O}} 
& \, & (X^{\mathcal O}/p)^{pet}\ar[ll]^{i}\ar[d]^{\epsilon^{\mathcal O}_{/p}:=\epsilon_{X^{\mathcal O}_{/p}}} \\
X^{et}\ar[rr]^{r} & \, & X^{\mathcal O,et} & \, & (X^{\mathcal O}/p)^{et}\ar[ll]^{i} }
\end{equation*}
where $i:(X^{\mathcal O}/p)^{pet}\hookrightarrow X^{\mathcal O,pet}$, 
$i(h:W^{\mathcal O}\to X^{\mathcal O}):=(h\times_{X^{\mathcal O}}X^{\mathcal O}/p:W^{\mathcal O}_{/p}\to X^{\mathcal O}_{/p})$, is the associated morphism of site.
We also have, for  $F=i_*F'\in\PSh(X^{\mathcal O,et})$ with $F'\in\PSh((X^{\mathcal O}/p)^{et})$, $q\in\mathbb Z$, $q\neq 0$, 
\begin{equation*}
R^q\epsilon^{\mathcal O}_*\epsilon^{\mathcal O*}F=R^q\epsilon^{\mathcal O}_*\epsilon^{\mathcal O*}i_*F'=R^q\epsilon^{\mathcal O}_*i_*\epsilon^{\mathcal O*}_{/p}F'
=i_*R^q\epsilon^{\mathcal O}_{/p*}\epsilon^{\mathcal O*}_{/p}F'=0
\end{equation*}
Since on the other hand $r:X^{pet}\to X^{\mathcal O,pet}$ is an equivalence of category of corresponding topos (for $h:Y\to X$ a flat morphism with $Y$ w-contractile, 
there exist an integral model $Y^{\mathcal O}$ such that $h$ extend to a flat morphism $\bar h:Y^{\mathcal O}\to X^{\mathcal O}$, 
indeed we can choose an arbirary integral model $Y^{\mathcal O,0}$ of $Y$ an take 
$Y^{\mathcal O}:=\bar\Gamma_h\subset Y^{\mathcal O,0}\times X^{\mathcal O}$ the closure of the graph of $h$), 
we get the proposition using the left square of the commutative diagram.
\end{proof}

Let $K$ be a field which is complete with respect to a $p$-adic norm and $X\in\PVar(K)$ projective.
For $X^{\mathcal O}\in\PSch/O_K$ an integral model of $X$, i.e. satisfying $X^{\mathcal O}\otimes_{O_K}K=X$.
Recall (see section 2.1) that $c:\hat X^{\mathcal O}\to X^{\mathcal O}$ 
is the morphism in $\RTop$ which is the completion along the ideal generated by $p$.
Then, by GAGA (c.f. EGA 3), for $F\in\Coh_{O_X}(X)$ a coherent sheaf of $O_X$ module, 
for all $k\in\mathbb Z$, the canonical map
\begin{equation*}
c^*:H_{pet}^k(X^{\mathcal O},F)=H_{zar}^k(X^{\mathcal O},F)\xrightarrow{\sim}H_{pet}^k(\hat X^{\mathcal O},c^{*mod}F)=H_{zar}^k(\hat X^{\mathcal O},c^{*mod}F) 
\end{equation*}
is an isomorphism.

\subsection{De Rham cohomology}

We recall some properties of the De Rham cohomology.

\begin{itemize}
\item We have
\begin{eqnarray*}
\Omega^{\bullet}\in C(\Sch), X\mapsto\Omega^{\bullet}(X):=\Gamma(X,\Omega^{\bullet}_X), \\
(f:X'\to X)\mapsto\Omega^{\bullet}(f):=f^*:\Gamma(X,\Omega^{\bullet}_X)\to\Gamma(X',\Omega^{\bullet}_{X'})
\end{eqnarray*}
Let $X\in\Sch$. Considering its De Rham complex $\Omega_X^{\bullet}:=DR(X)(O_X)$,
we have for $j\in\mathbb Z$ its De Rham cohomology $H^j_{DR}(X):=\mathbb H^j(X,\Omega^{\bullet}_X)$.
\item We will consider
\begin{eqnarray*}
\Omega^{\bullet}_{/k}\in C(\Var(k)), X\mapsto\Omega^{\bullet}_{/k}(X):=\Gamma(X,\Omega^{\bullet}_X), \\
(f:X'\to X)\mapsto\Omega^{\bullet}_{/k}(f):=f^*:\Gamma(X,\Omega^{\bullet}_X)\to\Gamma(X',\Omega^{\bullet}_{X'})
\end{eqnarray*}
and its restriction to $\SmVar(k)\subset\Var(k)$.
Let $X\in\Var(k)$. Considering its De Rham complex $\Omega_X^{\bullet}:=\Omega_{X/k}^{\bullet}:=DR(X/k)(O_X)$,
we have for $j\in\mathbb Z$ its De Rham cohomology $H^j_{DR}(X):=\mathbb H^j(X,\Omega^{\bullet}_X)$.
The differentials of $\Omega_X^{\bullet}:=\Omega_{X/k}^{\bullet}$ are by definition $k$-linear,
thus $H^j_{DR}(X):=\mathbb H^j(X,\Omega^{\bullet}_X)$ has a structure of a $k$ vector space.
\end{itemize}

If $X\in\SmVar(k)$, then $H^j_{DR}(X)=\mathbb H_{et}^j(X,\Omega^{\bullet}_X)$
since $\Omega^{\bullet}_{/k}\in C(\SmVar(k))$ is $\mathbb A^1$ local and admits transfers where
(see \cite{B5}).
Note that for $X\in\Var(k)$ singular, we also consider its realized De Rham cohomology 
\begin{equation*}
\tilde H^j_{DR}(X):=\mathbb H^j(X_{\bullet},\Omega^{\bullet}_{X_{\bullet}})=
\mathbb H^j(\tilde X,R\Gamma_X\Omega^{\bullet}_{\tilde X})
\end{equation*}
where $\epsilon:X_{\bullet}\to X$ in $\Fun(\Delta,\SmVar(k))$ is a simplicial desingularization of $X$
and $X\hookrightarrow\tilde X$ is a closed embedding with $\tilde X\in\SmVar(k)$,
and note that $H^j_{DR}(X)$ is NOT isomorphic to $\tilde H^j_{DR}(X)$ in general 
since $\Omega^{\bullet}_{/k}\in C(\Var(k))$ does NOT satisfied cdh descent. 

Let $X\in\Var(k)$. Let $X=\cup_{i=1}^sX_i$ an open affine cover. 
For $I\subset[1,\ldots,s]$, we denote $X_I:=\cap_{i\in I}X_i$. 
We get $X_{\bullet}\in\Fun(P([1,\ldots,s]),\Var(k))$.
Since quasi-coherent sheaves on affine noetherian schemes are acyclic, we have
for each $j\in\mathbb Z$, $H^j_{DR}(X)=\Gamma(X_{\bullet},\Omega^{\bullet}_{X^{\bullet}})$.

\subsection{Singular chains}

We denote $\mathbb I^n:=[0,1]^n\in\Diff(\mathbb R)$ (with boundary).
For $X\in\Top$ and $R$ a ring, we consider its singular cochain complex
\begin{equation*}
C^*_{\sing}(X,R):=(\mathbb Z\Hom_{\Top}(\mathbb I^*,X)^{\vee})\otimes R 
\end{equation*}
and for $l\in\mathbb Z$ its singular cohomology $H^l_{\sing}(X,R):=H^nC^*_{\sing}(X,R)$.
For $f:X'\to X$ a continous map with $X,X'\in\Top$, we have the canonical map of complexes
\begin{equation*}
f^*:C^*_{\sing}(X,R)\to C^*_{\sing}(X,R), \sigma\mapsto f^*\sigma:=(\gamma\mapsto\sigma(f\circ\gamma)).
\end{equation*}
In particular, we get by functoriality the complex 
\begin{equation*}
C^*_{X,R\sing}\in C_R(X), \; (U\subset X)\mapsto C^*_{\sing}(U,R)
\end{equation*}
We will consider the canonical embedding 
\begin{equation*}
C^*\iota_{2i\pi\mathbb Z/\mathbb C}(X):C^*_{\sing}(X,2i\pi\mathbb Z)\hookrightarrow C^*_{\sing}(X,\mathbb C), \,
\alpha\mapsto\alpha\otimes 1
\end{equation*}
whose image consists of cochains $\alpha\in C^j_{\sing}(X,\mathbb C)$ such that $\alpha(\gamma)\in 2i\pi\mathbb Z$
for all $\gamma\in\mathbb Z\Hom_{\Top}(\mathbb I^*,X)$.
We get by functoriality the embedding in $C(X)$
\begin{eqnarray*}
C^*\iota_{2i\pi\mathbb Z/\mathbb C,X}:C^*_{X,2i\pi\mathbb Z,\sing}\hookrightarrow C^*_{X,\mathbb C,\sing}, \\
(U\subset X)\mapsto 
(C^*\iota_{2i\pi\mathbb Z/\mathbb C}(U):C^*_{\sing}(U,2i\pi\mathbb Z)\hookrightarrow C^*_{\sing}(U,\mathbb C))
\end{eqnarray*}
We recall we have 
\begin{itemize}
\item For $X\in\Top$ locally contractile, e.g. $X\in\CW$, and $R$ a ring, the inclusion in $C_R(X)$
$c_X:R_X\to C^*_{X,R\sing}$ is by definition an equivalence top local and that we get 
by the small chain theorem, for all $l\in\mathbb Z$, an isomorphism 
$H^lc_X:H^l(X,R_X)\xrightarrow{\sim}H^l_{\sing}(X,R)$.
\item For $X\in\Diff(\mathbb R)$, the restriction map 
\begin{equation*}
r_X:\mathbb Z\Hom_{\Diff(\mathbb R)}(\mathbb I^*,X)^{\vee}\to C^*_{\sing}(X,R), \; 
w\mapsto w:(\phi\mapsto w(\phi))
\end{equation*}
is a quasi-isomorphism by Whitney approximation theorem.
\end{itemize}

\subsection{Algebraic cycles, motives and the theorem of Voevodsky}

For $X\in\Sch$ noetherian irreducible and $d\in\mathbb N$, 
we denote by $\mathcal Z^d(X)$ the group of algebraic cycles of codimension $d$,
which is the free abelian group generated by irreducible closed subsets of codimension $d$.

For $X,X'\in\Sch$ noetherian, with $X'$ irreducible, we denote 
$\mathcal Z^{fs/X'}(X'\times X)\subset\mathcal Z_{d_{X'}}(X'\times X)$
which consist of algebraic cycles $\alpha=\sum_in_i\alpha_i\in\mathcal Z_{d_{X'}}(X'\times X)$ such that,
denoting $\supp(\alpha)=\cup_i\alpha_i\subset X'\times X$ its support and $p':X'\times X\to X'$ the projection,
$p'_{|\supp(\alpha)}:\supp(\alpha)\to X'$ is finite surjective.

\begin{itemize}
\item Let $k$ be a field.
We denote by $\Cor\SmVar(k)$ the category such that the objects are $\left\{X\in\SmVar(k)\right\}$
and such that $\Hom_{\Cor\SmVar(k)}(X',X):=\mathcal Z^{fs/X'}(X'\times X)$.
See \cite{CD} for the composition law.
We denote by $\Tr:\Cor\SmVar(k)\to\SmVar(k)$ the morphism of site given by the embedding 
$\Tr:\SmVar(k)\hookrightarrow\Cor\SmVar(k)$. Let $F\in\PSh(\SmVar(k))$. We say that $F$ admits
transfers if $F=\Tr_*F$ with $F\in\PSh(\Cor\SmVar(k))$.

\item Let $A$ a regular commutative noetherian ring.
We denote by $\Cor\Sch^{sm}/A$ the category such that the objects are $\left\{X\in\Sch^{sm}/A\right\}$
and such that $\Hom_{\Cor\Sch^{sm}/A}(X',X):=\mathcal Z^{fs/X'}(X'\times X)$.
See \cite{CD} for the composition law.
We denote by $\Tr:\Cor\Sch^{sm}/A\to\Sch^{sm}/A$ the morphism of site given by the embedding 
$\Tr:\Sch^{sm}/A\hookrightarrow\Cor\Sch^{sm}/A$. Let $F\in\PSh(\Sch^{sm}/A)$. We say that $F$ admits
transfers if $F=\Tr_*F$ with $F\in\PSh(\Cor\Sch^{sm}/A)$.
\end{itemize}

We recall the following standard notion (see e.g. \cite{CD} or \cite{B5}) :

\begin{defi}
\begin{itemize}
\item[(i)]Let $F\in\PSh(\Var(k))$ or $F\in\PSh(\SmVar(k))$. We say that $F$ is $\mathbb A^1$ invariant
if for all $X\in\Var(k)$ (resp. $X\in\SmVar(k)$), $p^*:=F(p):F(X)\to F(X\times\mathbb A^1)$ is an isomorphism
where $p:X\times\mathbb A^1\to X$ is the projection.
\item[(ii)]Let $F\in\PSh(\Var(k))$ or $F\in\PSh(\SmVar(k))$. We say that $F$ is $\mathbb A^1$ local
if for all $j\in\mathbb Z$ and all $X\in\Var(k)$ (resp. $X\in\SmVar(k)$), 
$p^*:=H^jE_{et}(F)(p):H_{et}^j(X,F)\to H^j_{et}(X\times\mathbb A^1,F)$ is an isomorphism.
\item[(ii)']Let $F\in C(\Var(k))$ or $F\in C(\SmVar(k))$. We say that $F$ is $\mathbb A^1$ local
if for all $j\in\mathbb Z$ and all $X\in\Var(k)$ (resp. $X\in\SmVar(k)$), 
$p^*:=H^jE_{et}(F)(p):\mathbb H_{et}^j(X,F)\to\mathbb H^j_{et}(X\times\mathbb A^1,F)$ is an isomorphism.
Note that (ii) is a particular case of (ii)'.
\end{itemize}
\end{defi}

For $X\in\Var(k)$ and $Z\subset X$ a closed subset, denoting $j:X\backslash Z\hookrightarrow X$ the open complementary, 
we will consider
\begin{equation*}
\Gamma^{\vee}_Z\mathbb Z_X:=\Cone(\ad(j_{\sharp},j^*)(\mathbb Z_X):
\mathbb Z_X\hookrightarrow\mathbb Z_X)\in C(\Var(k)^{sm}/X)
\end{equation*}
and denote for short $\gamma^{\vee}_Z:=\gamma^{\vee}_Z(\mathbb Z_X):\mathbb Z_X\to\Gamma^{\vee}_Z\mathbb Z_X$
the canonical map in $C(\Var(k)^{sm}/X)$. Denote $a_X:X\to\Spec k$ the structural map.
For $X\in\Var(k)$ and $Z\subset X$ a closed subset, we have the motive of $X$ with support in $Z$ defined as 
\begin{equation*}
M_Z(X):=a_{X!}\Gamma^{\vee}_Za_X^!\mathbb Z\in\DA(k).
\end{equation*}
If $X\in\SmVar(k)$, we will also consider
\begin{equation*}
a_{X\sharp}\Gamma^{\vee}_Z\mathbb Z_X:=\Cone(a_{X\sharp}\circ\ad(j_{\sharp},j^*)(\mathbb Z_X):
\mathbb Z(U)\hookrightarrow\mathbb Z(X))=:\mathbb Z(X,X\backslash Z)\in C(\SmVar(k)).
\end{equation*}
Then for $X\in\SmVar(k)$ and $Z\subset X$ a closed subset
\begin{equation*}
M_Z(X):=a_{X!}\Gamma^{\vee}_Za_X^!\mathbb Z
=a_{X\sharp}\Gamma^{\vee}_Z\mathbb Z_X=:\mathbb Z(X,X\backslash Z)\in\DA(k).
\end{equation*}

\begin{itemize}
\item Let $(X,Z)\in\Sch^2$ with $X\in\Sch$ a noetherian scheme and $Z\subset X$ a closed subset.
We have the deformation $(D_ZX,\mathbb A^1_Z)\to\mathbb A^1$, $(D_ZX,\mathbb A^1_Z)\in\Sch^2$ 
of $(X,Z)$ by the normal cone $C_{Z/X}\to Z$, i.e. such that 
\begin{equation*}
(D_ZX,\mathbb A^1_Z)_s=(X,Z), \, s\in\mathbb A^1\backslash 0, \;  (D_ZX,\mathbb A^1_Z)_0=(C_{Z/X},Z).
\end{equation*}
We denote by $i_1:(X,Z)\hookrightarrow (D_ZX,\mathbb A^1_Z)$ and 
$i_0:(C_{Z/X},Z)\hookrightarrow (D_ZX,\mathbb A^1_Z)$ the closed embeddings in $\Sch^2$.
\item Let $k$ be a field of characteristic zero. Let $X\in\SmVar(k)$. 
For $Z\subset X$ a closed subset of pure codimension $c$,
consider a desingularisation $\epsilon:\tilde Z\to Z$ of $Z$ and denote $n:\tilde Z\xrightarrow{\epsilon}Z\subset X$.
We have then the morphism in $\DA(k)$
\begin{equation*}
G_{Z,X}:M(X)\xrightarrow{D(\mathbb Z(n))}M(\tilde Z)(c)[2c]\xrightarrow{\mathbb Z(\epsilon)}M(Z)(c)[2c]
\end{equation*}
where $D:\Hom_{\DA(k)}(M_c(\tilde Z),M_c(X))\xrightarrow{\sim}\Hom_{\DA(k)}(M(X),M(\tilde Z)(c)[2c])$
is the duality isomorphism from the six functors formalism (moving lemma of Suzlin and Voevodsky)
and $\mathbb Z(n):=\ad(n_!,n^!)(a_X^!\mathbb Z)$, noting that $n_!=n_*$ since $n$ is proper and that
$a_X^!=a_X^*[d_X]$ and $a_{\tilde Z}^!=a_{\tilde Z}^*[d_Z]$ since $X$, resp. $\tilde Z$, are smooth
(considering the connected components, we may assume $X$ and $\tilde Z$ of pure dimension).
\end{itemize}

We recall the following facts (see \cite{CD} and \cite{B5}):

\begin{prop}\label{smXZ}
Let $k$ be a field a characteristic zero.
Let $X\in\SmVar(k)$ and $i:Z\subset X$ a smooth subvariety of pure codimension d.
Then $C_{Z/X}=N_{Z/X}\to Z$ is a vector bundle of rank $d$.
The closed embeddings $i_1:(X,Z)\hookrightarrow (D_ZX,\mathbb A^1_Z)$ and 
$i_0:(C_{Z/X},Z)\hookrightarrow (D_ZX,\mathbb A^1_Z)$ in $\SmVar^2(k)$ induces isomorphisms of motives
$\mathbb Z(i_1):M_Z(X)\xrightarrow{\sim}M_{\mathbb A^1_Z}(D_ZX)$ and 
$\mathbb Z(i_0):M_Z(N_{Z/X})\xrightarrow{\sim}M_{\mathbb A^1_Z}(D_ZX)$ in $\DA(k)$.
We get the excision isomorphism in $\DA(k)$
\begin{equation*}
P_{Z,X}:=\mathbb Z(i_0)^{-1}\circ\mathbb Z(i_1):M_Z(X)\xrightarrow{\sim}M_Z(N_{Z/X}).
\end{equation*}
We have 
\begin{equation*}
Th(N_{Z/X})\circ P_{Z,X}\circ\gamma^{\vee}_Z(\mathbb Z_X)=G_{Z,X}:=D(\mathbb Z(i)):M(X)\to M(Z)(d)[2d].
\end{equation*}
\end{prop}

\begin{proof}
See \cite{CD}.
\end{proof}

We will use the following theorem of Voevodsky :

\begin{thm}\label{Voethm}
(Voevodsky)Let $k$ be a perfect field (e.g. $k$ a field of characteristic zero). 
Let $F\in\PSh(\SmVar(k),\mathbb Q)$. If $F$ is $\mathbb A^1$ invariant and admits transfers, 
then for all $j\in\mathbb Z$, $H^jE_{et}(F)\in\PSh(\SmVar(k))$ are $\mathbb A^1$ invariant.
That is, if $F$ is $\mathbb A^1$ invariant and admits transfers then $F$ is $\mathbb A^1$ local.
\end{thm}

\begin{proof}
By \cite{VoeMW}, $H^jE_{Nis}(F)\in\PSh(\SmVar(k))$ are $\mathbb A^1$ invariant.
On the other hand since $F$ takes values in $\mathbb Q$-vector spaces, 
$H^jE_{Nis}(F)=H^jE_{et}(F)$.
\end{proof}

\subsection{The logaritmic De Rham complexes}

We introduce the logarithmic De Rham complexes

\begin{defi}\label{wlogdef}
\begin{itemize}
\item[(i)] Let $X=(X,O_X)\in\RCat$ a ringed topos, we have in $C(X)$ the subcomplex of presheaves of abelian groups
\begin{eqnarray*}
OL_X:\Omega^{\bullet}_{X,\log}\hookrightarrow\Omega_X^{\bullet}, \; 
\mbox{s.t. for} \; X^o\in X \; \mbox{and}\;  p\in\mathbb N, \, p\geq 1, \\
\Omega^p_{X,\log}(X^o):=
<df_{\alpha_1}/f_{\alpha_1}\wedge\cdots\wedge df_{\alpha_p}/f_{\alpha_p}, f_{\alpha_k}\in O_X^*(X^o)>
\subset\Omega^p_X(X^o),
\end{eqnarray*}
where $\Omega_X^{\bullet}:=DR(X)(O_X)\in C(X)$ is the De Rham complex and 
$O_X^*(X^o)\subset O_X(X^o)$ is the multiplicative group 
consisting of invertible elements for the multiplication,
here $<,>$ stand for the sub-abelian group generated by. By definition, for $w\in\Omega^p_X(X^o)$, 
$w\in\Omega^p_{X,\log}(X^o)$ if and only if there exists $(n_i)_{1\leq i\leq s}\in\mathbb Z$ and 
$(f_{i,\alpha_k})_{1\leq i\leq s,1\leq k\leq p}\in O_X^*(X^o)$ such that
\begin{equation*}
w=\sum_{1\leq i\leq s}n_idf_{i,\alpha_1}/f_{i,\alpha_1}\wedge\cdots\wedge df_{i,\alpha_p}/f_{i,\alpha_p}
\in\Omega^p_X(X^o).
\end{equation*}
For $p=0$, we set $\Omega^0_{X,\log}:=\mathbb Z$ if $\mathbb Z\subset O_X$
and $\Omega^0_{X,\log}:=\mathbb Z/n$ if $O_X$ is a ring of characteristic $n$.
Let $f:X'=(X',O_{X'})\to X=(X,O_X)$ a morphism with $X,X'\in\RCat$.
Consider the morphism $\Omega(f):\Omega_X^{\bullet}\to f_*\Omega_{X'}^{\bullet}$ in $C(X)$.
Then, $\Omega(f)(\Omega^{\bullet}_{X,\log})\subset f_*\Omega^{\bullet}_{X',\log}$.
\item[(ii)] For $k$ a field, we get from (i), for $X\in\Var(k)$, the embedding in $C(X)$
\begin{equation*}
OL_X:\Omega^{\bullet}_{X,\log}\hookrightarrow\Omega_X^{\bullet}:=\Omega^{\bullet}_{X/k},
\end{equation*}
such that, for $X^o\subset X$ an open subset and $w\in\Omega^p_X(X^o)$,
$w\in\Omega^p_{X,\log}(X^o)$ if and only if there exists $(n_i)_{1\leq i\leq s}\in\mathbb Z$ and 
$(f_{i,\alpha_k})_{1\leq i\leq s,1\leq k\leq p}\in O_X^*(X^o)$ such that
\begin{equation*}
w=\sum_{1\leq i\leq s}n_idf_{i,\alpha_1}/f_{i,\alpha_1}\wedge\cdots\wedge df_{i,\alpha_p}/f_{i,\alpha_p}
\in\Omega^p_X(X^o),
\end{equation*}
and for $p=0$, $\Omega^0_{X,\log}:=\mathbb Z$ if $k$ is of characteristic zero, $\Omega^0_{X,\log}:=\mathbb Z/p$
if $k$ is of characteristic $p$. 
We get an embedding in $C(\Var(k))$
\begin{eqnarray*}
OL:\Omega^{\bullet}_{/k,\log}\hookrightarrow\Omega^{\bullet}_{/k}, \; 
\mbox{given by}, \, \mbox{for} \, X\in\Var(k), \\
OL(X):=OL_X:\Omega^{\bullet}_{/k,\log}(X):=\Gamma(X,\Omega^{\bullet}_{X,\log})
\hookrightarrow\Gamma(X,\Omega^{\bullet}_X)=:\Omega^{\bullet}_{/k}(X)
\end{eqnarray*}
and its restriction to $\SmVar(k)\subset\Var(k)$.
\item[(iii)]Let $K$ be a field of characteristic zero which is complete for a $p$-adic norm. 
Recall that $O_K\subset K$ denotes its ring of integers. Let $X\in\Var(K)$. 
Let $X^{\mathcal O}\in\Sch^{int}/O_K$ be an integral model of $X$, in particular $X^{\mathcal O}\otimes_{O_K}K=X$. 
Consider the full subcategories $X^{\mathcal O,et}\subset(\Sch^{int}/O_K)/X^{\mathcal O}$ and $X^{\mathcal O,pet}\subset(\Sch^{int}/O_K)/X^{\mathcal O}$.
We have then the morphisms of sites $r:X^{et}\to X^{\mathcal O,et}$
and $r:X^{pet}\to X^{\mathcal O,pet}$ such that $\nu_{X^{\mathcal O}}\circ r=r\circ\nu_X$.
We will consider mainly the embedding of $C(X^{\mathcal O,pet})$
\begin{eqnarray*}
OL_{\hat X^{\mathcal O,(p)}}:=(OL_{X^{\mathcal O}/p^n})_{n\in\mathbb N}:
\Omega_{\hat X^{(p)},\log,\mathcal O}^{\bullet}:=
\varprojlim_{n\in\mathbb N}\nu_{X^{\mathcal O}}^*a_{et}\Omega^{\bullet}_{X^{\mathcal O,et}/p^n,\log} \\
\hookrightarrow\Omega_{\hat X^{\mathcal O}}^{\bullet}:=
\varprojlim_{n\in\mathbb N}\nu_{X^{\mathcal O}}^*a_{et}\Omega^{\bullet}_{X^{\mathcal O,et}/p^n/(O_K/p^n)}
\end{eqnarray*}
where for $p=0$, we set $\Omega^0_{X^{\mathcal O,et}/p^n,\log}:=\mathbb Z/p^n$, and 
we recall $c:\hat X^{\mathcal O}\to X^{\mathcal O}$ the morphism in $\RTop$ 
is given by the completion along the ideal generated by $p$, and 
$a_{et}:\PSh(\hat X^{\mathcal O})\to\Shv_{et}(\hat X^{\mathcal O})$ is the sheaftification functor. 
It induces the canonical morphism of $C(X^{pet})$
\begin{equation*}
OL_{\hat X^{(p)}}:=r^*OL_{\hat X^{\mathcal O,(p)}}:
r^*\Omega_{\hat X^{(p)},\log,\mathcal O}^{\bullet}\hookrightarrow r^*\Omega_{\hat X^{\mathcal O}}^{\bullet}.
\end{equation*} 
Note that the inclusion $\Omega^l_{X^{\mathcal O,et},\log}/p^n\subset\Omega^{\bullet}_{X^{\mathcal O,et}/p^n,\log}$ is strict in general.
Note that 
\begin{equation*}
\Omega_{X^{pet},\log,\mathcal O}^{\bullet}:=\nu_X^*\Omega_{X^{et},\log,\mathcal O}^{\bullet}\in C(X^{pet}),
\end{equation*}
but 
\begin{eqnarray*}
\underline{\mathbb Z_p}_X:=\varprojlim_{n\in\mathbb N}\nu_X^*(\mathbb Z/p^n\mathbb Z)_{X^{et}}\in C(X^{pet}), \, \mbox{and} \;
r^*\Omega_{\hat X^{(p)},\log,\mathcal O}^{\bullet}
:=r^*\varprojlim_{n\in\mathbb N}\nu_{X^{\mathcal O}}^*a_{et}\Omega_{X^{\mathcal O,et}/p^n,\log}
\in C(X^{pet}) 
\end{eqnarray*}
are NOT the pullback of etale sheaves by $\nu_X$. 
We will denote for short $\Omega_{\hat X^{(p)},\log,\mathcal O}^{\bullet}:=r^*\Omega_{\hat X^{(p)},\log,\mathcal O}^{\bullet}\in C(X^{pet})$, 
in particular for $e:U\to X$ an etale map 
$\mathbb H^i_{\tau}(U,\Omega_{\hat X^{(p)},\log,\mathcal O}^{\bullet}):=\mathbb H^i_{\tau}(U,r^*\Omega_{\hat X^{(p)},\log,\mathcal O}^{\bullet})$, 
where $\tau$ is either the etale or pro-etale topology.
For $i\in\mathbb Z$, we will consider the formal de Rham cohomology of $X^{\mathcal O}$
\begin{equation*}
H^i_{DR}(\hat X^{\mathcal O}):=\mathbb H^i_{et}(X^{\mathcal O},\Omega_{\hat X^{\mathcal O}}^{\bullet})=
\mathbb H^i_{et}(X,r^*\Omega_{\hat X^{\mathcal O}}^{\bullet}), \;
\end{equation*}
and for $Z^{\mathcal O}\subset X^{\mathcal O}$ a closed embedding of integral models
\begin{equation*}
H^i_{DR,\hat Z^{\mathcal O}}(\hat X^{\mathcal O}):=\mathbb H^i_{et,Z^{\mathcal O}}(X^{\mathcal O},\Omega_{\hat X^{\mathcal O}}^{\bullet})=
\mathbb H^i_{et,Z}(X,r^*\Omega_{\hat X^{\mathcal O}}^{\bullet})
\end{equation*}
where the last equalities follows from proposition \ref{oetetprop}.
\item[(iii)'] We consider the morphism of sites 
\begin{equation*}
r:\SmVar(K)\to\Sch^{int,sm}/O_K, \; 
X^{\mathcal O}\in\Sch^{int,sm}/O_K\mapsto X:=X^{\mathcal O}\otimes_{O_K}K,
\end{equation*}
and for $X^{\mathcal O}\in\Sch^{int,sm}/O_K$ and $X:=X^{\mathcal O}\otimes_{O_K}K$
the commutative diagram of sites
\begin{equation*}
\xymatrix{\SmVar(K)\ar[r]^r\ar[d]^{o_X} & (\Sch^{int,sm}/O_K)\ar[d]^{o_X} \\
X^{et}\ar[r]^r & X^{\mathcal O,et}}
\end{equation*}
with $o_X(U^{\mathcal O}\to X^{\mathcal O})=U^{\mathcal O}$ and $o_X(U\to X)=U$.
We will consider the embedding of $C(\Sch^{int}/O_K)$
\begin{eqnarray*}
OL_{/O_K,an}:\Omega^{\bullet,an}_{/K,\log,\mathcal O}\hookrightarrow\Omega_{/O_K}^{\bullet,an}, \;
\mbox{for} \, X^{\mathcal O}\in(\Sch^{int,sm}/O_K), \\ 
OL_{/O_K,an}(X^{\mathcal O}):=OL_{\hat X^{\mathcal O,(p)}}(X^{\mathcal O}):
\Omega^{\bullet}_{\hat X^{(p)},\log,\mathcal O}(X^{\mathcal O})\hookrightarrow\Omega^{\bullet}_{\hat X^{\mathcal O}}(X^{\mathcal O}).
\end{eqnarray*}
and its restriction to $\Sch^{int,sm}/O_K\subset\Sch^{int}/O_K$.
We get the embedding of $C(\SmVar(K))$
\begin{eqnarray*}
OL_{/K,an}:=r^*OL_{/O_K,an}:r^*\Omega^{\bullet,an}_{/K,\log,\mathcal O}\hookrightarrow r^*\Omega_{/O_K}^{\bullet,an}.
\end{eqnarray*}
For $X^{\mathcal O}\in\Sch^{int,sm}/O_K$ and $i\in\mathbb Z$, we have 
\begin{equation*}
\mathbb H^i_{et}(X,r^*\Omega_{/O_K}^{\bullet,an})=\mathbb H^i_{et}(X^{\mathcal O},r_*r^*\Omega_{/O_K}^{\bullet,an}), \; 
\mathbb H^i_{et}(X,r^*\Omega^{\bullet,an}_{/K,\log,\mathcal O})=\mathbb H^i_{et}(X^{\mathcal O},r_*r^*\Omega^{\bullet,an}_{/K,\log,\mathcal O}),
\end{equation*}
where the last equalities follows from the fact that a smooth algebraic variety over $K$ admits locally for the etale topology a smooth integral model.
\end{itemize}
\end{defi}

Let $K$ be a field of characteristic zero which is complete for a $p$-adic norm.
Recall that $O_K\subset K$ denotes its ring of integers.
We will consider for $X\in\SmVar(K)$ and $X^{\mathcal O}\in\Sch^{int,sm}/O_K$ a smooth integral model of $X$, the canonical maps in $C(X^{et})$
\begin{eqnarray*}
S(X^{\mathcal O}/X):r^*\Omega^{\bullet}_{\hat X^{(p)},\log,\mathcal O}\to o_{X*}r^*\Omega^{\bullet,an}_{/K,\log,\mathcal O}, \; 
(U\to X)\mapsto (S(X^{\mathcal O}/X)(U\to X): \\ 
\varinjlim_{U\to V\to X \mbox{et}, V^{\mathcal O}\to X^{\mathcal O} \mbox{et}}
\Omega^{\bullet}_{\hat X^{(p)},\log,\mathcal O}(V^{\mathcal O})=\Omega^{\bullet,an}_{/K,\log,\mathcal O}(V^{\mathcal O})\to 
\varinjlim_{U\to U',\, U^{'\mathcal O} \mbox{int of }U'}\Omega^{\bullet,an}_{/K,\log,\mathcal O}(U^{'\mathcal O})), \\
S(X^{\mathcal O}/X):r^*\Omega_{\hat X^{\mathcal O}}^{\bullet}\to o_{X*}r^*\Omega_{/O_K}^{\bullet,an}, \; (U\to X)\mapsto \\
(S(X^{\mathcal O}/X)(U\to X):\varinjlim_{U\to V\to X \mbox{et}, V^{\mathcal O}\to X^{\mathcal O} \mbox{et}}
\Omega^{\bullet}_{\hat X^{\mathcal O}}(V^{\mathcal O})=\Omega_{/O_K}^{\bullet,an}(V^{\mathcal O})\to
\varinjlim_{U\to U',\, U^{'\mathcal O} \mbox{int of }U'}\Omega_{/O_K}^{\bullet,an}(U^{'\mathcal O}))
\end{eqnarray*}
where we recall $o_X:\SmVar(K)\to X^{et}$ is given by $o_X(U\to X)=U$. In particular 
\begin{equation*}
S(X^{\mathcal O}/X)\circ OL_{\hat X^{(p)}}=OL_{/K,an}\circ S(X^{\mathcal O}/X).
\end{equation*}

Let $X^{\mathcal O}\in\Sch^{int,sm}/O_K$.
For $h=(h_i:V^{\mathcal O}_i\to X^{\mathcal O})_{i\in I}$ be a family of morphism in $\Sch^{int,sm}/O_K$, we denote 
\begin{equation*}
V^{\mathcal O}_{\bullet}\in\Fun(\Delta(I),\Sch^{int,sm}/O_K), \; 
V^{\mathcal O}_{i_1,\cdots,i_l}:=V^{\mathcal O}_{i_1}\times_{X^{\mathcal O}}\cdots\times_{X^{\mathcal O}}V^{\mathcal O}_{i_l}
\end{equation*}

\begin{lem}\label{splitlem}
Let $X^{\mathcal O}\in\PSch^{sm}/O_K$ and $e:U^{\mathcal O}\to X^{\mathcal O}$ be an etale morphism of integral models 
with $U^{\mathcal O}\in\Sch^{int,sm}/O_K$ affine such that the (co)tangent bundle $\Omega^1_{U^{\mathcal O}}\in\Shv_O(U^{\mathcal O})$ is trivial and such that
$e^*\Br(X^{\mathcal O}/p)[p]=0$.
Let $\bar r:W^{\mathcal O}\to X^{\mathcal O}$ be a finite surjective morphism of integral model with $W^{\mathcal O}\in\Sch^{int}/O_K$.
Let $D^{\mathcal O}\subset X^{\mathcal O}$ be a divisor which is an integral model. Denote $V^{\mathcal O}:=W^{\mathcal O}\backslash\bar r^{-1}(D^{\mathcal O})$
Denote by $e_V:V^{\mathcal O}\times_{X^{\mathcal O}}U^{\mathcal O}\to V^{\mathcal O}$ and $\bar e_V:W^{\mathcal O}\times_{X^{\mathcal O}}U^{\mathcal O}\to W^{\mathcal O}$
the base change maps. Then for each $n\in\mathbb N$,
\begin{equation*}
e_V^*\Gamma(V^{\mathcal O},\Omega^1_{X^{\mathcal O}/p^n,\log})\simeq\oplus_{j=1}^sO_K/p^n[(\bar e_V^{-1}\bar r^{-1}(D^{\mathcal O}/p))_j]
\end{equation*}
where $(\bar e_V^{-1}\bar r^{-1}(D^{\mathcal O}/p))_j\subset\bar e_V^{-1}\bar r^{-1}(D^{\mathcal O}/p)$, $1\leq j\leq s$ 
are the principal divisors of $W^{\mathcal O}\times_{X^{\mathcal O}}U^{\mathcal O}/p$ contained in $\bar e_V^{-1}\bar r^{-1}(D^{\mathcal O}/p)$.
\end{lem}

\begin{proof}
We proceed by induction on $n\in\mathbb N$.
For $n=1$, this follows from the exact sequence
\begin{equation*}
\Gamma(V^{\mathcal O}\times_{X^{\mathcal O}}U^{\mathcal O},\Omega^1_{X^{\mathcal O}/p,\log})\to 
H^1_{\bar r^{-1}(D^{\mathcal O})\times_{X^{\mathcal O}}U^{\mathcal O}}(W^{\mathcal O}\times_{X^{\mathcal O}}U^{\mathcal O},\Omega^1_{X^{\mathcal O}/p,\log})\to
H^1(W^{\mathcal O}\times_{X^{\mathcal O}}U^{\mathcal O},\Omega^1_{X^{\mathcal O}/p,\log})
\end{equation*}
and the fact that 
$\bar e_V^*H^1(W^{\mathcal O},\Omega^1_{X^{\mathcal O}/p,\log})\simeq\bar e_V^*\Pic(W^{\mathcal O}/p)/[p]$ since $e^*\Br(X^{\mathcal O}/p)[p]=0$.
Note that 
\begin{equation*}
H^2_{\bar r^{-1}(D^{\mathcal O})\times_{X^{\mathcal O}}U^{\mathcal O}}(W^{\mathcal O}\times_{X^{\mathcal O}}U^{\mathcal O},O^*)\to
H^2(W^{\mathcal O}\times_{X^{\mathcal O}}U^{\mathcal O},O^*)
\end{equation*}
is injective since the restriction of the Picard group is surjective.
The induction then follows from the exact sequences
\begin{equation*}
0\to e_V^*\Gamma(V^{\mathcal O},\Omega^1_{X^{\mathcal O}/p,\log})\xrightarrow{p^{n-1}\cdot(-)}e_V^*\Gamma(V^{\mathcal O},\Omega^1_{X^{\mathcal O}/p^n,\log})
\xrightarrow{(-)/p}e_V^*\Gamma(V^{\mathcal O},\Omega^1_{X^{\mathcal O}/p^{n-1},\log})\to 0.
\end{equation*}
\end{proof}

We now construct functorial non canonical splittings
\begin{itemize}
\item[(i)]Let $Y^{\mathcal O}\in\PSch^{sm}/O_K$ and $e:U^{\mathcal O}\to Y^{\mathcal O}$ be a morphism of integral models with
$U^{\mathcal O}\in\Sch^{int,sm}/O_K$ affine of dimension $d$ such that the (co)tangent bundle $\Omega^1_{U^{\mathcal O}}\in\Shv_O(U^{\mathcal O})$ is trivial. 
In particular
\begin{equation*}
\Omega^1(U^{\mathcal O})=\oplus_{i=1}^dO(U^{\mathcal O})dx_i,  \; \Omega^1(U^{\mathcal O}/p^n)=\oplus_{i=1}^dO(U^{\mathcal O}/p^n)dx_i,
\Omega^1(\hat U^{(p)})=\oplus_{i=1}^dO(\hat U^{(p)})dx_i,
\end{equation*} 
Consider a cartesian square in $\Sch^{int,sm}/O_K$
\begin{equation*}
\xymatrix{V^{\mathcal O}_1\times_{Y^{\mathcal O}}V^{\mathcal O}_2\ar[d]^{p_1}\ar[rr]^{p_2} & \, & \ar[d]^{r_2}V^{\mathcal O}_2 \\
V^{\mathcal O}_1\ar[rr]^{r_1} & \, & Y^{\mathcal O}},
\end{equation*}
where $r_1:V^{\mathcal O}_1\to X^{\mathcal O}$ and $r_2:V^{\mathcal O}_2\to X^{\mathcal O}$ are etale, $V^{\mathcal O}_1$ and $V^{\mathcal O}_2$ are affine, 
and $r_1:V^{\mathcal O}_1\to r_1(V^{\mathcal O}_1)$, $r_2:V^{\mathcal O}_2\to r_2(V^{\mathcal O}_2)$ are finite etale.
Then the cotangent bundles 
$\Omega^1_{V^{\mathcal O}_1\times_{Y^{\mathcal O}}U^{\mathcal O}}\in\Shv_O(V^{\mathcal O}_1\times_{Y^{\mathcal O}}U^{\mathcal O})$, 
$\Omega^1_{V^{\mathcal O}_2\times_{Y^{\mathcal O}}U^{\mathcal O}}\in\Shv_O(V^{\mathcal O}_2\times_{Y^{\mathcal O}}U^{\mathcal O})$, 
$\Omega^1_{V^{\mathcal O}_{1}\times_{Y^{\mathcal O}}V^{\mathcal O}_{2}\times_{Y^{\mathcal O}}U^{\mathcal O}}
\in\Shv_O(V^{\mathcal O}_{1}\times_{Y^{\mathcal O}}V^{\mathcal O}_{2}\times_{Y^{\mathcal O}}U^{\mathcal O})$ 
are trivials and we have in particular the splittings
\begin{eqnarray*}
\Omega^1(V^{\mathcal O}_1\times_{Y^{\mathcal O}}U^{\mathcal O})=\oplus_{i=1}^dO(V^{\mathcal O}_1\times_{Y^{\mathcal O}}U^{\mathcal O})dx_i, \; 
\Omega^1(V^{\mathcal O}_2\times_{Y^{\mathcal O}}U^{\mathcal O})=\oplus_{i=1}^dO(V^{\mathcal O}_2\times_{Y^{\mathcal O}}U^{\mathcal O})dx_i,\\
\Omega^1(V^{\mathcal O}_{1}\times_{Y^{\mathcal O}}V^{\mathcal O}_{2}\times_{Y^{\mathcal O}}U^{\mathcal O})=
\oplus_{i=1}^dO(V^{\mathcal O}_{1}\times_{Y^{\mathcal O}}V^{\mathcal O}_{2}\times_{Y^{\mathcal O}}U^{\mathcal O})dx_i.
\end{eqnarray*}
Assume there exist compactifications
\begin{equation*}
r_1:V^{\mathcal O}_1\xrightarrow{j_1}W^{\mathcal O}_1\xrightarrow{\bar r_1} U^{\mathcal O}, \; r_2:V^{\mathcal O}_2\xrightarrow{j_2}W^{\mathcal O}_2\xrightarrow{\bar r_2} U^{\mathcal O}
\end{equation*}
where $j_1,j_2$ are open embeddings and $\bar r_1,\bar r_2$ finite (and surjective) morphisms, and $W^{\mathcal O}_1,W^{\mathcal O}_2\in\Sch^{int}/O_K$ are integral models,
such that
\begin{eqnarray*}
(W^{\mathcal O}_1\backslash V^{\mathcal O}_1)/p=V(f_1)\subset W^{\mathcal O}_1, \, (W^{\mathcal O}_2\backslash V^{\mathcal O}_2)/p=V(f_2)\subset W^{\mathcal O}_2, \\
(W^{\mathcal O}_1\times_{X^{\mathcal O}}W^{\mathcal O}_2\backslash V^{\mathcal O}_1\times_{X^{\mathcal O}}W^{\mathcal O}_2)/p=
V(f_1)\subset W^{\mathcal O}_1\times_{U^{\mathcal O}}W^{\mathcal O}_2, \\
(W^{\mathcal O}_1\times_{X^{\mathcal O}}W^{\mathcal O}_2\backslash W^{\mathcal O}_1\times_{X^{\mathcal O}}V^{\mathcal O}_2)/p=
V(f_2)\subset W^{\mathcal O}_1\times_{X^{\mathcal O}}W^{\mathcal O}_2
\end{eqnarray*}
are irreducible and reduced given by one equation.
Consider (non canonical) splittings of $O_K/p^n$-vector spaces
\begin{eqnarray*}
O(W_{1}^{\mathcal O}\times_{Y^{\mathcal O}}U^{\mathcal O}/p^n)=f_1O(W_{1}^{\mathcal O}\times_{Y^{\mathcal O}}U^{\mathcal O}/p^n)\oplus T_{f_1}, \; 
T_{f_1}\supset\{\partial_{x_l}f_1, 1\leq l\leq d,\}, , \\
O(W_{2}^{\mathcal O}\times_{Y^{\mathcal O}}U^{\mathcal O}/p^n)=f_2O(W_{1U}^{\mathcal O}/p^n)\oplus T_{f_2}, \; T_{f_2}\supset\{\partial_{x_l}f_2, 1\leq l\leq d. \}.
\end{eqnarray*}
Then, we get the (non canonical) splitting of $O_K/p^n$-vector spaces
\begin{eqnarray*}
O(W_{1}^{\mathcal O}\times_{Y^{\mathcal O}}W_2^{\mathcal O}\times_{Y^{\mathcal O}}U^{\mathcal O}/p^n)=
O(W_{1}^{\mathcal O}\times_{Y^{\mathcal O}}W_2^{\mathcal O}\times_{Y^{\mathcal O}}U^{\mathcal O}/p^n)\oplus T_{f_1f_2}, \\
T_{f_1f_2}:=f_2T_{f_1}[e_2]+f_1T_{f_2}[e_1]+T_{f_1}[e_2]+T_{f_2}[e_1].
\end{eqnarray*}
\item[(ii)]Let $X^{\mathcal O}\in\PSch^{sm}/O_K$ irreducible. Let $\pi:X^{\mathcal O}\to\mathbb P_{O_K}^{d_X}$ be a finite surjective morphism.
We say that an etale morphism $r:V^{\mathcal O}\to X^{\mathcal O}$ in $\Sch^{int,sm}/O_K$ belong to $\mathcal H(\pi)\subset X^{\mathcal O,et}$
if $r:V^{\mathcal O}\to r(V^{\mathcal O})$ is finite etale and if there exists a compactification 
$r:V^{\mathcal O}\xrightarrow{j}W^{\mathcal O}\xrightarrow{\bar r}X^{\mathcal O}$ with $j$ an open embedding of integral models, $\bar r$ finite surjective,
such that 
\begin{equation*}
D^{\mathcal O}:=W^{\mathcal O}\backslash V^{\mathcal O}=\bar{r}^{-1}\pi^{-1}(H_{\bar r})=V(f_V)\subset W^{\mathcal O}, 
\end{equation*}
where $H_{\bar r}:=V(f_V)\subset\mathbb P_{O_K}^{d_X}$ is an hypersurface.
For $(r:V^{\mathcal O}\to X^{\mathcal O})\in\mathcal H(\pi)$, we take a pencil $\mathbb P_{O_K}^1\subset\mathbb P(H_{\bar r})$ passing through $H_{\bar r}$
and consider the universal family $\tilde D^{\mathcal O}=V(\tilde f_V)\subset W^{\mathcal O}\times_{O_K}\mathbb P_{O_K}^1$, 
so that $D^{\mathcal O}=\tilde D^{\mathcal O}\cap(W^{\mathcal O}\times 0)$ and denote 
$\tilde V^{\mathcal O}:=(W^{\mathcal O}\times_{O_K}\mathbb P_{O_K}^1)\backslash\tilde D^{\mathcal O}$ and
\begin{equation*}
\tilde r:\tilde V^{\mathcal O}\xrightarrow{\tilde j}W^{\mathcal O}\times_{O_K}\mathbb P_{O_K}^1\xrightarrow{\bar r\times I}
X^{\mathcal O}\times_{O_K}\mathbb P_{O_K}^1
\end{equation*}
Then $\tilde D^{\mathcal O}/p=V(\tilde f_V)\subset W^{\mathcal O}\times_{O_K}\mathbb P_{O_K}^1$ is irreducible and reduced since the morphism
$(\bar r\times I)_{|\tilde D^{\mathcal O}/p}:\tilde D^{\mathcal O}/p\to W^{\mathcal O}/p$ is a birational, and it is given by one equation.
Let $e=(e_i:U^{\mathcal O}_i\to X^{\mathcal O})_{1\leq i\leq s}$ be an etale cover (finite) in $\mathcal H(\pi)\subset X^{\mathcal O,et}$
such that for each $1\leq i\leq s$, $U_i^{\mathcal O}\in\Sch^{int,sm}/O_K$ is affine 
whose (co)tangent bundle $\Omega^1_{U_i^{\mathcal O}}\in\Shv_O(U_i^{\mathcal O})$ is trivial and $e_i^*\Br(X^{\mathcal O}/p)[p]=0$.
For each $\tilde f\in O(\mathbb P_{O_K}^{d_X}\times_{O_K}\mathbb P^1_{O_K})$ irreducible of bidegree $(-,1)$ and $1\leq i\leq s$, we choose (non canonical) splittings
\begin{eqnarray*}
O(X^{\mathcal O}\times_{O_K}\mathbb P^1_{O_K}\times_{X^{\mathcal O}}U_i^{\mathcal O})/p^n=
\tilde fO(X^{\mathcal O}\times_{O_K}\mathbb P^1_{O_K}\times_{O_K}U_i^{\mathcal O}/p^n)\oplus T_{\tilde f}, \; T_f:=T_{\tilde f}(0) \\
a_{\bar r}^{-1}(T_f)\supset\{\partial_{x_l}f', 1\leq l\leq d_X, f'|f, 
f=\tilde f(0)\in\lim_{\bar r:W^{\mathcal O}\to X^{\mathcal O}f.s.}\in O(W^{\mathcal O}\times_{X^{\mathcal O}}U_i^{\mathcal O})\}
\end{eqnarray*}
where $f.s.$ stand for finite surjective such that $\bar r_{|W^{\mathcal O}\backslash V(f)}$ is (finite) etale, compatible for $n\in\mathbb N$ 
(by induction on $n\in\mathbb N$ taking pullbacks by $(-)/p(Y^{\mathcal O}):O(Y^{\mathcal O}/p^n)\to O(Y^{\mathcal O}/p^{n-1})$, $Y^{\mathcal O}\in\Sch^{int}/O_K$).
By (i) applied to the cartesian squares 
\begin{eqnarray*}
\tilde r_1:\tilde V_1^{\mathcal O}\to X^{\mathcal O}\times_{O_K}\mathbb P_{O_K}^1 \, 
\tilde r_2:\tilde V_2^{\mathcal O}\to X^{\mathcal O}\times_{O_K}\mathbb P_{O_K}^1, \\
(W^{\mathcal O}\times_{O_K}\mathbb P_{O_K}^1)\backslash\tilde V_1^{\mathcal O}=V(\tilde f_1)\subset W^{\mathcal O}\times_{O_K}\mathbb P_{O_K}^1, \,
(W^{\mathcal O}\times_{O_K}\mathbb P_{O_K}^2)\backslash\tilde V_1^{\mathcal O}=V(\tilde f_2)\subset W^{\mathcal O}\times_{O_K}\mathbb P_{O_K}^1, \\
\end{eqnarray*}
where $\tilde f_1,\tilde f_2\in O(\mathbb P_{O_K}^{d_X}\times_{O_K}\mathbb P^1_{O_K})$ are irreducible, we get 
\begin{eqnarray*}
T^1_{n,U^{\mathcal O}_{\bullet}}\in\Fun(\mathbb N,P([1,\cdots,s]),\PSh(\mathcal H(\pi)), 
(r:V^{\mathcal O}\xrightarrow{j}W^{\mathcal O}\xrightarrow{\bar r}X^{\mathcal O}) \\
\mapsto T^1_{n,U^{\mathcal O}_J}(V^{\mathcal O}):=\frac{T_{f_V}}{f_V}\subset\Omega^1_{X^{\mathcal O,et}/p^n}(V^{\mathcal O}\times_{X^{\mathcal O}}U_J^{\mathcal O}), \\
T'_{n,U^{\mathcal O}_{\bullet}}\in\Fun(\mathbb N,P([1,\cdots,s]),\PSh(\mathcal H(\pi)), 
(r:V^{\mathcal O}\xrightarrow{j}W^{\mathcal O}\xrightarrow{\bar r}X^{\mathcal O}) \\
\mapsto T_{n,U^{\mathcal O}_J}(V^{\mathcal O}):=
O(V^{\mathcal O}\times_{X^{\mathcal O}}U^{\mathcal O}_J)/p^n\oplus(\oplus_{k\in\mathbb N,k\geq 2}\frac{T_{f_V}}{f_V^k},
\subset\Omega^1_{X^{\mathcal O,et}/p^n}(V^{\mathcal O}\times_{X^{\mathcal O}}U_J^{\mathcal O})\\
D^{\mathcal O}/p^n:=(W^{\mathcal O}\backslash V^{\mathcal O})/p^n=V(f_V)\subset W^{\mathcal O}, \\
\tilde D^{\mathcal O}/p^n:=((W^{\mathcal O}\times_{O_K}\mathbb P^1_{O_K})\backslash\tilde V^{\mathcal O})/p^n=V(\tilde f_V)\subset W^{\mathcal O}\times_{O_K}\mathbb P^1_{O_K},
\tilde f_V(0)=f_V, T_{f_V}:=T_{\tilde f_V}(0)
\end{eqnarray*}
and hence in particular (non canonical) splittings of complexes of $O_K/p^n$ modules
for $r:=(r_i:V^{\mathcal O}_i\to X^{\mathcal O})_{i\in I}$ an etale cover in $\mathcal H(\pi)\subset X^{\mathcal O,et}$ and $n\in\mathbb N$
\begin{eqnarray*}
\Gamma(V^{\mathcal O}_{\bullet}\times_{X^{\mathcal O}}U_{\bullet}^{\mathcal O},\Omega^1_{X^{\mathcal O,et}/p^n})=
\Gamma(V^{\mathcal O}_{\bullet},T^1_{n,U^{\mathcal O}_{\bullet}})\oplus\Gamma(V^{\mathcal O}_{\bullet},T'_{n,U^{\mathcal O}_{\bullet}})
\end{eqnarray*}
functorial in $n\in\mathbb N$ and $r\in\mathcal H(\pi)$.
\item[(iii)]Let $X^{\mathcal O}\in\PSch^{sm}/O_K$ irreducible. Let $\pi:X^{\mathcal O}\to\mathbb P_{O_K}^{d_X}$ be a finite surjective morphism.
Let $r=(r_i:V_i^{\mathcal O}\to X^{\mathcal O})_{i\in I}\in\mathcal H(\pi)$ be an etale cover in $\mathcal H(\pi)$ (see (ii)). 
Let $e=(e_i:U^{\mathcal O}_i(r)\to X^{\mathcal O})_{1\leq i\leq s}$ be an etale cover (finite) in $\mathcal H(\pi)\subset X^{\mathcal O,et}$
such that for each $1\leq i\leq s$, $U_i^{\mathcal O}(r)\in\Sch^{int,sm}/O_K$ is affine 
whose (co)tangent bundle $\Omega^1_{U_i^{\mathcal O}(r)}\in\Shv_O(U_i^{\mathcal O})$ is trivial and $e_i^*\Br(X^{\mathcal O}/p)[p]=0$,
and such that no irreducible components of $X^{\mathcal O}\backslash e_i(U_i^{\mathcal O}(r))$
are contained in $X^{\mathcal O}\backslash r_i(V_i^{\mathcal O})$ for all $i\in I$.
Take a pencil $\mathbb P_{O_K}^1\subset\mathbb P(H_{\bar r})$ passing through $H_{\bar r}$
For each $\tilde f\in O(\mathbb P_{O_K}^{d_X}\times_{O_K}\mathbb P^1_{O_K})$ irreducible of bidegree $(-,1)$ and $1\leq i\leq s$, we choose (non canonical) splittings
\begin{eqnarray*}
O(X^{\mathcal O}\times_{O_K}\mathbb P^1_{O_K}\times_{X^{\mathcal O}}U_i^{\mathcal O}(r))/p^n=
\tilde fO(X^{\mathcal O}\times_{O_K}\mathbb P^1_{O_K}\times_{O_K}U_i^{\mathcal O}(r)/p^n)\oplus T_{\tilde f}, \; T_f:=T_{\tilde f}(0) \\
a_{\bar r}^{-1}(T_f)\supset\{\partial_{x_l}f', 1\leq l\leq d_X, f'|f, 
f=\tilde f(0)\in\lim_{\bar r:W^{\mathcal O}\to X^{\mathcal O}f.s.}\in O(W^{\mathcal O}\times_{X^{\mathcal O}}U_i^{\mathcal O}(r))\}
\end{eqnarray*}
For each $\tilde f\in O(\mathbb P_{O_K}^{d_X}\times_{O_K}\mathbb P^1_{O_K})$ irreducible of bidegree $(-,1)$ and $f:=\tilde f(0)$, 
we consider $V_f:=\varprojlim_{\bar r:W^{\mathcal O}\to X^{\mathcal O}, f.s.} W^{\mathcal O}\backslash V(f)$, 
where $f.s.$ stand for finite surjective such that $\bar r_{|W^{\mathcal O}\backslash V(f)}$ is (finite) etale.
We then choose for $f\in r$ i.e. $f=f_i$ for some $i\in I$ where $W_i^{\mathcal O}\backslash V_i^{\mathcal O}=V(f_i)\subset W_i^{\mathcal O}$, 
a (non canonical) splitting of $O_K/p$-vector spaces
\begin{equation*}
T^1_{1,U^{\mathcal O}_{[1,\cdots,s]}(r)}(V_f^{\mathcal O})=
e_{V_f}^*\Omega^1_{X^{\mathcal O,et}/p,\log}(V_f^{\mathcal O})\oplus H^1_{1,U^{\mathcal O}_{[1,\cdots,s]}(r)}(V_f^{\mathcal O})
\subset O(V^{\mathcal O}_f\times_{X^{\mathcal O}}U_{[1,\cdots,s]}^{\mathcal O}(r)).
\end{equation*} 
It induces by pullback a (non canonical) splitting of complex of $O_K/p$-vector spaces
\begin{eqnarray*}
T^1_{1,U^{\mathcal O}_{\bullet}}(V_f^{\mathcal O})=
e_{V_f}^*\Omega^1_{X^{\mathcal O,et}/p,\log}(V_f^{\mathcal O})\oplus H^1_{1,U^{\mathcal O}_{\bullet}(r)}(V_f^{\mathcal O}), \\
H^1_{1,U^{\mathcal O}_J(r)}(V_f^{\mathcal O}):=H^1_{1,U^{\mathcal O}_{[1,\cdots,s]}(r)}(V_f^{\mathcal O})\cap O(V^{\mathcal O}_f\times_{X^{\mathcal O}}U_J^{\mathcal O}(r)).
\end{eqnarray*} 
By lemma \ref{splitlem}, we get since $e_J^*\Br(X^{\mathcal O}/p)[p]=0$, for each $J\subset[1,\cdots,s]$ and $n\in\mathbb N$
(non canonical) compatible splittings of complexes of $O_K/p^n$-modules
\begin{eqnarray*}
\Gamma(V^{\mathcal O}_{\bullet},T^1_{n,U^{\mathcal O}_{\bullet}(r)})=
\Gamma(V^{\mathcal O}_{\bullet},a_{et}\Omega^1_{X^{\mathcal O,et}/p^n,\log})\oplus\Gamma(V^{\mathcal O}_{\bullet},H^1_{n,U^{\mathcal O}_{\bullet}(r)}), \\
\Gamma(V^{\mathcal O}_{\bullet},H^1_{n,U^{\mathcal O}_{\bullet}(r)}):=((-)/p^{n-1})^{-1}\Gamma(V^{\mathcal O}_{\bullet},H^1_{1,U^{\mathcal O}_{\bullet}(r)}), \\
\Gamma(V^{\mathcal O}_L,H^1_{1,U^{\mathcal O}_J(r)}):=H^1_{1,U^{\mathcal O}_J}(V_f^{\mathcal O})\cap O(V^{\mathcal O}_L\times_{X^{\mathcal O}}U_J^{\mathcal O}(r))
\end{eqnarray*}
functorial in $n\in\mathbb N$ and in $r\in\mathcal H(\pi)\subset X^{\mathcal O,et}$.
We then get (non canonical) splittings of complexes of $O_K/p^n$ modules
\begin{eqnarray*}
\Gamma(V^{\mathcal O}_{\bullet}\times_{X^{\mathcal O}}U_{\bullet}^{\mathcal O}(r),\Omega^1_{X^{\mathcal O,et}/p^n})=
\Gamma(V^{\mathcal O}_{\bullet},a_{et}\Omega^1_{X^{\mathcal O,et}/p^n,\log})\oplus\Gamma(V^{\mathcal O}_{\bullet},NL_{n,U^{\mathcal O}_{\bullet}(r)}), \\
\Gamma(V^{\mathcal O}_{\bullet},NL_{n,U^{\mathcal O}_{\bullet}(r)}):=
\Gamma(V^{\mathcal O}_{\bullet},H^1_{n,U^{\mathcal O}_{\bullet}(r)})\oplus\Gamma(V^{\mathcal O}_{\bullet},T'_{n,U^{\mathcal O}_{\bullet}(r)})
\end{eqnarray*}
functorial in $n\in\mathbb N$ and in $r\in\mathcal H(\pi)\subset X^{\mathcal O,et}$.
We get for $r:=(r_i:V^{\mathcal O}_i\to X^{\mathcal O})_{i\in I}$ an etale cover in $\mathcal H(\pi)\subset X^{\mathcal O,et}$ and
$e=(e_i:U^{\mathcal O}_i(r)\to X^{\mathcal O})_{1\leq i\leq s}$ be an etale cover (finite) in $\mathcal H(\pi)\subset X^{\mathcal O,et}$
such that for each $1\leq i\leq s$, $U_i^{\mathcal O}(r)\in\Sch^{int,sm}/O_K$ is affine 
whose (co)tangent bundle $\Omega^1_{U_i^{\mathcal O}(r)}\in\Shv_O(U_i^{\mathcal O})$ is trivial and $e_i^*\Br(X^{\mathcal O}/p)[p]=0$,
and such that no irreducible components of $X^{\mathcal O}\backslash e_i(U_i^{\mathcal O}(r))$
are contained in $X^{\mathcal O}\backslash r_i(V_i^{\mathcal O})$ for all $i\in I$, (non canonical) splittings of complexes of $O_K/p^n$ modules
\begin{eqnarray}\label{splitO1}
\Omega^p_{X^{\mathcal O,et}/p^n}(V^{\mathcal O}_{\bullet}\times_{X^{\mathcal O}}U_{\bullet}^{\mathcal O}(r))=
a_{et}\Omega^p_{X^{\mathcal O}/p^n,\log}(V^{\mathcal O}_{\bullet})\oplus\Omega^p_{n,U_{\bullet}^{\mathcal O}(r),nl}(V^{\mathcal O}_{\bullet}), \\ 
\Omega^1_{n,U_{\bullet}^{\mathcal O}(r),nl}(V^{\mathcal O}_{\bullet}):=NL_{n,U_{\bullet}^{\mathcal O}(r)}(V^{\mathcal O}_{\bullet}), \\ 
\Omega^p_{n,U^{\mathcal O}(r),nl}(V^{\mathcal O}_{\bullet}):=
\oplus_{l=1}^pNL_{n,U^{\mathcal O}_{\bullet}(r)}(V^{\mathcal O}_{\bullet})\wedge
\Omega^{p-1}_{X^{\mathcal O,et}}(V^{\mathcal O}_{\bullet}\times_{X^{\mathcal O}}U_{\bullet}^{\mathcal O}(r))
\end{eqnarray}
functorial in $n\in\mathbb N$ and in $r\in\mathcal H(\pi)\subset X^{\mathcal O,et}$.
\end{itemize}

\begin{lem}\label{UUO}
Let $K$ be a field of characteristic zero which is complete for a $p$-adic norm. Recall that $O_K\subset K$ denotes its ring of integers.
Let $X\in\PSmVar(K)$ with good reduction.
Let $X^{\mathcal O}\in\PSch^{sm}/O_K$ be a smooth integral model of $X$, 
in particular $X^{\mathcal O}\otimes_{O_K}K=X$ and $X^{\mathcal O}$ is smooth over $O_K$.
Let $x\in X^{\mathcal O}$.
There exists an etale map $e:U^{\mathcal O}\to X^{\mathcal O}$ with $U^{\mathcal O}\in\Sch^{int,sm}/O_K$) with $x=e(x')$, $x'\in U^{\mathcal O}$, 
such that $\Pic(U^{\mathcal O}/p)=0$ and $e^*\Br(X^{\mathcal O}/p)[p]=0$.
\end{lem}

\begin{proof}
Since $\Br(X^{\mathcal O}/p)[p]$ is finite as $X^{\mathcal O}/p$ is smooth proper, 
there exists an etale map $e:U^{\mathcal O}\to X^{\mathcal O}$ with $U^{\mathcal O}\in\Sch^{int,sm}/O_K$) with $x=e(x')$, $x'\in U^{\mathcal O}$, 
such that $e^*\Br(X^{\mathcal O}/p)[p]=0$.
Since $O_K/p$ is a finite field, $\Pic^o(Y)$ is finite, hence $\Pic(Y)$ is finitely generated for any $Y\in\SmVar(O_K/p)$ 
since $\Pic^s(Y)(O_K/p)$ is a finite group where $\Pic^s(Y)$ is the Picard scheme. 
Hence, up to shrinking $U^{\mathcal O}$ by removing a finite number of closed subsets, we may assume that $\Pic(U^{\mathcal O}/p)=0$.
\end{proof}

We have the following result :

\begin{prop}\label{UXpet}
Let $K$ be a field of characteristic zero which is complete for a $p$-adic norm. Recall that $O_K\subset K$ denotes its ring of integers.
\begin{itemize}
\item[(i0)]Let $X\in\PSmVar(K)$ with good reduction 
and $X^{\mathcal O}\in\PSch^{sm}/O_K$ be a smooth integral model of $X$, in particular $X^{\mathcal O}\otimes_{O_K}K=X$,
Then for $\alpha\in H_{pet}^q(X^{\mathcal O},\Omega^p_{\hat X^{(p)},\log,\mathcal O})=0$, $p,q\in\mathbb Z$, $q\neq 0$, there exists
an etale map $e=e_{\alpha}:U^{\mathcal O}\to X^{\mathcal O}$ in $\Sch^{int,sm}/O_K$ (depending on $\alpha$) with $U^{\mathcal O}$ is affine with trivial tangent bundle 
$e^*\Br(X^{\mathcal O}/p)[p]=0$, such that $e^*\alpha=0$.
\item[(i)]Let $X\in\PSmVar(K)$ with good reduction.
Let $X^{\mathcal O}\in\PSch^{sm}/O_K$ be a smooth integral model of $X$, in particular
$X^{\mathcal O}\otimes_{O_K}K=X$ and $X^{\mathcal O}$ is smooth projective over $O_K$. Then, for each $p,q\in\mathbb Z$,
\begin{eqnarray*}
T(E,\varprojlim)(X^{\mathcal O}):H^q_{et}(X,\Omega^p_{\hat X^{\mathcal O,(p)},\log,\mathcal O})\xrightarrow{\sim}
H^q_{pet}(X,\Omega^p_{\hat X^{\mathcal O,(p)},\log,\mathcal O})
\end{eqnarray*}
is an isomorphism. 
\item[(ii)]Let $X\in\SmVar(K)$ with good reduction.
Let $X^{\mathcal O}\in\Sch^{int,sm}/O_K$ be a smooth integral model of $X$, in particular
$X^{\mathcal O}\otimes_{O_K}K=X$ and $X^{\mathcal O}$ is smooth over $O_K$. Then, for each $p,q\in\mathbb Z$,
\begin{eqnarray*}
S(X^{\mathcal O}/X):H^q_{et}(X,\Omega^p_{\hat X^{\mathcal O,(p)},\log,\mathcal O})\hookrightarrow H^q_{et}(X,r^*\Omega^{p,an}_{/K,\log,\mathcal O})
\end{eqnarray*}
is injective.
\end{itemize}
\end{prop}

\begin{proof}
\noindent(i0): We have for $Y^{\mathcal O}\in\Sch^{int,sm}/O_K$
\begin{equation*}
H^q_{pet}(Y,r^*\Omega^p_{\hat U,\log,\mathcal O})=H^q_{pet}(Y^{\mathcal O},\Omega^p_{\hat Y,\log,\mathcal O})
=H^q_{pet}(Y^{\mathcal O}/p,\Omega^p_{\hat Y,\log,\mathcal O})=\varprojlim_{n\in\mathbb N}H^q_{et}(Y^{\mathcal O}/p,\Omega^p_{Y^{\mathcal O}/p^n,\log}). 
\end{equation*}
Recall that for a variety $X\in\Var(L)$ over a field $L$, and $F\in\PSh(X^{et})$, we have, for $q\in\mathbb Z$,
\begin{eqnarray*}
H_{et}^q(X,F)=\varinjlim_{(r_i:X_i\to X)_{i\in I}}H^qF(X_{\bullet}), \\
(r_i:X_i\to X)_{i\in I} \, \mbox{s.e.c}, \; X_{\bullet}\in\Fun(\Delta(I),\Var(L)), \; X_{i_1,\cdots,i_l}:=X_{i_1}\times_{X}\cdots\times_{X}X_{i_l}
\end{eqnarray*}
where $s.e.c$ means standard etale covers (i.e $r_i$ are etale, $X_i$ are affine and $r_i:X_i\to r_i(X_i)$ finite etale) 
that is Chech etale cohomology coincide with etale cohomology.
Let $\pi:X^{\mathcal O}\to\mathbb P_{O_K}^{d_X}$ be a finite surjective morphism.
Hence there exists $r=r_{\alpha}=(r_i:V_i^{\mathcal O}\to X^{\mathcal O})$ in $\mathcal H(\pi)\subset X^{\mathcal O,et}$ depending on $\alpha$ such that 
\begin{equation*}
\alpha\in H^q\Omega^p_{X^{\mathcal O}/p^n,\log}(V^{\mathcal O}_{\bullet}) 
\end{equation*}
Let $e:U^{\mathcal O}(r)\to X^{\mathcal O}$ in $\mathcal H(\pi)\subset X^{\mathcal O,et}$ with $U^{\mathcal O}(r)$ affine with trivial tangent bundle and 
$e^*\Br(X^{\mathcal O}/p)[p]=0$
such that no irreducible components of $X^{\mathcal O}\backslash e(U^{\mathcal O})$ is contained in $X^{\mathcal O}\backslash r_i(V_i^{\mathcal O})$.
By the splitting (\ref{splitO1})
\begin{equation*}
0=H^q\Omega^p_{X^{\mathcal O,et}/p^n}(V^{\mathcal O}_{\bullet}\times_{X^{\mathcal O}}U^{\mathcal O}(r))=
H^qa_{et}\Omega^p_{X^{\mathcal O}/p^n,\log}(V^{\mathcal O}_{\bullet})\oplus H^q\Omega^p_{n,U^{\mathcal O}(r),nl}(V^{\mathcal O}_{\bullet}).
\end{equation*}
This proves (i0).

\noindent(i):We have the following commutative diagram
\begin{equation*}
\xymatrix{H^q_{et}(X^{\mathcal O},\Omega^p_{\hat X^{\mathcal O,(p)},\log,\mathcal O})\ar[rrr]^{T(E,\varprojlim)(X^{\mathcal O})}\ar[d]^{\ad(Rr_*,r^*)(-)} & \, & \, &
H^q_{pet}(X^{\mathcal O},\Omega^p_{\hat X^{\mathcal O,(p)},\log,\mathcal O})\ar[d]^{\ad(Rr_*,r^*)(-)} \\
H^q_{et}(X,\Omega^p_{\hat X^{\mathcal O,(p)},\log,\mathcal O})\ar[rrr]^{T(E,\varprojlim)(X^{\mathcal O})} & \, & \, &
H^q_{pet}(X,\Omega^p_{\hat X^{\mathcal O,(p)},\log,\mathcal O})}
\end{equation*}
whose column are isomorphism by proposition \ref{oetetprop}.
We have for $Y^{\mathcal O}\in\Sch^{int,sm}/O_K$
\begin{equation*}
H^q_{pet}(Y^{\mathcal O},\Omega^p_{\hat Y,\log,\mathcal O})
=H^q_{pet}(Y^{\mathcal O}/p,\Omega^p_{\hat Y,\log,\mathcal O})=\varprojlim_{n\in\mathbb N}H^q_{et}(Y^{\mathcal O}/p,\Omega^p_{Y^{\mathcal O}/p^n,\log}). 
\end{equation*}
Hence we have to show that
\begin{eqnarray*}
T(E,\varprojlim)(X^{\mathcal O}):H^q_{et}(X^{\mathcal O},\Omega^p_{\hat X^{\mathcal O,(p)},\log,\mathcal O})\to
H^q_{pet}(X^{\mathcal O},\Omega^p_{\hat X^{\mathcal O,(p)},\log,\mathcal O})=\varprojlim_{n\in\mathbb N}H^q_{et}(X^{\mathcal O},\Omega^p_{X^{\mathcal O,et}/p^n,\log})
\end{eqnarray*}
is an isomorphism. 
Let $\pi:X^{\mathcal O}\to\mathbb P_{O_K}^{d_X}$ be a finite surjective morphism. 
We have then
\begin{eqnarray*}
T(E,\varprojlim)(X^{\mathcal O}):
H^q_{et}(X^{\mathcal O},\Omega^p_{\hat X^{\mathcal O,(p)},\log,\mathcal O})= 
\varinjlim_{r=(r_i:V^{\mathcal O}\to X^{\mathcal O})_{i\in I}\mbox{s.e.c}, r\in\mathcal H(\pi)}
H^q\Gamma(V^{\mathcal O}_{\bullet},\Omega^p_{\hat X^{\mathcal O,(p)},\log,\mathcal O}) \\
\to H^q_{pet}(X^{\mathcal O},\Omega^p_{\hat X^{\mathcal O,(p)},\log,\mathcal O})=
\varprojlim_{n\in\mathbb N}\varinjlim_{r=(r_i:V^{\mathcal O}_i\to X^{\mathcal O})_{i\in I}\mbox{s.e.c}, r\in\mathcal H(\pi)}
H^q\Gamma(V^{\mathcal O}_{\bullet},\Omega^p_{X^{\mathcal O,et}/p^n,\log,\mathcal O})
\end{eqnarray*}
whose columns are isomorphism since for a noetherian scheme $Y\in\Sch$ and $Y=\cup_{j=1}^tY_j$ an affine cover $\tau$, and $F\in\PSh(X^{et})$, we have, for $q\in\mathbb Z$,
\begin{eqnarray*}
H_{et}^q(Y,F)=\varinjlim_{(r_i:Y'_i\to Y)_{i\in I}}H^qF(Y_{\bullet}), \\
(r_i:Y'_i\to Y)_{i\in I} \, \mbox{s.e.c}, \; Y'_{\bullet}\in\Fun(\Delta(I),\Var(L)), \; Y'_{i_1,\cdots,i_l}:=Y'_{i_1}\times_{Y}\cdots\times_{Y}Y'_{i_l}
\end{eqnarray*}
where $s.e.c$ stands for standard etale covers (i.e $r_i$ are etale, $Y'_i$ are affine and $r_i:Y'_i\to r_i(Y'_i)$ finite etale).
For each $r=(r_i:V^{\mathcal O}\to X^{\mathcal O})_{i\in I}$ in $\mathcal H(\pi)\subset X^{\mathcal O,et}$ choose 
$e_r=(e_i:U^{\mathcal O}_i(r)\to X^{\mathcal O})_{1\leq i\leq s}$ be an etale cover (finite) in $\mathcal H(\pi)\subset X^{\mathcal O,et}$
such that for each $1\leq i\leq s$, $U_i^{\mathcal O}(r)\in\Sch^{int,sm}/O_K$ is affine 
whose (co)tangent bundle $\Omega^1_{U_i^{\mathcal O}}(r)\in\Shv_O(U_i^{\mathcal O}(r))$ is trivial and $e_i^*\Br(X^{\mathcal O}/p)[p]=0$. 
and such that no irreducible components of $X^{\mathcal O}\backslash e_i(U_i^{\mathcal O}(r))$
are contained in $X^{\mathcal O}\backslash r_i(V_i^{\mathcal O})$ for all $i\in I$.
Similarly, we have 
\begin{eqnarray*}
T(E,\varprojlim)(X^{\mathcal O}):
H^q_{et}(X^{\mathcal O},\Omega^p_{\hat X^{\mathcal O,(p)}})=
\varinjlim_{r=(r_i:V^{\mathcal O}\to X^{\mathcal O})_{i\in I}\mbox{s.e.c}, r\in\mathcal H(\pi)}
H^q\Gamma(V^{\mathcal O}_{\bullet}\times_{X^{\mathcal O}}U_{\bullet}^{\mathcal O}(r),\Omega^p_{\hat X^{\mathcal O,(p)}}) \\
\to H^q_{pet}(X^{\mathcal O},\Omega^p_{\hat X^{\mathcal O,(p)}})=
\varprojlim_{n\in\mathbb N}\varinjlim_{r=(r_i:V^{\mathcal O}_i\to X^{\mathcal O})_{i\in I}\mbox{s.e.c}, r\in\mathcal H(\pi)}
H^q\Gamma(V^{\mathcal O}_{\bullet}\times_{X^{\mathcal O}}U_{\bullet}^{\mathcal O}(r),\Omega^p_{X^{\mathcal O,et}/p^n})
\end{eqnarray*}
whose maps are isomorphisms.
But we have then the splitting (\ref{splitO1})
\begin{eqnarray*}
T(E,\varprojlim)(X^{\mathcal O})=(T(E,\varprojlim)(X^{\mathcal O}),T(E,\varprojlim)(X^{\mathcal O})): \\
\varinjlim_{r=(r_i:V_i^{\mathcal O}\to X^{\mathcal O})_{i\in I},r\in\mathcal H(\pi)}
H^q\Omega^p_{\hat X^{\mathcal O,(p)}}(V^{\mathcal O}_{\bullet}\times_{X^{\mathcal O}}U_{\bullet}^{\mathcal O}(r))= \\
\varinjlim_{r=(r_i:V_i^{\mathcal O}\to X^{\mathcal O})_{i\in I}\mbox{s.e.c}, r\in\mathcal H(\pi)}
H^qa_{et}\Omega^p_{\hat X^{\mathcal O,(p)}}(V^{\mathcal O}_{\bullet})
\oplus
\varinjlim_{r=(r_i:V_i^{\mathcal O}\to X^{\mathcal O})_{i\in I}\mbox{s.e.c}, r\in\mathcal H(\pi)}\varprojlim_{n\in\mathbb N}
\Omega^p_{n,U_{\bullet}^{\mathcal O}(r),nl}(V^{\mathcal O}_{\bullet}) \\
\to  
\varprojlim_{n\in\mathbb N}\varinjlim_{r=(r_i:V_i^{\mathcal O}\to X^{\mathcal O})_{i\in I}\mbox{s.e.c}, r\in\mathcal H(\pi)}
H^q\Omega^p_{X^{\mathcal O,et}/p^n}(V^{\mathcal O}_{\bullet}\times_{X^{\mathcal O}}U_{\bullet}^{\mathcal O}(r))= \\
\varprojlim_{n\in\mathbb N}\varinjlim_{r=(r_i:V_i^{\mathcal O}\to X^{\mathcal O})_{i\in I}\mbox{s.e.c}, r\in\mathcal H(\pi)}
H^qa_{et}\Omega^p_{X^{\mathcal O}/p^n,\log}(V^{\mathcal O}_{\bullet}) \\
\oplus 
\varprojlim_{n\in\mathbb N}\varinjlim_{r=(r_i:V^{\mathcal O}\to X^{\mathcal O})_{i\in I}\mbox{s.e.c}, r\in\mathcal H(\pi)}
H^q\Omega^p_{n,U_{\bullet}^{\mathcal O}(r),nl}(V^{\mathcal O}_{\bullet}),  
\end{eqnarray*}
This proves (i).

\noindent(ii): By proposition \ref{oetetprop}
\begin{eqnarray*}
S(X^{\mathcal O}/X):H^q_{et}(X^{\mathcal O},\Omega^p_{\hat X^{\mathcal O,(p)},\log,\mathcal O})\to H^q_{et}(X^{\mathcal O},r_*r^*\Omega^{p,an}_{/K,\log,\mathcal O}).
\end{eqnarray*}
Since 
\begin{eqnarray*}
S(X^{\mathcal O}/X):H_{et}^q(X^{\mathcal O},\Omega^p_{\hat X^{\mathcal O,(p)},\log,\mathcal O})=
\varinjlim_{(r_i:X^{\mathcal O}_i\to X^{\mathcal O})_{i\in I}}H^q\Omega^p_{\hat X^{\mathcal O,(p)},\log,\mathcal O}(X^{\mathcal O}_{\bullet})\to , \\
H_{et}^q(X^{\mathcal O},r_*r^*\Omega^{p,an}_{/K,\log,\mathcal O})=
\varinjlim_{(r_i:X^{\mathcal O}\to X^{\mathcal O})_{i\in I}}H^qr_*r^*\Omega^{p,an}_{/K,\log,\mathcal O}(X^{\mathcal O}_{\bullet}) \\
(r_i:X^{\mathcal O}_i\to X^{\mathcal O})_{i\in I} \, \mbox{s.e.c}, \; X^{\mathcal O}_{\bullet}\in\Fun(\Delta(I),\Sch^{int,sm}/O_K), \; 
X^{\mathcal O}_{i_1,\cdots,i_l}:=X^{\mathcal O}_{i_1}\times_{X^{\mathcal O}}\cdots\times_{X^{\mathcal O}}X^{\mathcal O}_{i_l}
\end{eqnarray*}
and since inductive colimit commutes with finite limit, to prove (ii2) it suffice to show that for each $(r_i:X^{\mathcal O}_i\to X^{\mathcal O})_{i\in I}$ a standard etale cover
\begin{eqnarray*}
S(X^{\mathcal O}/X):H^q\Omega^p_{\hat X^{\mathcal O,(p)},\log,\mathcal O}(X^{\mathcal O}_{\bullet})\to H^qr_*r^*\Omega^{p,an}_{/K,\log,\mathcal O}(X^{\mathcal O}_{\bullet})
\end{eqnarray*}
is injective. But this follows from the fact that for $V^{\mathcal O}\in\Sch^{int,sm}/O_K$, 
\begin{equation*}
r_*r^*\Omega^{p,an}_{/K,\log,\mathcal O}(V^{\mathcal O})=\varinjlim_{V^{\mathcal O(l)}\in IntSm(V)}\Omega^{p,an}_{/K,\log,\mathcal O}(V^{\mathcal O(l)})
\end{equation*}
where $V:=V^{\mathcal O}\otimes_{O_K}K$ and $IntSm(V)$ is the set of smooth integral models of $V$.
Indeed for 
\begin{equation*}
\alpha=\partial_{\bullet}(\beta)\in\oplus_{J\subset I, card J=m}\Omega^{p,an}_{/K,\log,\mathcal O}(V_J^{\mathcal O}) \; \mbox{and} \; 
\beta=(\beta_l)\in\oplus_{K\subset I, card K=m-1}r_*r^*\Omega^{p,an}_{/K,\log,\mathcal O}(V_K^{\mathcal O}), 
\end{equation*}
either there is no map between $V_J^{\mathcal O}$ and $V_K^{\mathcal O(l)}$ extending $V_J\to V_K$ in codimension one in the special fiber for $J\subset K$
and then $\partial_{J/K}(\beta_l)=0$, 
or there is a map between $V_J^{\mathcal O}$ and $V_K^{\mathcal O(l)}$ extending $V_J\to V_K$ in codimension one in the special fiber
and then after taking compactification and normalization 
there is a map $\rho_l$ between $V_J^{\mathcal O}$ and $V_K^{\mathcal O(l)}$ extending $V_J\to V_K$ in codimension two in the special fiber 
and then we may replace $\beta_l$ by $\rho_{l*}\rho_l^*\beta_l:=(1/\deg\rho_l)\Tr(\rho_l)\omega(\rho_l)(\beta_l)$.
\end{proof}

We will use the following result from Illusie and Bhatt-Scholze:

\begin{prop}\label{Ilprop}
Let $K$ be a $p$ adic field. Let $X\in\SmVar(K)$ with good reduction.  
Consider $X^{\mathcal O}\in\Sch^{int,sm}/O_K$ a smooth integral model of $X$,
in particular $X^{\mathcal O}\otimes_{O_K}K=X$ and $X^{\mathcal O}$ is smooth with smooth special fiber.
Assume there exist lifts $\phi^{\mathcal O}:X^{\mathcal O}\to X^{\mathcal O}$ of the Frobenius 
$\phi_1:X^{\mathcal O}/p\to X^{\mathcal O}/p$. In particular for $n'>n$ the following diagram commutes
\begin{equation*}
\xymatrix{0\ar[r] & O_{X^{\mathcal O}/p^{n'-n}}\ar[r]^{p^n\cdot} & 
O_{X^{\mathcal O}/p^{n'}}\ar[r]^{/p^{n'-n}} & O_{X^{\mathcal O}/p^n}\ar[r] & 0 \\
0\ar[r] & O_{X^{\mathcal O}/p^{n'-n}}\ar[r]^{p^n\cdot}\ar[u]^{\phi_{n'-n}} & 
O_{X^{\mathcal O}/p^{n'}}\ar[r]^{/p^{n'-n}}\ar[u]^{\phi_{n'}} & O_{X^{\mathcal O}/p^n}\ar[u]^{\phi_n}\ar[r] & 0},
\end{equation*}
and the canonical map $JW:=(JW_n:O(X^{\mathcal O}/p^n)\to W_n(X^{\mathcal O}/p))_{n\in\mathbb N}$ commutes with the (divided) Frobenius, 
that is $\phi_n\circ JW_n=JW_n\circ\phi_n$.
We consider the map in $C(X^{\mathcal O,pet})$, 
\begin{equation*}
IW:=(IW_n)_{n\in\mathbb N}:\Omega^{\bullet}_{\hat X^{\mathcal O}}\xrightarrow{(\Omega(JW_n))_{n\in\mathbb N}}
\Omega^{\bullet}_{W(X^{\mathcal O}/p)}\to W\Omega^{\bullet}_{X^{\mathcal O}}:=W\Omega^{\bullet}_{X^{\mathcal O}/p}.
\end{equation*}
We get in particular a factorization in $C(X^{\mathcal O,pet})$
\begin{eqnarray*}
IW\circ OL_X:\Omega^{\bullet}_{\hat X^{\mathcal O,(p)},\log,\mathcal O}\xrightarrow{IW_{\log}=(IW_{n,\log})_{n\in\mathbb N}}
W\Omega^{\bullet}_{X^{\mathcal O},\log}:=\ker(\phi=(\phi_n)_{n\in\mathbb N}:W\Omega^{\bullet}_{X^{\mathcal O}}\to W\Omega^{\bullet}_{X^{\mathcal O}})
\end{eqnarray*}
Then,
\begin{itemize}
\item[(i)] for each $k\in\mathbb Z$ the map
\begin{equation*}
H^kIW:\mathbb H_{pet}^k(X,\Omega^{\bullet}_{\hat X^{\mathcal O}})=\mathbb H_{et}^k(X,\Omega^{\bullet}_{\hat X^{\mathcal O}})
\to\mathbb H_{pet}^k(X,r^*W\Omega^{\bullet}_{X^{\mathcal O}/p}).
\end{equation*}
is an isomorphism,
\item[(ii)] for each $n,p\in\mathbb N$, the map in $\PSh(X^{\mathcal O,et})$
\begin{equation*}
IW_{n,\log}:a_{et}\Omega^p_{X^{\mathcal O}/p^n,\log}\to a_{et}W_n\Omega^{\bullet}_{X^{\mathcal O},\log}
\end{equation*}
is an isomorphism, where $a_{et}:\PSh(X^{\mathcal O,et})\to\Shv_{et}(X^{\mathcal O,et})$ is the sheaftification functor,
and for each $p\in\mathbb N$, the map in $\PSh(X^{\mathcal O,pet})$
\begin{equation*}
IW_{\log}:a_{pet}\Omega^{\bullet}_{\hat X^{\mathcal O,(p)},\log,\mathcal O}\to a_{pet}W\Omega^{\bullet}_{X^{\mathcal O},\log}
\end{equation*}
is an isomorphism, where $a_{pet}:\PSh(X^{\mathcal O,pet})\to\Shv_{et}(X^{\mathcal O,pet})$ is the sheaftification functor,
\item[(iii)] the sequence in $C(X^{\mathcal O,pet})$
\begin{equation*}
0\to a_{pet}W\Omega^{\bullet}_{X^{\mathcal O},\log}\xrightarrow{WOL_X,\subset}a_{pet}W\Omega^{\bullet}_{X^{\mathcal O}}
\xrightarrow{\phi}a_{pet}W\Omega^{\bullet}_{X^{\mathcal O}}\to 0
\end{equation*}
is exact for the pro-etale topology, i.e. $\phi$ is surjective for the pro-etale topology (as a map of pro etale sheaves).
\end{itemize}
\end{prop}

\begin{proof}
\noindent(i):Follows from \cite{Illusie}

\noindent(ii)
It follows from \cite{Illusie}: for $n=1$ its by definition, 
for $n\geq 2$ it follows by induction on $n$ by the commutative diagram
\begin{equation*}
\xymatrix{0\ar[r] & a_{et}\Omega^p_{X^{\mathcal O}/p^{n-1},\log}\ar[r]^{p\cdot (-)}\ar[d]^{IW_{n-1,\log}} & 
a_{et}\Omega^p_{X^{\mathcal O}/p^n,\log}\ar[r]^{/p^{n-1}}\ar[d]^{IW_{n,\log}} &
a_{et}\Omega^p_{X^{\mathcal O}/p,\log}\ar[d]^{IW_{1,\log}}\ar[r] & 0 \\
0\ar[r] & a_{et}W_{n-1}\Omega^p_{X^{\mathcal O}/p,\log}a_{et}\ar[r]^{V} & 
a_{et}W_{n-1}\Omega^p_{X^{\mathcal O}/p,\log}\ar[r]^{R^{n-1}} & a_{et}W_{n-1}\Omega^p_{X^{\mathcal O}/p,\log}\ar[r] & 0}
\end{equation*}
whose rows are exact sequences in $\Shv_{et}(X^{\mathcal O,et})$.
The fact that the lower row is exact follows from the exact sequence in $\Shv_{et}(X^{\mathcal O,et})$.
\begin{equation*}
0\to a_{et}(W_{n-1}\Omega^p_{X^{\mathcal O}/p}+\partial W_n\Omega^{p-1}_{X^{\mathcal O}/p})\xrightarrow{V}a_{et}W_n\Omega^p_{X^{\mathcal O}/p}
\xrightarrow{R^{n-1}}a_{et}\Omega^p_{X^{\mathcal O}/p}\to 0
\end{equation*}

\noindent(iii):Follows from  \cite{Illusie} and the standard results of Bhatt-Scholze.
\end{proof}

\section{De Rham logarithmic classes}

\subsection{The De Rham classes of algebraic cycles vs De Rham logarithmic classes}

Let $X\in\Sch$. Recall we have the canonical sub-complex 
$OL_X:\Omega^{\bullet}_{X,\log}\hookrightarrow\Omega_X^{\bullet}$ in $C(X^{et})$ (c.f. definition \ref{wlogdef}).
All the differential of $\Omega^{\bullet}_{X,\log}$ vanishes since by definition the logarithmic forms are closed.
For $j\in\mathbb Z$, the De Rham logarithmic classes consist of the image
\begin{equation*}
H^jOL_X(\mathbb H_{et}^j(X,\Omega^{\bullet}_{X,\log}))\subset\mathbb H_{et}^j(X,\Omega_X^{\bullet})=H^j_{DR}(X).
\end{equation*}
The differentials of the filtered complex $\Gamma(X,E_{et}(\Omega^{\bullet}_{X^{et},\log},F_b))\in C_{fil}(\mathbb Z)$
vanishes at the $E_1$ level since the logarithmic forms are closed,
hence we have a canonical splitting
\begin{equation*}
\mathbb H_{et}^j(X,\Omega^{\bullet}_{X,\log})=\oplus_{0\leq l\leq j}H_{et}^{j-l}(X,\Omega^l_{X,\log}).
\end{equation*}
Let $X\in\Sch$ a noetherian scheme. We have by definition the exact sequence in $C(X^{et})$
\begin{equation*}
0\to F^*\to O_X^*\xrightarrow{dlog}\Omega_{X,\log}\to 0
\end{equation*}
where $F$ is a prime field (i.e. $F=\mathbb Q$ or $F=\mathbb Z/p\mathbb Z$ for $p$ a prime number). 
Hence $H_{et}^1(X,\Omega_{X,\log})=H^1_{et}(X,O_X^*)$ and 
$H_{et}^q(X,\Omega_{X,\log})=H^q_{et}(X,O_X^*)=0$ for $q\leq 2$.
Let $X\in\Sch$ a noetherian proper scheme. We have $H^0(X,\Omega^l_{X,\log})=0$.

Let $k$ be a field of characteristic zero. We get (see definition \ref{wlogdef}) the embedding in $C(\SmVar(k))$
\begin{eqnarray*}
OL:\Omega^{\bullet}_{/k,\log}\hookrightarrow\Omega^{\bullet}_{/k}, \; 
\mbox{given by}, \, \mbox{for} \, X\in\SmVar(k), \\
OL(X):=OL_X:\Omega^{\bullet}_{/k,\log}(X):=\Gamma(X,\Omega^{\bullet}_{X,\log})
\hookrightarrow\Gamma(X,\Omega^{\bullet}_X)=:\Omega^{\bullet}_{/k}(X).
\end{eqnarray*}
We have also the sheaves
\begin{eqnarray*}
O_k,O_k^*\in\PSh(\SmVar(k)), \; X\in\SmVar(k)\mapsto O_k(X):=O(X),\, O_k^*(X):=O(X)^*, \\ 
(g:Y\to X)\mapsto a_g(X):O_X(X)\to O_Y(Y), \, O_X(X)^*\to O_Y(Y)^*
\end{eqnarray*}

\begin{lem}\label{a1tr}
\begin{itemize}
\item[(i0)]The sheaves $O_k^*\in\PSh(\SmVar(k))$ and $\Omega^1_{/k,\log}\in\PSh(\SmVar(k))$ admit transfers
compatible with transfers on $\Omega^1_{/k}\in\PSh(\SmVar(k))$.
\item[(i)] For each $l\in\mathbb Z$, the sheaf $\Omega^l_{/k,\log}\in\PSh(\SmVar(k))$ admits transfers
compatible with transfers on $\Omega^l_{/k}\in\PSh(\SmVar(k))$,
that is $\Omega^{\bullet}_{/k}\in C(\Cor\SmVar(k))$ (see \cite{LW}) and the inclusion 
$OL:\Omega^l_{/k,\log}[-l]\hookrightarrow\Omega^{\bullet}_{/k}$ in $C(\SmVar(k))$ 
is compatible with transfers.
\item[(ii)]For each $l\in\mathbb Z$, the sheaf $\Omega^l_{/k,\log}\in\PSh(\SmVar(k))$ is $\mathbb A^1$ invariant.
\end{itemize}
\end{lem}

\begin{proof}
\noindent(i0):The sheaf $O_k^*\in\PSh(\SmVar(k))$ admits transfers : for $W\subset X'\times X$ with $X,X'\in\SmVar(k)$
and $W$ finite over $X'$ and $f\in O(X)^*$, $W^*f:=N_{W/X'}(p_X^*f)$ where $p_X:W\hookrightarrow X'\times X\to X$
is the projection and $N_{W/X'}:k(W)^*\to k(X')^*$ is the norm map.
This gives transfers on $\Omega^1_{/k,\log}\in\PSh(\SmVar(k))$ compatible with transfers on 
$\Omega^1_{/k}\in\PSh(\SmVar(k))$ :
for $W\subset X'\times X$ with $X,X'\in\SmVar(k)$ and $W$ finite over $X'$ and $f\in O(X)^*$, 
\begin{equation*}
W^*df/f:=dW^*f/W^*f=Tr_{W/X'}(p_X^*(df/f)), 
\end{equation*}
where where $p_X:W\hookrightarrow X'\times X\to X$ is the projection and $Tr_{W/X'}:O_W\to O_X$ is the trace map. 
Note that $d(fg)/fg=df/f+dg/g$. 

\noindent(i): By (i0), we get transfers on 
\begin{equation*}
\otimes^l_{\mathbb Q}\Omega^1_{/k,\log}, \, \otimes_{O_k}^l\Omega^1_{/k}\in\PSh(\SmVar(k))
\end{equation*}
since $\otimes^l_{\mathbb Q}\Omega^1_{/k,\log}=H^0(\otimes^{L,l}_{\mathbb Q}\Omega^1_{/k,\log})$ and 
$\otimes_{O_k}^l\Omega^1_{/k}=H^0(\otimes_{O_k}^{L,l}\Omega^1_{/k})$.
This induces transfers on  
\begin{equation*}
\wedge^l_{\mathbb Q}\Omega^1_{/k,\log}:=
\coker(\oplus_{I_2\subset[1,\ldots,l]}\otimes^{l-1}_{\mathbb Q}\Omega^1_{/k,\log}
\xrightarrow{\oplus_{I_2\subset[1,\ldots,l]}\Delta_{I_2}:=(w\otimes w'\mapsto w\otimes w\otimes w')}
\otimes^l_{\mathbb Q}\Omega^1_{/k,\log})
\in\PSh(\SmVar(k)).
\end{equation*}
and
\begin{equation*}
\wedge^l_{O_k}\Omega^1_{/k}:=
\coker(\oplus_{I_2\subset[1,\ldots,l]}\otimes^{l-1}_{O_k}\Omega^1_{/k}
\xrightarrow{\oplus_{I_2\subset[1,\ldots,l]}\Delta_{I_2}:=(w\otimes w'\mapsto w\otimes w\otimes w')})
\otimes^l_{O_k}\Omega^1_{/k}
\in\PSh(\SmVar(k)).
\end{equation*}

\noindent(ii): Follows from the fact that for $X\in\Var(k)$, $O^*(X\times\mathbb A^1)=O^*(X)$
since for a commutative ring $A$, $(A[X])^*=A^*$.
\end{proof}

Let $X\in\Var(k)$. We have by definition the exact sequence in $C(X^{et})$
\begin{equation*}
0\to k^*\to O_X^*\xrightarrow{dlog}\Omega_{X,\log}\to 0
\end{equation*}
Hence $H_{et}^1(X,\Omega_{X,\log})=H^1_{et}(X,O_X^*)$ and 
$H_{et}^q(X,\Omega_{X,\log})=H^q_{et}(X,O_X^*)=0$ for $q\leq 2$.
For $X\in\PVar(k)$, we have $H^0(X,\Omega^l_{X,\log})=0$.

\begin{prop}\label{smXZlog}
For each $j\in\mathbb Z$, the excision isomorphism induced by (i)
\begin{eqnarray*}
H^j\Omega(P_{Z,X}):=H^jE_{et}(\Omega_{/k}^{\bullet})(\mathbb Z(i_0))\circ 
H^jE_{et}(\Omega_{/k}^{\bullet})(\mathbb Z(i_1))^{-1}: \\
H^j_{DR,Z}(X):=H^j\Gamma_Z(X,E_{et}(\Omega_X^{\bullet}))\xrightarrow{\sim}
H^j_{DR,Z}(N_{Z/X}):=H^j\Gamma_Z(N_{Z/X},E_{et}(\Omega_{N_{Z/X}}^{\bullet}))
\end{eqnarray*}
preserve logarithmic De Rham classes, that is for each $l\in\mathbb Z$,
\begin{equation*}
H^j\Omega(P_{Z,X})(H^jOL_X(H_{et,Z}^{j-l}(X,\Omega^l_{X,\log})))=
H^jOL_{N_{Z/X}}(H_{et,Z}^{j-l}(N_{Z/X},\Omega^l_{N_{Z/X},\log}))
\end{equation*}
\end{prop}

\begin{proof}
Since $\Omega^l_{/k,\log}\in\PSh(\SmVar(k))$ is $\mathbb A^1_k$ invariant 
and admits transfers by lemma \ref{a1tr}, for each $j\in\mathbb Z$,
the presheaves $H^{j-l}E_{et}(\Omega^l_{/k,\log})\in\PSh(\SmVar(k))$ are $\mathbb A^1_k$ invariant
by a theorem \ref{Voethm}.
It then follows from proposition \ref{smXZ} that 
\begin{eqnarray*}
H^j\Omega_{\log}(P_{Z,X}):=H^jE_{et}(\Omega_{/k,\log}^{\bullet})(\mathbb Z(i_0))\circ 
H^jE_{et}(\Omega_{/k,\log}^{\bullet})(\mathbb Z(i_1))^{-1}: \\
H_{et,Z}^{j-l}(X,\Omega^l_{X,\log})\xrightarrow{\sim}H_{et,Z}^{j-l}(N_{Z/X},\Omega^l_{N_{Z/X},\log}).
\end{eqnarray*}
Hence,
\begin{eqnarray*}
H^j\Omega(P_{Z,X})(H^jOL_X(H_{et,Z}^{j-l}(X,\Omega^l_{X,\log})))
&=& H^jOL_{N_{Z/X}}(\Omega_{\log}(P_{Z,X})(H_{et,Z}^{j-l}(X,\Omega^l_{X,\log}))) \\
&=& H^jOL_{N_{Z/X}}(H_{et,Z}^{j-l}(N_{Z,X},\Omega^l_{N_{Z/X},\log})).
\end{eqnarray*}
\end{proof}

\begin{rem}
Let $k$ be a field a characteristic zero.
\begin{itemize}
\item[(i)] The statement of proposition \ref{smXZlog} also holds for the subcomplex 
$\Omega^{\bullet,\partial=0}_{/k}\hookrightarrow\Omega^{\bullet}_{/k}$ of closed forms by the same argument
since the inclusion 
$\Omega^{l,\partial=0}_{/k}[-l]\hookrightarrow\Omega^{\bullet}_{/k}$ in $C(\SmVar(k))$ 
is compatible with transfers.
\item[(ii)] The statement of proposition \ref{smXZlog} does NOT hold for presheaves which do not admits transfers
(they are $\mathbb A^1$ invariant but not $\mathbb A^1$ local in general).
Note that the result of proposition \ref{smXZlog} does NOT hold for the embedding in $C(\SmVar(k))$,
associated to an embedding $\sigma:k\hookrightarrow\mathbb C$,
\begin{equation*}
\mathcal H^{l,j-l}[-l]\hookrightarrow\mathcal A^{\bullet}, \, X\in\SmVar(k), \, 
\Gamma(X,\mathcal H_{X_{\mathbb C}^{an}}^{l,j-l})[-l]\hookrightarrow\Gamma(X,\mathcal A_{X_{\mathbb C}^{an}}^{\bullet})
\end{equation*}
of the subsheaf of harmonic differential forms, the sheaves of differential forms 
$\mathcal A\in C(\SmVar(k))$ does NOT admits transfers (finite algebraic correspondences are not smooth and 
for $f:X'\to X$ a morphism with $X',X\in\Diff(\mathbb R)$ vector fields on $X$ only lift if $f$ is a smooth morphism
i.e. if the differential of $df$ is surjective), transfers maps are only defined on cohomology.
Recall that we do NOT have the Hodge decomposition for open complex varieties.
\end{itemize}
\end{rem}

Let $k$ be a field of characteristic zero.
Let $X\in\SmVar(k)$ and $Z\subset X$ a smooth subvariety of pure codimension d.
We have for each $j\in\mathbb Z$, the purity isomorphism given by $H^j\Omega(P_{Z,X})$ (see proposition \ref{smXZ})
and the cup product with the Euler class of of the normal tangent bundle $N_{Z/X}\to Z$ :
\begin{equation*}
H^jDR(P_{Z,X}):H^j_{DR,Z}(X)\xrightarrow{H^j\Omega(P_{Z,X})}H^j_{DR,Z}(N_{Z/X})
\xrightarrow{((-)\cdot e(N_{Z/X}))^{-1}}H^{j-2d}_{DR}(Z).
\end{equation*}

Now, we have the following :

\begin{prop}\label{drlogZprop}
Let $k$ be a field of characteristic zero. 
Let $p:E\to X$ a vector bundle of rank $d\in\mathbb N$ with $X,E\in\SmVar(k)$ connected. Then 
\begin{itemize}
\item[(i)]the Euler class $e(E)\in H^{2d}_{DR,X}(E)$ is logarithmic of type $(d,d)$, that is
$e(E)=H^{2d}OL_E(e(E))\in H^{2d}_{DR,X}(E)$ with $e(E)\in H_{X,et}^d(E,\Omega^d_{E,\log})$,
\item[(ii)] the Euler class $e(E)\in H^{2d}_{DR,X}(E)$ induces for each $i,j\in\mathbb Z$ an isomorphism
\begin{equation*}
((-)\cdot e(E)):H^{i+j}OL_X(H_{et}^j(X,\Omega_{X,\log}^i))\xrightarrow{\sim}
H^{2d+i+j}OL_E(H_{et}^{j+d}(E,\Omega_{E,\log}^{i+d})).
\end{equation*}
\end{itemize}
\end{prop}

\begin{proof}
\noindent(i):Let $X=\cup_iX_i$ an open affine cover such that $E_i:=E_{|X_i}$ is trivial :
$(s_1,\cdots,s_d):E_{|X_i}\xrightarrow{\sim}X_i\times\mathbb A^d$ with $s_j\in\Gamma(X_i,E)$. Then 
\begin{equation*}
e(E)_{|X_i}=\partial(ds_1/s_1\wedge\cdots\wedge ds_d/s_d)\in H^{2d}OL_{E_i}(H^d_{X_i}(E_i,\Omega^d_{E,\log}))
\end{equation*}
where $\partial:H^{d-1}(E_i\backslash X_i,\Omega^d_{E,\log})\to H^d_{X_i}(E_i,\Omega^d_{E,\log})$ is the boundary map. Hence
\begin{equation*}
e(E)\in H^{2d}OL_E(H^0(E,H_X^d\Omega^d_{E,\log}))=H_{et,X}^d(E,\Omega^d_{E,\log})
\subset H^0(E,H^{2d}_X\Omega_E^{\bullet\geq d})\subset H_{DR,X}^{2d}(E).
\end{equation*}

\noindent(ii):Follows from (i) and Kunneth formula for De Rham cohomology : 
let $X=\cup_iX_i$ an open affine cover such that $E_i:=E_{|X_i}$ is trivial and consider the morphism
of bi-complexes of abelian groups
\begin{equation*}
((-)\cdot e(E)):OL_X(\Gamma(X_{\bullet},E_{et}(\Omega_{X,\log}^i)))\to
OL_E(\Gamma(E_{\bullet},E_{et}(\Omega_{E,\log}^{i+d})))[2d].
\end{equation*}
By $(i)$ an Kunneth formula for De Rham cohomology, $((-)\cdot e(E))$ induces an isomorphism on the $E_1$ level
of the spectral sequences for the filtrations associated to bi-complex structures on the total complexes.
Hence, $((-)\cdot e(E))$ is a quasi-isomorphism. This proves (ii).
\end{proof}

We deduce from proposition \ref{smXZlog} and proposition \ref{drlogZprop}, the following key proposition

\begin{prop}\label{drlogZcor}
Let $k$ be a field of characteristic zero.
Let $X\in\SmVar(k)$ and $Z\subset X$ a smooth subvariety of codimension d.
For each $j\in\mathbb Z$, the purity isomorphism
\begin{equation*}
H^jDR(P_{Z,X}):H^j_{DR,Z}(X)\xrightarrow{H^j\Omega(P_{Z,X})}H^j_{DR,Z}(N_{Z/X})
\xrightarrow{((-)\cdot e(N_{Z/X}))^{-1}}H^{j-2d}_{DR}(Z)
\end{equation*}
preserve logarithmic De Rham classes, that is for each $l\in\mathbb Z$,
\begin{equation*}
H^jDR(P_{Z,X})(H^jOL_X(H_{et,Z}^{j-l}(X,\Omega^l_{X,\log})))=H^jOL_Z(H_{et}^{j-l}(Z,\Omega^l_{Z,\log})).
\end{equation*}
\end{prop}

\begin{proof}
Follows from proposition \ref{smXZlog} and proposition \ref{drlogZprop}(ii).
\end{proof}

Let $k$ be a field of characteristic zero. For $X\in\SmVar(k)$, we have the cycle class map for De Rham cohomology
\begin{eqnarray*}
Z\in\mathcal Z^d(X)\mapsto [Z]:=H^{2d}\Omega(\gamma^{\vee}_Z)([Z])\subset H_{DR}^{2d}(X), 
\; [Z]\in H^{2d}_{DR,|Z|}(X), \\ 
H^{2d}\Omega(\gamma^{\vee}_Z):H^{2d}_{DR,|Z|}(X)\to H^{2d}_{DR}(X)
\end{eqnarray*}
where
\begin{equation*}
\Omega(\gamma^{\vee}_Z):=\Hom(\gamma_Z^{\vee}(\mathbb Z_X),E_{et}(\Omega^{\bullet}_X)):
\Gamma_Z(X,E_{et}(\Omega^{\bullet}_X))\hookrightarrow\Gamma(X,E_{et}(\Omega^{\bullet}_X))
\end{equation*}
and, as for any Weil cohomology theory, we have the isomorphism given by purity : if
$X^o\subset X$ is an open subset such that $Z^o:=X^o\cap Z$ has smooth components we have the isomorphism
\begin{equation*}
H^{2d}DR(P_{Z^o,X^o}):H^{2d}_{DR,|Z|}(X)=H^{2d}_{DR,|Z^o|}(X^o)=H^{2d}_{DR,|Z^o|}(N_{Z^o/X^o})
\xrightarrow{\sim}H^0_{DR}(|Z^o|).
\end{equation*}
where the first equality follows from dimension reason : 
for $X\in\SmVar(k)$ and $Z'\subset X$ smooth, we have $H^i_{DR,Z'}(X)=0$ for $i<2\codim(Z',X)$
by the purity isomorphism $H^iDR(P_{Z',X})$.

The main result of this section is the following :

\begin{thm}\label{drlogZ}
Let $k$ be a field of characteristic zero. Let $X\in\SmVar(k)$. Let $d\in\mathbb N$.
\begin{itemize}
\item[(i)]The De Rham cohomology class of an algebraic cycle is logarithmic and is of type $(d,d)$,
that is, for $Z\in\mathcal Z^d(X)$  
\begin{equation*}
[Z]:=H^{2d}\Omega(\gamma^{\vee}_Z)([Z])\subset H^{2d}OL_X(H_{et}^d(X,\Omega^d_{X,\log}))\subset H_{DR}^{2d}(X).
\end{equation*}
\item[(ii)]Conversely, any $w\in H^{2d}OL_X(H_{et}^d(X,\Omega^d_{X,\log}))$ is the class of an algebraic cycle
$Z\in\mathcal Z^d(X)$, i.e. $w=[Z]$.
Note that it implies $w\in W_0H^{2d}_{DR}(X)$ as it is easily seen for $d=1$ by the Chow moving lemma
since $j^*:\Pic(\bar X)\to\Pic(X)$ is surjective if $j:X\hookrightarrow\bar X$ is a smooth compactification of $X$,
$\bar X\in\PSmVar(k)$ and since $H_{et}^1(Y,\Omega^1_{Y,\log})=\Pic(Y)$ for $Y\in\Var(k)$.
\item[(iii)] We have $H^jOL_X(H_{et}^{j-l}(X,\Omega^l_{X,\log}))=0$ for $j,l\in\mathbb Z$ such that $2l<j$.
\item[(iii)']If $X\in\PSmVar(k)$, we also have $H^jOL_X(H_{et}^{j-l}(X,\Omega^l_{X,\log}))=0$ 
for $j,l\in\mathbb Z$ such that $2l>j$. 
That is, if $X\in\PSmVar(k)$, we only have logarithmic classes in bidegree $(d,d)$ for $d\in\mathbb N$,
in particular there is no non trivial logarithmic classes for odd degree De Rham cohomology $H^{2d+1}_{DR}(X)$.
\end{itemize}
\end{thm}

\begin{proof}
\noindent(i):We have
\begin{equation*}
[Z]:=H^{2d}\Omega(\gamma^{\vee}_Z)([Z])=H^{2d}OL_X\circ H^d\Omega^d_{\log}(\gamma^{\vee}_Z)([Z])\subset H_{DR}^{2d}(X), \; 
[Z]\in H^d_{et,|Z|}(X,\Omega^d_{X,\log})
\end{equation*}
where
\begin{equation*}
\Omega^d_{\log}(\gamma^{\vee}_Z):=\Hom(\gamma_Z^{\vee}(\mathbb Z_X),E_{et}(\Omega^d_{X,\log})):
\Gamma_Z(X,E_{et}(\Omega^d_{X,\log}))\hookrightarrow\Gamma(X,E_{et}(\Omega^d_{X,\log})),
\end{equation*}
since if $X^o\subset X$ is a Zariski open subset such that $Z^o:=X^o\cap Z$ has smooth components 
and $N_{Z^o/X^o}\to Z^o$ is the normal tangent bundle, we have by proposition \ref{drlogZcor} 
the isomorphism
\begin{eqnarray*}
H^{2d}DR(P_{Z^o,X^o}):
H^{2d}OL_X(H^d_{et,|Z|}(X,\Omega^d_{X,\log}))&=&H^{2d}OL_X(H^d_{et,|Z^o|}(X^o,\Omega^d_{X^o,\log})) \\
&=&H^{2d}OL_X(H^d_{et,|Z^o|}(N_{Z^o/X^o},\Omega^d_{N_{Z^o/X^o},\log})).
\end{eqnarray*}

\noindent(ii):By assumption we have 
\begin{equation*}
w=H^{2d}OL_X(w)\in H^{2d}_{DR}(X), \; w\in H^d_{et}(X,\Omega^d_{X,\log})
\end{equation*}
As $\Omega^d_{X,\log}\in\PSh(X^{et})$ consist of a single presheaf, 
that is a complex of presheaves concentrated one degree, 
there exist an etale open cover $r=r(w)=(r_i:X_i\to X)_{1\leq i\leq s}$ depending on $w$,
such that $r_i^*w=0\in H^d_{et}(X_i,\Omega^d_{X,\log})$ for each $i$. 
Choose $i=1$ and denote $j:U:=r_1(X_1)\hookrightarrow X$ the corresponding (Zariski) open subset. 
As $\Omega^d_{X,\log}$ has no torsion and admits transfers, we have $j^*w=0\in H^d_{et}(U,\Omega^d_{X,\log})$.
Hence, denoting $D:=X\backslash U$, we have 
\begin{equation*}
w=H^{2d}OL_X(w), w\in H^d\Omega_{\log}^d(\gamma^{\vee}_D)(H^d_{et,D}(X,\Omega^d_{X,\log})).
\end{equation*}
We may assume, up to shrinking $U$, that $D\subset X$ is a divisor. 
Denote $D^o\subset D$ its smooth locus and $l:X^o\hookrightarrow X$ a Zariski open subset such that $X^o\cap D=D^o$.
We then have by proposition \ref{drlogZcor}
\begin{equation*}
l^*w\in H^{2d}DR(P_{D^o,X^o})(H^{2d-2}OL_{D^o}(H^{d-1}_{et}(D^o,\Omega^{d-1}_{D^o,\log}))). 
\end{equation*}
We repeat this procedure with each connected components of $D^o$ instead of $X$.
By a finite induction of $d$ steps, we get 
\begin{equation*}
w=H^{2d}OL_X(w), \; w\in H^d\Omega_{\log}^d(\gamma^{\vee}_Z)(H^d_{et,Z}(X,\Omega^d_{X,\log})),
\end{equation*}
with $Z:=D_d\subset\cdots\subset D\subset X$ a pure codimension $d$ (Zariski) closed subset and 
$w=\sum_in_i[Z_i]\in H^{2d}_{DR}(X)$, where $n_i\in\mathbb Z$ and $(Z_i)_{1\leq i\leq t}\subset Z$
are the irreducible components of $Z$,
using in the final step the isomorphism 
\begin{equation*}
H^1_{et}(D_{d-1}^o,\Omega_{D_{d-1}^o,\log})=H^1_{et}(D_{d-1}^o,O_{D_{d-1}}^*)=Pic(D_{d-1}^o).
\end{equation*}

\noindent(iii):Let $j>2l$. Let $w\in H^jOL_X(H_{et}^{j-l}(X,\Omega^l_{X,\log}))$.
By the proof of (ii) there exists $Z\subset X$ a closed subset of pure codimension $l$
such that
\begin{equation*}
w=H^jOL_X(w), \; w\in H^{j-l}\Omega_{\log}^l(\gamma^{\vee}_Z)(H^{j-l}_{et,Z}(X,\Omega^l_{X,\log})),
\end{equation*}
By a finite induction of $d_X-l$ steps, restricting to the smooth locus of closed subsets of $Z$, $w=0$ since
by proposition \ref{smXZ}, $H^jOL_X(H_{et,Z'}^{j-l}(X,\Omega^l_{X,\log}))=0$ for all $Z'\in\SmVar(k)$
such that $\codim(Z',X)\geq l$.

\noindent(iii)':Let $j<2l$. Let $w\in H^jOL_X(H_{et}^{j-l}(X,\Omega^l_{X,\log}))$. 
By the proof of (ii) there exists $Z\subset X$ a closed subset of pure codimension $j-l$ such that
\begin{equation*}
w=(H^jOL_X\circ H^{j-l}\Omega_{\log}^l(\gamma^{\vee}_Z))(w), \; w\in H^{j-l}_{et,Z}(X,\Omega^l_{X,\log}),
\end{equation*}  
For $Z'\subset X$ a closed subset of pure codimension $c$,
consider a desingularisation $\epsilon:\tilde Z'\to Z'$ of $Z'$ and denote $n:\tilde Z'\xrightarrow{\epsilon}Z'\subset X$.
The morphism in $\DA(k)$
\begin{equation*}
G_{Z',X}:M(X)\xrightarrow{D(\mathbb Z(n))}M(\tilde Z')(c)[2c]\xrightarrow{\mathbb Z(\epsilon)}M(Z')(c)[2c]
\end{equation*}
where $D:\Hom_{\DA(k)}(M_c(\tilde Z'),M_c(X))\xrightarrow{\sim}\Hom_{\DA(k)}(M(X),M(\tilde Z')(c)[2c])$
is the duality isomorphism from the six functor formalism (moving lemma of Suzlin and Voevodsky)
and $\mathbb Z(n):=\ad(n_!,n^!)(a_X^!\mathbb Z)$, is given by a morphism in $C(\SmVar(k))$
\begin{equation*}
\hat G_{Z',X}:\mathbb Z_{tr}(X)\to E_{et}(C_*\mathbb Z_{tr}(Z'))(c)[2c].
\end{equation*}
Let $l:X^o\hookrightarrow X$ be an open embedding such that $Z^o:=Z\cap X^o$ is the smooth locus of $Z$.
We then have the following commutative diagram of abelian groups
\begin{equation*}
\xymatrix{0\ar[r] & H_{et,Z}^{j-l}(X,\Omega_{X,\log}^l)\ar[r]^{l^*} & 
H_{et,Z^o}^{j-l}(X^o,\Omega_{X,\log}^l)\ar[r]^{\partial} & 
H_{et,Z\backslash Z^o}^{j-l+1}(X,\Omega^l_{X,\log})\ar[r] & \cdots \\
0\ar[r] & H^0_{et}(Z,\Omega^{2l-j}_{Z,\log})\ar[r]^{l*}\ar[u]^{\Omega(\hat G_{Z,X})} & 
H^0_{et}(Z^o,\Omega^{2l-j}_{Z^o,\log})\ar[r]^{\partial}\ar[u]^{DR(P_{Z^o,X^o})} &
H_{et,Z\backslash Z^o}^1(Z,\Omega^{2l-j}_{Z,\log})\ar[u]^{\Omega(\hat G_{Z,X})}\ar[r] & \cdots}
\end{equation*}
whose rows are exact sequences. Consider 
\begin{equation*}
l^*w=DR(P_{Z^o,X^o})(w^o)\in H_{et,Z^o}^{j-l}(X^o,\Omega_{X,\log}^l), \; 
w^o\in H^0_{et}(Z^o,\Omega^{2l-j}_{Z^o,\log}).
\end{equation*}
Since $\partial l^*w=0\in H_{et,Z\backslash Z^o}^{j-l+1}(X,\Omega^l_{X,\log})$,
we get using proposition \ref{drlogZcor} applied to $((Z\backslash Z^o)^o,X^{oo})$,
where $(Z\backslash Z^o)^o\subset Z\backslash Z^o$ is the smooth locus and $X^{oo}\subset X$ is an open subset
such that $X\cap(Z\backslash Z^o)^o=X^{oo}\cap (Z\backslash Z^o)$, 
\begin{equation*}
\partial w^o=0\in H_{et,Z\backslash Z^o}^1(Z,\Omega^{2l-j}_{Z,\log}), 
\end{equation*}
since $H_{et,(Z\backslash Z^o)\backslash(Z\backslash Z^o)^o}^1(Z,\Omega^{2l-j}_{Z,\log})=0$
for dimension reasons, that is 
\begin{equation*}
w=\Omega(\hat G_{Z,X})(w), \, \mbox{with} \, w\in H^0_{et}(Z,\Omega^{2l-j}_{Z,\log}).
\end{equation*} 
Hence $w=0$ since $H^0(Z',\Omega^k_{Z',\log})=0$ for all $k>0$ and all $Z'\in\PVar(k)$.
\end{proof}

\subsection{Complex integral periods}

Let $k$ be a field of characteristic zero. 

Let $X\in\SmVar(k)$ a smooth variety. Let $X=\cup_{i=1}^sX_i$ an open affine cover. 
We have for $\sigma:k\hookrightarrow\mathbb C$ an embedding, the evaluation period embedding map
which is the morphism of bi-complexes 
\begin{eqnarray*}
ev(X)^{\bullet}_{\bullet}:\Gamma(X_{\bullet},\Omega^{\bullet}_{X_{\bullet}})\to 
\mathbb Z\Hom_{\Diff}(\mathbb I^{\bullet},X^{an}_{\mathbb C,\bullet})^{\vee}\otimes\mathbb C, \\
w^l_I\in\Gamma(X_I,\Omega^l_{X_I})\mapsto 
(ev(X)^l_I(w^l_I):\phi^l_I\in\mathbb Z\Hom_{\Diff}(\mathbb I^l,X^{an}_{\mathbb C,I})^{\vee}\otimes\mathbb C
\mapsto ev^l_I(w^l_I)(\phi^l_I):=\int_{\mathbb I^l}\phi_I^{l*}w^l_I)
\end{eqnarray*}
given by integration. By taking all the affine open cover $(j_i:X_i\hookrightarrow X)$ of $X$,
we get for $\sigma:k\hookrightarrow\mathbb C$, the evaluation period embedding map 
\begin{eqnarray*}
ev(X):=\varinjlim_{(j_i:X_i\hookrightarrow X)}ev(X)^{\bullet}_{\bullet}:
\varinjlim_{(j_i:X_i\hookrightarrow X)}\Gamma(X_{\bullet},\Omega^{\bullet}_{X_{\bullet}})
\to \varinjlim_{(j_i:X_i\hookrightarrow X)}
\mathbb Z\Hom_{\Diff(\mathbb R)}(\mathbb I^{\bullet},X^{an}_{\mathbb C,\bullet})^{\vee}\otimes\mathbb C
\end{eqnarray*}
It induces in cohomology, for $j\in\mathbb Z$, the evaluation period map
\begin{eqnarray*}
H^jev(X)=H^jev(X)^{\bullet}_{\bullet}:H^j_{DR}(X)=H^j\Gamma(X_{\bullet},\Omega^{\bullet}_{X_{\bullet}})\to 
H_{\sing}^j(X_{\mathbb C}^{an},\mathbb C)=
H^j(\Hom_{\Diff(\mathbb R)}(\mathbb I^{\bullet},X^{an}_{\mathbb C,\bullet})^{\vee}\otimes\mathbb C). 
\end{eqnarray*}
which does NOT depend on the choice of the affine open cover 
by acyclicity of quasi-coherent sheaves on affine noetherian schemes for the left hand side
and from the small chains quasi-isomorphism theorem for singular cohomology of topological spaces
and Whitney approximation theorem for differential manifolds for the right hand side.

\begin{prop}\label{keypropC}
Let $k$ be a field of characteristic zero. Let $X\in\SmVar(k)$. 
Let $\sigma:k\hookrightarrow\mathbb C$ an embedding.
Let $w\in H^j_{DR}(X):=\mathbb H^j(X,\Omega^{\bullet}_X)=\mathbb H_{et}^j(X,\Omega^{\bullet}_X)$. If 
\begin{equation*}
w\in H^jOL_X(\mathbb H_{et}^j(X_{\mathbb C},\Omega^{\bullet}_{X_{\mathbb C}^{et},\log}))
\end{equation*}
then $H^jev(X)(w)\in H^j_{\sing}(X_{\mathbb C}^{an},(2i\pi\mathbb Q)^j)$. 
Note that if $X$ is projective, this follows from theorem \ref{drlogZ}.
\end{prop}

\begin{proof} 
Let 
\begin{equation*}
w\in H^j_{DR}(X):=\mathbb H^j(X,\Omega^{\bullet}_X)=H^j\Gamma(X_{\bullet},\Omega^{\bullet}_X).
\end{equation*}
where $(r_i:X_i\to X)_{1\leq i\leq s}$ is an affine etale cover.
Let $X^{an}_{\mathbb C}=\cup_{i\in\Lambda}\mathbb D_i$ an open cover with $\mathbb D_i\simeq D(0,1)^d$. 
Denote $j_{IJ}:X_I\cap\mathbb D_J\hookrightarrow X_I$ the open embeddings.
Then by definition $H^jev(X)(w)=H^jev(X_{\mathbb C}^{an})(j_{\bullet}^*\circ\an_{X_{\bullet}}^*w)$ with
\begin{equation*}
j_{\bullet}^*\circ\an_{X_{\bullet}}^*w
\in H^j\Gamma(X^{an}_{\bullet,\mathbb C}\cap\mathbb D_{\bullet},\Omega^{\bullet}_{X_{\mathbb C}^{an}}). 
\end{equation*}
Now, if 
$w=H^jOL_X(\mathbb H_{et}^j(X_{\mathbb C},\Omega^{\bullet}_{X_{\mathbb C}^{et},\log}))$,
we have a canonical splitting
\begin{eqnarray*}
w=\sum_{l=0}^jw_L^{l,j-l}=\sum_{l=d}^jw^{l,j-l}\in H^j_{DR}(X_{\mathbb C}), \; 
w_L^{l,j-l}\in H^{j-l}(X_{\mathbb C},\Omega^l_{X^{et}_{\mathbb C},\log}), \;
w^{l,j-l}:=H^jOL_{X_{\mathbb C}^{et}}(w_L^{l,j-l}).
\end{eqnarray*}
Let $0\leq l\leq j$. Using an affine w-contractile pro-etale cover of $X$,
we see that there exists a finite affine etale cover 
$r=r(w^{l,j-l})=(r_i:X_i\to X)_{1\leq i\leq n}$ of $X$ (depending on $w^{l,j-l}$) such that
\begin{eqnarray*}
w^{l,j-l}=[(w^{l,j-l}_{L,I})_I]\in  
H^jOL_{X_{\mathbb C}^{et}}(H^{j-l}\Gamma(X_{\mathbb C,\bullet},\Omega^l_{X_{\mathbb C},\log}))
\subset\mathbb H^j\Gamma(X_{\mathbb C,\bullet},\Omega^{\bullet}_{X_{\mathbb C}}).
\end{eqnarray*}
Note that since $X$ is an algebraic variety, this also follows from a comparison theorem between
Chech cohomology of etale covers and etale cohomology.
By \cite{B6} lemma 2, we may assume, 
up to take a desingularization $\pi:X'\to X$ of $(X,\cup_i(r_i(X\backslash X_i)))$ and replace $w$ with $\pi^*w$,
that $r_i(X_i)=r_i(X_i(w))=X\backslash D_i$ with $D_i\subset X$ smooth divisors with normal crossing
For $1\leq l\leq j$, we get
\begin{eqnarray*}
w^{l,j-l}_{L,I}=\sum_{\nu}df_{\nu_1}/f_{\nu_1}\wedge\cdots\wedge df_{\nu_l}/f_{\nu_l}
\in\Gamma(X_{\mathbb C,I},\Omega^l_{X_{\mathbb C}}).
\end{eqnarray*}
For $l=0$, we get 
\begin{equation*}
w^{0,j}=[(\lambda_I)]\in H^j\Gamma(X_{\mathbb C,I},O_{X_{\mathbb C,I}}), \;
\lambda_I\in\Gamma(X_{\mathbb C,I},\mathbb Z_{X_{\mathbb C,I}^{et}})
\end{equation*}
There exists $k'\subset\mathbb C$ containing $k$ such that $w^{l,j-l}_{L,I}\in\Gamma(X_{k',I},\Omega^l_{X_{k'}})$
for all $0\leq l\leq j$.
Taking an embedding $\sigma':k'\hookrightarrow\mathbb C$ such that $\sigma'_{|k}=\sigma$, we then have
\begin{eqnarray*}
j_{\bullet}^*\circ\an_{X_{\bullet}}^*w=j_{\bullet}^*((m_l\cdot w^{l,j-l}_L)_{0\leq l\leq j})=(w^{l,j-l}_{L,I,J})_{l,I,J}
\in H^j\Gamma(X^{an}_{\bullet,\mathbb C}\cap\mathbb D_{\bullet},\Omega_{X_{\mathbb C}^{an}}^{\bullet}). 
\end{eqnarray*}
where for each $(I,J,l)$ with $card I+card J+l=j$, 
\begin{equation*}
w^{l,j-l}_{L,I,J}:=j_{IJ}^*w_{L,I}^{l,j-l}\in 
\Gamma(X^{an}_{I,\mathbb C}\cap\mathbb D_J,\Omega_{X_{\mathbb C}^{an}}^l).
\end{equation*} 
We have by a standard computation, for each $(I,J,l)$ with $card I+card J+l=j$,
\begin{equation*}
H^*_{\sing}(X_{I,\mathbb C}^{an}\cap\mathbb D_J,\mathbb Z)=<\gamma_1,\cdots,\gamma_{card I}>,
\end{equation*}
where for $1\leq i\leq card I$, $\gamma_i\in\Hom(\Delta^*,X_{I,\mathbb C}^{an}\cap\mathbb D_J)$
are products of loops around the origin inside the pointed disc $\mathbb D^1\backslash 0$.
On the other hand,
\begin{itemize}
\item $w^{l,j-l}_{L,I,J}=j_J^*(\sum_{\nu}df_{\nu_1}/f_{\nu_1}\wedge\cdots\wedge df_{\nu_l}/f_{\nu_l})
\in\Gamma(X_{I,\mathbb C}^{an}\cap\mathbb D_J,\Omega^l_{X_{\mathbb C}^{an}})$
for $1\leq l\leq j$, 
\item $w^{0,j}_{L,I,J}=\lambda_I$ is a constant.
\end{itemize}
Hence, for $\mu\in P([1,\cdots,s])$ with $card\mu=l$, we get, for $l=0$ 
$H^lev(X_{\mathbb C}^{an})_{I,J}(w_{L,I,J}^{0,j})=0$ and, for $1\leq l\leq j$,
\begin{eqnarray*}
H^lev(X_{\mathbb C}^{an})_{I,J}(w_{L,I,J}^{l,j-l})(\gamma_{\mu})=\sum_k\delta_{\nu,\mu}2i\pi\in 2i\pi\mathbb Z.
\end{eqnarray*}
where $\gamma_{\mu}:=\gamma_{\mu_1}\cdots\gamma_{\mu_l}$.
We conclude by \cite{B6} lemma 1.

\end{proof}

Let $k$ be a field of characteristic zero.
Let $X\in\SmVar(k)$. Let $X=\cup_{i=1}^sX_i$ an open affine cover with $X_i:=X\backslash D_i$ 
with $D_i\subset X$ smooth divisors with normal crossing. Let $\sigma:k\hookrightarrow\mathbb C$ an embedding.
By proposition \ref{keypropC}, we have a commutative diagram of graded algebras
\begin{equation*}
\xymatrix{H^*_{DR}(X)\ar[rrr]^{H^*ev(X)} & \, & \, & H^*_{\sing}(X_{\mathbb C}^{an},\mathbb C) \\
H^*OL_X(\mathbb H_{et}^*(X,\Omega^{\bullet}_{X^{et},\log}))
\cap H^*_{DR}(X)\ar[u]^{\subset}\ar[rrr]^{H^*ev(X)} & \, & \, &
H^*_{\sing}(X_{\mathbb C}^{an},2i\pi\mathbb Q)\ar[u]_{H^*C^*\iota_{2i\pi\mathbb Q/\mathbb C}(X_{\mathbb C}^{an})}}
\end{equation*}
where 
\begin{equation*}
C^*\iota_{2i\pi\mathbb Q/\mathbb C}(X_{\mathbb C}^{an}):
C^{\bullet}_{\sing}(X_{\mathbb C}^{an},2i\pi\mathbb Q)\hookrightarrow C^{\bullet}_{\sing}(X_{\mathbb C}^{an},\mathbb C)
\end{equation*}
is the subcomplex consiting of $\alpha\in C^j_{\sing}(X_{\mathbb C}^{an},\mathbb C)$ 
such that $\alpha(\gamma)\in 2i\pi\mathbb Q$ for all $\gamma\in C_j^{\sing}(X_{\mathbb C}^{an},\mathbb Q)$.
Recall that 
\begin{equation*}
H^*ev(X_{\mathbb C})=H^*R\Gamma(X_{\mathbb C}^{an},\alpha(X))^{-1}\circ\Gamma(X_{\mathbb C}^{an},E_{zar}(\Omega(\an_X))):
H^*_{DR}(X_{\mathbb C})\xrightarrow{\sim}H^*_{\sing}(X_{\mathbb C}^{an},\mathbb C)
\end{equation*}
is the canonical isomorphism induced by the analytical functor and 
$\alpha(X):\mathbb C_{X_{\mathbb C}^{an}}\hookrightarrow\Omega^{\bullet}_{X^{an}_{\mathbb C}}$, 
which gives the periods elements $H^*ev(X)(H^*_{DR}(X))\subset H^*_{\sing}(X_{\mathbb C}^{an},\mathbb C)$.
On the other side the induced map
\begin{equation*}
H^*ev(X_{\mathbb C}):
H^*OL_X(\mathbb H_{et}^*(X_{\mathbb C},\Omega^{\bullet}_{X_{\mathbb C}^{et},\log}))
\hookrightarrow H^*\iota_{2i\pi\mathbb Q/\mathbb C}H^*_{\sing}(X_{\mathbb C}^{an},2i\pi\mathbb Q)
\end{equation*}
is NOT surjective in general since the left hand side is invariant by the action of the group
$Aut(\mathbb C)$ (the group of field automorphism of $\mathbb C$) 
whereas the right hand side is not.
The fact for a de Rham cohomology class of being logarithmic is algebraic 
and invariant under isomorphism of (abstract) schemes.

\begin{cor}
Let $X\in\PSmVar(\mathbb C)$. Then the Hodge conjecture holds for $X$ if and only if the Hodge classes
are given by logarithmic De Rham classes.
\end{cor}

\begin{proof}
Follows from theorem \ref{drlogZ}.
\end{proof}

\subsection{Analytic logarithmic de Rham classes}

Let $K$ be a field of characteristic zero which is complete for a $p$-adic norm.
Denote by $O_K\subset K$ its ring of integer.
We consider $\Sch^{int}/O_K:=O(\PSch^2/O_K)\subset\Sch^{ft}/O_K$ the full subcategory
consisting of integral models of algebraic varieties over $K$,
where $O:\PSch^2/O_K\to\Sch^{ft}/O_K, \; O(X,Z)=X\backslash Z$ is the canonical functor, and 
\begin{equation*}
\Sch^{int,sm}/O_K:=\Sch^{int}/O_K\cap\Sch^{sm}/O_K\subset\Sch/O_K
\end{equation*}
the full subcategory consisting of smooth integral models i.e. integral models of (smooth) algebraic varieties over $K$ which are smooth over $O_K$. 
We then have the morphism of etale sites 
\begin{equation*}
r:\SmVar(K)\to\Sch^{int,sm}/O_K, \; 
X^{\mathcal O}\in\Sch^{int,sm}/O_K\mapsto X^{\mathcal O}\otimes_{O_K}K
\end{equation*}
We will consider for each $n\in\mathbb N$ 
\begin{eqnarray*}
\Omega^{\bullet}_{/(O_K/p^n),\log}=a_{et}\Omega^{\bullet}_{/(O_K/p^n),\log}\in C(\Sch^{int,sm}/O_K), \\
X^{\mathcal O}\mapsto\Omega^{\bullet}_{X^{\mathcal O}/p^n,\log}(X^{\mathcal O}), \; \; 
(f:X'\to X)\mapsto\Omega(f):=f^*
\end{eqnarray*}
which gives
\begin{eqnarray*}
\Omega^{\bullet,an}_{/K,\log,\mathcal O}:=\varprojlim_{n\in\mathbb N}a_{et}\Omega^{\bullet}_{/(O_K/p^n),\log}\in C(\Sch^{int,sm}/O_K), \\
X^{\mathcal O}\mapsto\Omega^{\bullet}_{\hat X^{(p)},\log,\mathcal O}(\hat X^{\mathcal O}), \; \; 
(f:X'\to X)\mapsto\Omega(f):=f^*
\end{eqnarray*}
together with the embedding in $C(\Sch^{int,sm}/O_K)$ (see definition \ref{wlogdef}(iii))
\begin{eqnarray*}
OL:=OL_{/O_K,an}:\Omega^{\bullet,an}_{/K,\log,\mathcal O}\hookrightarrow\Omega_{/O_K}^{\bullet,an}, \;
\mbox{for} \, X^{\mathcal O}\in(\Sch^{int,sm}/O_K), \\ 
OL_{/O_K,an}(X^{\mathcal O}):=OL_{\hat X^{\mathcal O}}(X^{\mathcal O}):
\Omega^{\bullet}_{\hat X^{\mathcal O,(p)},\log,\mathcal O}(\hat X^{\mathcal O})\to\Omega^{\bullet}_{\hat X^{\mathcal O}}(\hat X^{\mathcal O}).
\end{eqnarray*}
which induces the embedding of $C(\SmVar(K))$ (see definition \ref{wlogdef}(iii))
\begin{eqnarray*}
OL:=OL_{/K,an}:r^*\Omega^{\bullet,an}_{/K,\log,\mathcal O}\xrightarrow{r^*OL_{/O_K,an}}r^*\Omega_{/O_K}^{\bullet,an}, \; \; \mbox{for} \, X\in\SmVar(K), \\
OL_{/K,an}(X):=(OL_{/O_K,an}(Y^{\mathcal O})): \\
\varinjlim_{Y^{\mathcal O}\in(\Sch^{int,sm}/O_K),e:X\to Y^{\mathcal O}\times_{O_K}K=Y}
\Omega^{\bullet}_{\hat Y^{(p)},\log,\mathcal O}(\hat Y^{\mathcal O})\to\Omega^{\bullet}_{\hat Y^{\mathcal O}}(\hat Y^{\mathcal O}).
\end{eqnarray*}
We have also, for each $n\in\mathbb N$, the sheaves
\begin{eqnarray*}
O_n,O^*_n\in\PSh((\Sch^{int,sm}/O_K)), \; X^{\mathcal O}\mapsto 
O_n(X^{\mathcal O}):=O(X^{\mathcal O}_{/p^n}),\, 
O^*_n(X^{\mathcal O}):=O(X^{\mathcal O}_{/p^n})^*, \\ 
(g:Y\to X)\mapsto a_g(X^{\mathcal O}_{/p^n}):
O_n(X^{\mathcal O})\to O_n(Y^{\mathcal O}), \, O_n(X^{\mathcal O})^*\to O_n(Y^{\mathcal O})^*,
\end{eqnarray*}
and $O_n:=r^*O_n,O^*_n:=r^*O^*_n\in\PSh(\SmVar(K))$.

We will use the following :

\begin{lem}\label{a1trn}
\begin{itemize}
\item[(i0)]For each $l,n\in\mathbb N$, the sheaf $\Omega^l_{/(O_K/p^n)}\in\PSh(\Sch^{sm}/O_K)$ admits transfers
\item[(i1)]For each $n\in\mathbb N$,
the presheaves $O_n^*\in\PSh(\Sch^{sm}/O_K)$ and $\Omega^1_{/(O_K/p^n),\log}\in\PSh(\Sch^{sm}/O_K)$ admit transfers
compatible with transfers on $\Omega^1_{/(O_K/p^n)}\in\PSh(\Sch^{sm}/O_K)$.
\item[(i2)] For each $l,n\in\mathbb N$, the sheaf $a_{et}\Omega^l_{/(O_K/p^n),\log}\in\PSh(\Sch^{sm}/O_K)$ admits transfers
compatible with transfers on $\Omega^l_{/(O_K/p^n)}\in\PSh(\Sch^{sm}/O_K)$,
that is $\Omega^{\bullet}_{/(O_K/p^n)}\in C(\Cor\Sch^{sm}/O_K)$ and the inclusion 
$OL:\Omega^l_{/(O_K/p^n),\log}[-l]\hookrightarrow\Omega^{\bullet}_{/(O_K/p^n)}$ in $C(\Sch^{sm}/O_K)$ 
is compatible with transfers.
\item[(i3)] The presheaf $r^*\Omega^{l,an}_{/K,\log,\mathcal O}\in\PSh(\SmVar(K))$ admits transfers.
\item[(ii1)]For each $n\in\mathbb N$,
the etale sheaves $O_n^*\in\PSh(\Sch^{sm}/O_K)$ and $O_n^*\in\PSh(\Sch^{int,sm}/O_K)$ are $\mathbb A^1$ invariant.
\item[(ii2)] For each $l,n\in\mathbb N$, the sheaf $a_{et}\Omega^l_{/(O_K/p^n),\log}\in\PSh(\Sch^{int,sm}/O_K)$ is $\mathbb A^1$ invariant.
\item[(ii3)]For each $l\in\mathbb N$, the sheaf $a_{et}r^*\Omega^{l,an}_{/K,\log,\mathcal O}\in\PSh(\SmVar(K))$ is $\mathbb A^1$ invariant.
\end{itemize}
Recall that $a_{et}:\PSh(\SmVar(K))\to\Shv_{et}(\SmVar(K))$ is the sheaftification functor.
\end{lem}

\begin{proof}
\noindent (i0):The sheaves $\Omega^{1}_{/(O_K/p^n)}\in\PSh(\Sch^{sm}/O_K)$ admit transfers. 
Indeed, let $W\subset X'\times X$ finite over $X'$, with $X,X'\in\Sch^{sm}/O_K$.
Take $n:\tilde W\to W$ the normalisation of $W$. Since $X'$ is of finite type over $O_K$ it is an excellent scheme,
thus $n$ is a finite surjective morphism. Hence $m:=p_{X'}\circ n:\tilde W\to X'$ is a finite surjective morphism.
Since $X'$ is smooth and $\tilde W$ is normal, $m$ is a finite flat morphism. By base change
$m_{/p^n}:\tilde W_{/p^n}\to X'_{/p^n}$ are finite flat, since $m$ is finite flat. We thus have a canonical
trace map $Tr(m_{/p^n}):O_{\tilde W_{/p^n}}\to O_{X'_{/p^n}}$.
We get transfers on $\Omega^l_{/(O_K/p^n)}\in\PSh(\Sch^{sm}/O_K)$ by induction on $l$ (see the proof of (i2)).

\noindent(i1):By (i0), the sheaves $\Omega^{1}_{/(O_K/p^n)}\in\PSh(\Sch^{sm}/O_K)$ admit transfers.
On the other hand the sheaf $O_1^*\in\PSh(\Sch^{sm}/O_K)$ admits transfers : 
for $W\subset X'\times X$ with $X,X'\in\Sch^{sm}/O_K$
and $W$ finite over $X'$ and $f\in O(X^{\mathcal O}_{/p})^*$, $W^*f:=N_{W/X'}(p_X^*f)$ 
where $p_X:W\hookrightarrow X'\times X\to X$ is the projection and $N_{W/X'}:O_K/p(W)^*\to O_K/p(X')^*$ is the norm map.
This gives transfers on $\Omega^{1}_{/(O_K/p),\log}\in\PSh(\Sch^{sm}/O_K)$ compatible with transfers on 
$\Omega^{1}_{/(O_K/p)}\in\PSh(\Sch^{sm}/O_K)$ :
for $W\subset X'\times X$ with $X,X'\in\Sch^{sm}/O_K$ and $W$ finite over $X'$ and $f\in O(X^{\mathcal O}_{/p^n})^*$, 
\begin{equation*}
W^*df/f:=dW^*f/W^*f=Tr_{W/X'}(p_X^*(df/f)), 
\end{equation*}
where where $p_X:W\hookrightarrow X'\times X\to X$ is the projection and $Tr_{W/X'}:O_W\to O_X$ is the trace map. 
Note that $d(fg)/fg=df/f+dg/g$. 
Considering the commutative diagram in $\PSh(\Sch^{sm}/O_K)$
\begin{equation*}
\xymatrix{0\ar[r] & \Omega^{1}_{/(O_K/p^{n-1})}\ar[r] & \Omega^{1}_{/(O_K/p^n)}\ar[r]^{\Omega(/p)} & 
\Omega^{1}_{/(O_K/p)}\ar[r] & 0 \\
0\ar[r] & a_{et}\Omega^{1}_{/(O_K/p^{n-1}),\log}\ar[r]\ar[u]^{OL} & 
a_{et}\Omega^{1}_{/(O_K/p^n),\log}\ar[u]^{OL}\ar[r]^{\Omega(/p)} & a_{et}\Omega^{1}_{/(O_K/p),\log}\ar[u]^{OL}\ar[r] & 0}
\end{equation*}
we see by induction on $n\in\mathbb N$ that the transfers on $\Omega^{1}_{/(O_K/p^n)}\in\PSh(\Sch^{sm}/O_K)$
send logarithmic forms to logarithimc forms. 
Hence the transfers on $\Omega^{1}_{/(O_K/p^n)}\in\PSh(\Sch^{sm}/O_K)$
send logarithmic forms to logarithimc forms. This proves (i0).

\noindent(i2): By (i1), we get transfers on 
\begin{equation*}
\otimes^l_{\mathbb Z/p^n}\Omega^{1}_{/(O_K/p^n),\log}, \, \otimes_{O_n}^l\Omega^{1}_{/(O_K/p^n)}\in\PSh(\Sch^{sm}/O_K)
\end{equation*}
since 
$\otimes^l_{\mathbb Z/p^n}\Omega^{1}_{/K,\log}=H^0(\otimes^{L,l}_{\mathbb Z/p^n}\Omega^{1}_{/(O_K/p^n),\log})$ 
and $\otimes_{O_n}^l\Omega^{1,an}_{/(O_K/p^n)}=H^0(\otimes_{\hat O}^{L,l}\Omega^{1,an}_{/(O_K/p^n)})$.
This induces transfers on  
\begin{eqnarray*}
\wedge^l_{\mathbb Z/p^n}\Omega^{1}_{/(O_K/p^n),\log}:=
\coker(\oplus_{I_2\subset[1,\ldots,l]}\otimes^{l-1}_{\mathbb Z/p^n}\Omega^{1}_{/(O_K/p^n),\log} \\
\xrightarrow{\oplus_{I_2\subset[1,\ldots,l]}\Delta_{I_2}:=(w\otimes w'\mapsto w\otimes w\otimes w')}
\otimes^l_{\mathbb Z/p^n}\Omega^{1}_{/(O_K/p^n),\log}) 
\in\PSh(\Sch^{sm}/O_K).
\end{eqnarray*}
and
\begin{eqnarray*}
\wedge^l_{O_n}\Omega^{1}_{/(O_K/p^n)}:=
\coker(\oplus_{I_2\subset[1,\ldots,l]}\otimes^{l-1}_{O_n}\Omega^{1}_{/(O_K/p^n)} \\
\xrightarrow{\oplus_{I_2\subset[1,\ldots,l]}\Delta_{I_2}:=(w\otimes w'\mapsto w\otimes w\otimes w')})
\otimes^l_{O_n}\Omega^{1}_{/(O_K/p^n)}
\in\PSh(\Sch^{sm}/O_K).
\end{eqnarray*}

\noindent(i3):The presheaves for each $l\in\mathbb Z$ 
\begin{equation*}
\Omega^{l,an}_{/K,\log,\mathcal O}:=\varprojlim_{n\in\mathbb N}a_{et}\Omega^{l}_{/(O_K/p^n),\log}\in\PSh(\Sch^{int,sm}/O_K)
\end{equation*}
admits transfers by (i2). Hence, the presheaves for each $l\in\mathbb Z$ 
\begin{equation*}
a_{et}r^*\Omega^{l,an}_{/K,\log,\mathcal O}=r^*\varprojlim_{n\in\mathbb N}a_{et}\Omega^{l}_{/(O_K/p^n),\log}
\in\PSh(\SmVar(K))
\end{equation*}
admit transfers since if $\Gamma\subset X'\times X''$ with $X',X''\in\SmVar(K)$ is finite surjective over $X'$,
then, $\bar\Gamma\subset X^{'\mathcal O}\times X^{''\mathcal O}$ with $X^{'\mathcal O},X^{''\mathcal O}\in\Sch^{int,sm}/O_K$
such that $X'=X^{'\mathcal O}\times_{O_K}K, X''=X^{''\mathcal O}\times_{O_K}K$ is finite surjective over $X^{'\mathcal O}$
($\bar\Gamma$ is proper and affine over $X^{'\mathcal O}$ and recall that a proper affine morphism is finite).

\noindent(ii1):We have for $X\in\Sch^{sm}/O_K$, $O^*(X\times\mathbb A^1)=O^*(X)$ and
since for a commutative ring $A$, $(A[X])^*=A^*$. This proves (ii1).

\noindent(ii2):Let $l\in\mathbb N$. It follows (ii1) and the exact sequence in $\Shv_{et}(\Sch^{int,sm}/O_K)$
\begin{equation*}
0\to O^*\xrightarrow{(-)^p}O^*\xrightarrow{dlog}a_{et}\Omega^1_{/(O_K/p),\log}\to 0
\end{equation*}
that $a_{et}\Omega^1_{/(O_K/p^n),\log}\in\PSh(\Sch^{int,sm}/O_K)$ is $\mathbb A^1$ invariant
Hence $a_{et}\Omega^l_{/(O_K/p^n),\log}\in\PSh(\Sch^{int,sm}/O_K)$ is $\mathbb A^1$ invariant by right exactness of tensor product
using the factorization
\begin{equation*}
\wedge^lOL:\wedge^la_{et}\Omega^1_{/(O_K/p),\log}\xrightarrow{\wedge^lOL}a_{et}\Omega^l_{/(O_K/p),\log}\xrightarrow{OL}\Omega^l_{/(O_K/p)}.
\end{equation*}
(ii2) then follows from induction on $n\in\mathbb N$ by the exact sequence
\begin{equation*}
0\to a_{et}\Omega^l_{/(O_K/p^{n-1}),\log}\xrightarrow{p\cdots(-)}a_{et}\Omega^l_{/(O_K/p^n),\log}\xrightarrow{/p^{n-1}}a_{et}\Omega^1_{/(O_K/p),\log}\to 0
\end{equation*}

\noindent(ii3): The question is local for the etale topology, hence we may assume that $X=\Spec(E)$ the local hensel ring. 
We thus get by (ii2) that 
\begin{equation*}
r^*\Omega^{l,an}_{/K,\log,\mathcal O}(X\times\mathbb A^1)=\Omega^{l,an}_{/K,\log,\mathcal O}(X^{\mathcal O}\times\mathbb A_{O_K}^1)
=\Omega^{l,an}_{/K,\log,\mathcal O}(X^{\mathcal O})=r^*\Omega^{l,an}_{/K,\log,\mathcal O}(X). 
\end{equation*}
This proves (ii1).
\end{proof}

For $X^{\mathcal O}\in\PSch^{sm}/O_K$ smooth projective over $O_K$ and $X:=X^{\mathcal O}\otimes_{O_K}K$, the morphism  
for each $j\in\mathbb Z$  
\begin{equation*}
c^*:H_j^{DR}(X)=H^j_{DR}(X^{\mathcal O})\otimes_{O_K}K=\mathbb H_{pet}^j(X,\Omega^{\bullet}_X)\to
\mathbb H_{pet}^j(X^{\mathcal O},\Omega_{/O_K}^{\bullet,an})\otimes_{O_K}K=H_j^{DR}(\hat X^{\mathcal O})\otimes_{O_K}K
\end{equation*}
is an isomorphism by GAGA and the Hodge De Rham spectral sequence.

\begin{prop}\label{GAGAlog}
Let $K$ be a field of characteristic zero which is complete for a $p$-adic norm. Let $X\in\PSmVar(K)$ with good reduction.
Let $X^{\mathcal O}\in\PSch/O_K$, smooth over $O_K$ such that $X^{\mathcal O}\otimes_{O_K}K=X$.
Consider the morphism $c:\hat X^{\mathcal O}\to X^{\mathcal O}$ in $\RTop$ given by completion of $X_{\mathcal O}$ with respect to $(p)$.
\begin{itemize}
\item[(i)]The analytic De Rham cohomology class of a formal algebraic cycle is logarithmic and is of type $(d,d)$,
that is, for $Z^p\in\mathcal Z^d(X^{\mathcal O}/p)$  
\begin{eqnarray*}
[Z^p]:=H^{2d}\Omega(\gamma^{\vee}_Z)([Z^p])\in H^{2d}OL_{\hat X}(H_{pet}^d(X,\Omega^d_{\hat X^{(p)},\log,\mathcal O}))
\subset H_{DR}^{2d}(\hat X^{\mathcal O})\otimes_{O_K}K=H^{2d}_{DR}(X), \\ Z:=Z^{\mathcal O}\times_{O_K}K, \; Z^{\mathcal O}\cap X^{\mathcal O}/p=Z^p\cup Z^{'p},
\end{eqnarray*}
where $Z^{\mathcal O}\subset X^{\mathcal O}$ has its irreducible components of codimension $d$ and (proper) surjective over $O_K$.
\item[(ii)]Conversely, if $w\in H^{2d}OL_{\hat X}(H_{pet}^d(X,\Omega^d_{\hat X^{(p)},\log,\mathcal O}))$ 
then there exists a zariski closed subset $Z^{\mathcal O}\subset X^{\mathcal O}$ whose irreducible components are of codimension $d$ and (proper) surjective over $O_K$
such that
\begin{equation*}
w=H^{2d}\Omega(\gamma^{\vee}_Z)(w)\subset
\Im(H_{DR,\hat Z^{\mathcal O}}^{2d}(\hat X^{\mathcal O})\to H_{DR}^{2d}(\hat X^{\mathcal O})\otimes_{O_K}K=H^{2d}_{DR}(X)).
\end{equation*}
\item[(iii)] We have $H^jOL_{\hat X}(H_{pet}^{j-l}(X,\Omega^l_{\hat X^{(p)},\log,\mathcal O}))=0$ 
for $j,l\in\mathbb Z$ such that $2l\neq j$. 
That is, we only have analytic logarithmic classes in bidegree $(d,d)$ for $d\in\mathbb N$,
in particular there is no non trivial logarithmic classes for odd degree De Rham cohomology 
$H^{2d+1}_{DR}(\hat X^{\mathcal O})\otimes_{O_K}K=H^{2d+1}_{DR}(X)$.
\end{itemize}
\end{prop}

\begin{proof}
Consider, for $j,l\in\mathbb Z$ and $Y\in\PSmVar(K)$ with good reduction and $Y^{\mathcal O}\in\PSch^{sm}/O_K$ a smooth integral model, 
\begin{equation*}
L^{l,j-l}(\hat Y^{\mathcal O}):=H^jOL_{\hat Y^{(p)}}(H_{pet}^{j-l}(Y,\Omega^l_{\hat Y^{(p)},\log,\mathcal O}))\subset H_{DR}^j(\hat Y^{\mathcal O})\otimes_{O_K}K.
\end{equation*} 
By proposition \ref{UXpet}(ii) we have,
\begin{equation*}
L^{l,j-l}(\hat Y^{\mathcal O})=H^jOL_{\hat Y^{(p)}}(H_{et}^{j-l}(Y,\Omega^l_{\hat Y^{(p)},\log,\mathcal O}))\subset H_{DR}^j(\hat Y^{\mathcal O})\otimes_{O_K}K.
\end{equation*}
For $j,l\in\mathbb Z$ and $T^{\mathcal O}\subset Y^{\mathcal O}$ a closed subset which is an integral model, denoting by $T:=T^{\mathcal O}\times_{O_K}K$,
we consider
\begin{equation*}
L_{\hat T^{\mathcal O}}^{l,j-l}(\hat Y^{\mathcal O}):=H^jOL_{\hat Y^{(p)}}(H_{et,T}^{j-l}(Y,\Omega_{\hat Y^{(p)},\log,\mathcal O}^l))
\subset H_{DR,\hat T^{\mathcal O}}^j(\hat Y^{\mathcal O})\otimes_{O_K}K.
\end{equation*}
We also consider for $Y\in\SmVar(K)$ (not necessary with good reduction)
\begin{eqnarray*}
L^{l,j-l}(Y):=H^jOL_{/K,an}(H^{j-l}\Hom(\mathbb Z(Y),E_{et}(r^*\Omega_{/K,\log,\mathcal O}^{l,an}))) \\
\subset H^j\Hom(\mathbb Z(Y),E_{et}(r^*\Omega_{/O_K}^{\bullet,an}))\otimes_{O_K}K=\mathbb H_{et}^j(Y,r^*\Omega_{/O_K}^{\bullet,an})\otimes_{O_K}K
\end{eqnarray*}
and if $T\subset Y$ is a closed subset
\begin{eqnarray*}
L_T^{l,j-l}(Y):=H^jOL_{/K,an}(H^{j-l}\Hom(\mathbb Z(Y,Y\backslash T),E_{et}(r^*\Omega_{/K,\log,\mathcal O}^{l,an}))) \\
\subset H^j\Hom(\mathbb Z(Y,Y\backslash T),E_{et}(r^*\Omega_{/O_K}^{\bullet,an}))\otimes_{O_K}K=\mathbb H_{et,T}^j(Y,r^*\Omega_{/O_K}^{\bullet,an})\otimes_{O_K}K.
\end{eqnarray*}
By lemma \ref{a1trn} (i3), the presheaves for each $l\in\mathbb Z$ 
\begin{equation*}
a_{et}r^*\Omega^{l,an}_{/K,\log,\mathcal O}:=a_{et}r^*\varprojlim_{n\in\mathbb N}a_{et}\Omega^{l}_{/(O_K/p^n),\log}\in\PSh(\SmVar(K))
\end{equation*}
admit transfers. The presheaves for each $l\in\mathbb Z$ 
\begin{equation*}
(a_{et}r^*\Omega^{l,an}_{/K,\log,\mathcal O})\otimes_{\mathbb Z_p}\mathbb Q_p
=(r^*\varprojlim_{n\in\mathbb N}a_{et}\Omega^{l}_{/(O_K/p^n),\log})\otimes_{\mathbb Z_p}\mathbb Q_p\in\PSh(\SmVar(K))
\end{equation*}
are also $\mathbb A^1$ invariant by lemma \ref{a1trn} (ii3).
Hence by theorem \ref{Voethm}, they are $\mathbb A^1$ local since they are $\mathbb A^1$ invariant and admit transfers.
This gives in particular, for $T\subset Y$ a smooth subvariety of (pure) codimension $d$, 
by proposition \ref{smXZ} an isomorphism
\begin{eqnarray*}
E_{et}r^*\Omega^{l,an}_{/K,\log,\mathcal O}(P_{T,Y}):L_T^{l,j-l}(Y)\xrightarrow{\sim}L_T^{l,j-l}(N_{T/Y})\xrightarrow{\sim}L^{l-d,j-l-d}(T).
\end{eqnarray*} 
Note that $T$ does NOT have good reduction in general 
(the number of irreducible components of $T^{\mathcal O}/p$ is greater then the number of irreducible components of $T$ in general).
So, let 
\begin{eqnarray*}
\alpha\in L^{l,j-l}(\hat X^{\mathcal O}):=
H^jOL_{\hat X^{(p)}}(H_{pet}^{j-l}(X,\Omega^l_{\hat X^{(p)},\log,\mathcal O}))=H^jOL_{\hat X^{(p)}}(H_{et}^{j-l}(X,\Omega^l_{\hat X^{(p)},\log,\mathcal O})).
\end{eqnarray*} 
with $j-l\neq 0$. Let 
\begin{eqnarray*}
\alpha:=S(X^{\mathcal O}/X)(\alpha)=H^jOL_{/K,an}\circ S(X^{\mathcal O}/X)(\alpha) \\
\in L^{l,j-l}(X):=H^jOL_{/K,an}(H^{j-l}\Hom(\mathbb Z(X),E_{et}(r^*\Omega_{/K,\log,\mathcal O}^{l,an}))).
\end{eqnarray*}
Since $\Omega^l_{\hat X^{(p)},\log,\mathcal O}\in\PSh(X^{et})$ consist of a single presheaf, there exists
a zariski open subset $U\subset X$ such that  
\begin{equation*}
j^*\alpha=0\in H_{et}^{j-l}(U,r^*\Omega^{l,an}_{/K,\log,\mathcal O})\otimes_{\mathbb Z_p}\mathbb Q_p.
\end{equation*}
Considering a divisor $X\backslash U\subset D\subset X$, we get 
\begin{equation*}
\alpha=H^{j-l}E_{et}(r^*\Omega_{/K,\log,\mathcal O}^{l,an})(\gamma^{\vee}_D)(\alpha), \, 
\alpha\in L_D^{l,j-l}(X).
\end{equation*}
\begin{itemize}
\item Proof of (i):Similar to the proof of theorem \ref{drlogZ}(i).
\item Proof of (ii):By assumption we have $\alpha\in L^{d,d}(\hat X^{\mathcal O})$. Let $\alpha:=S(X^{\mathcal O}/X)(\alpha)\in L^{d,d}(X)$.
Considering a divisor $X\backslash U\subset D\subset X$, we get 
\begin{equation*}
\alpha=H^{d}E_{et}(r^*\Omega_{/K,\log,\mathcal O}^{d,an})(\gamma^{\vee}_D)(\alpha), \, \alpha\in L_D^{d,d}(X).
\end{equation*}
Denote $D^o\subset D$ its smooth locus of $D$, and $l:X^o\hookrightarrow X$ a Zariski open subset such that $X^o\cap D=D^o$. We then have 
\begin{equation*}
l^*\alpha\in L^{d,d}_{D^o}(X^o)=L^{d-1,d-1}(D^o)
\end{equation*}
Considering each connected components of $D^o$ instead of $X$, since $\Omega_{/K,\log,\mathcal O}^{d,an}$ consists in a single presheaf,
after a finite induction of $d$ steps, we get 
\begin{equation*}
\alpha=H^{d}E_{et}(r^*\Omega_{/K,\log,\mathcal O}^{d,an})(\gamma^{\vee}_Z)(\alpha), \, \alpha\in L_Z^{d,d}(X).
\end{equation*}
with $Z:=D_d\subset\cdots\subset D\subset X$ a pure codimension $d$ (Zariski) closed subset. Let $Z^{\mathcal O}:=\bar Z\subset X^{\mathcal O}$.
By the injectivity of $S(X^{\mathcal O}\backslash Z^{\mathcal O}/X\backslash Z)$ (c.f. proposition \ref{UXpet} (ii)) 
\begin{equation*} 
\alpha\in L_{\hat Z^{\mathcal O}}^{d,d}(X^{\mathcal O})=\oplus_{i=1}^t\oplus_{j=1}^{s_i}\mathbb Q_p[(Z^{\mathcal O}_i/p)_j]
\subset H^{2d}_{DR,\hat Z^{\mathcal O}}(\hat X^{\mathcal O}).
\end{equation*}
where $(Z^{\mathcal O}_i)_{1\leq i\leq t}\subset Z$ are the irreducible components of $Z^{\mathcal O}:=V(\mathcal I^O_Z)\subset X^{\mathcal O}$,
and $((Z^{\mathcal O}_i/p)_j)_{1\leq j\leq s_i}$ are the irreducible components of $Z^{\mathcal O}_i/p$ 
(this equality is standard for formal cohomology using the fact that for $W^{\mathcal O}\subset V^{\mathcal O}$ a closed embedding of smooth integral models
$\widehat{\hat V^{\mathcal O}}_{\hat W^{\mathcal O}}=\widehat{\hat N_{W^{\mathcal O}/V^{\mathcal O}}}_{\hat W^{\mathcal O}}$).

\item Proof of (iii) part 1: Let $j>2l$. Let $\alpha\in L^{l,j-l}(\hat X^{\mathcal O})$ and $\alpha:=S(X^{\mathcal O}/X)(\alpha)\in L^{l,j-l}(X)$.
By the proof of (ii) there exists $Z\subset X$ a closed subset of pure codimension $l$ such that
\begin{equation*}
\alpha=H^{j-l}E_{et}(r^*\Omega_{/K,\log,\mathcal O}^{l,an})(\gamma^{\vee}_Z)(\alpha), \, 
\alpha\in L_Z^{l,j-l}(X).
\end{equation*}
By a finite induction of $d_X-l$ steps, restricting to the smooth locus of closed subsets of $Z$, 
$\alpha=0$ since $L_{Z'}^{l,j-l}(X)=0$ for all $Z'\in\SmVar(k)$ such that $\codim(Z',X)\geq l$.

\item Proof of (iii) part 2: Let $j<2l$. Let $\alpha\in L^{l,j-l}(\hat X^{\mathcal O})$ and $\alpha:=S(X^{\mathcal O}/X)(\alpha)\in L^{l,j-l}(X)$. 
By the proof of (ii) there exists $Z\subset X$ a closed subset of pure codimension $j-l$ such that
\begin{equation*}
\alpha=H^{j-l}E_{et}(r^*\Omega_{/K,\log,\mathcal O}^{l,an})(\gamma^{\vee}_Z)(\alpha), \, \alpha\in L_Z^{l,j-l}(X).
\end{equation*} 
Since $S((X\backslash Z)^{\mathcal O}/X\backslash Z)$ is injective by proposition \ref{UXpet}(ii), 
we get $\alpha\in L^{l,j-l}_{\hat Z^{\mathcal O}}(\hat X^{\mathcal O})$ where $Z^{\mathcal O}:=\bar Z\subset X^{\mathcal O}$.
Let $l:X^o\hookrightarrow X$ be an open embedding such that $Z^o:=Z\cap X^o$ is the smooth locus of $Z$.
We then have the following commutative diagram of abelian groups
\begin{equation*}
\xymatrix{0\ar[r] & H_{et,Z}^{j-l}(X,\Omega^l_{\hat X^{(p)},\log,\mathcal O})\ar[r]^{l^*} & 
H_{et,Z^o}^{j-l}(X^o,\Omega^l_{\hat X^{(p)},\log,\mathcal O})\ar[r]^{\partial} & 
H_{et,Z\backslash Z^o}^{j-l+1}(X,\Omega^l_{\hat X^{(p)},\log,\mathcal O})\ar[r] & \cdots \\
0\ar[r] & H^0_{et}(Z,\Omega^{2l-j}_{\hat Z^{(p)},\log,\mathcal O})\ar[r]^{l*}\ar[u]^{\Omega(\hat G_{Z^{\mathcal O},X^{\mathcal O}})} & 
H^0_{et}(Z^o,\Omega^{2l-j}_{\hat Z^{(p)},\log,\mathcal O})\ar[r]^{\partial}\ar[u]^{\Omega(P_{Z^{\mathcal O,o},X^{\mathcal O,o}})} &
H_{et,Z\backslash Z^o}^1(Z,\Omega^{2l-j}_{\hat Z^{(p)},\log,\mathcal O})\ar[u]^{\Omega(\hat G_{Z^{\mathcal O},X^{\mathcal O}})}\ar[r] & \cdots}
\end{equation*}
whose rows are exact sequences and $\Omega(\hat G_{Z^{\mathcal O},X^{\mathcal O}})$ and $\Omega(P_{Z^o,X^o})$ are the map given similarly to the motivic case
using the fact that for $W^{\mathcal O}\subset V^{\mathcal O}$ a closed embedding of smooth integral models
$\widehat{\hat V^{\mathcal O}}_{\hat W^{\mathcal O}}=\widehat{\hat N_{W^{\mathcal O}/V^{\mathcal O}}}_{\hat W^{\mathcal O}}$. Consider 
\begin{equation*}
l^*\alpha=\Omega(P_{Z^{\mathcal O,o},X^{\mathcal O,o}})(\alpha^o)\in H_{et,Z^o}^{j-l}(X^o,\Omega^l_{\hat X^{(p)},\log,\mathcal O}), \; 
\alpha^o\in H^0_{et}(Z^o,r^*\Omega^{2l-j}_{\hat Z^{(p)},\log,\mathcal O}).
\end{equation*}
Since $\partial l^*\alpha=0\in H_{et,Z\backslash Z^o}^{j-l+1}(X,\Omega^l_{\hat X^{(p)},\log,\mathcal O})$, we get,
$(Z\backslash Z^o)^o\subset Z\backslash Z^o$ and $X^{oo}\subset X$ is an open subset
such that $X\cap(Z\backslash Z^o)^o=X^{oo}\cap (Z\backslash Z^o)$, 
\begin{equation*}
\partial\alpha^o=0\in H_{et,Z\backslash Z^o}^1(Z,\Omega^{2l-j}_{\hat Z^{(p)},\log,\mathcal O}), 
\end{equation*}
since $H_{et,(Z\backslash Z^o)\backslash(Z\backslash Z^o)^o}^1(Z,\Omega^{2l-j}_{\hat Z^{(p)},\log,\mathcal O})=0$
for dimension reasons, that is 
\begin{equation*}
\alpha=\Omega(\hat G_{Z^{\mathcal O},X^{\mathcal O}})(\alpha), \, \mbox{with} \, \alpha\in H^0_{et}(Z,\Omega^{2l-j}_{\hat Z^{(p)},\log,\mathcal O}).
\end{equation*} 
Hence $\alpha=0$ since $H^0(Z',\Omega^k_{Z^{'(p)},\log,\mathcal O})=0$ for all $k>0$ and all $Z'\in\PVar(K)$.
\end{itemize}

\end{proof}

\section{Tate conjecture}

Let $k$ be a field of finite type over $\mathbb Q$. 
Denote $\bar k$ the algebraic closure of $k$ and $G=Gal(\bar k/k)$ the absolute Galois group of $k$.

Let $X\in\SmVar(k)$ be a smooth variety.   
Let $p\in\mathbb N$ a prime number. Consider an embedding $\sigma_p:k\hookrightarrow\mathbb C_p$.
Denote $\hat k_{\sigma_p}\subset\mathbb C_p$ the $p$-adic completion of $k$ with respect to $\sigma_p$.
We have the commutative diagram in $C_{\mathbb B_{dr}fil,\hat G_{\sigma_p}}(X_{\mathbb C_p}^{an,pet})$
\begin{equation*}
\xymatrix{(\mathbb B_{dr,X_{\mathbb C_p}},F)\ar[rrr]^{\alpha(X)} & \, & \, &
(\Omega^{\bullet}_{X_{\mathbb C_p}},F_b)\otimes_{O_{X_{\mathbb C_p}}}(O\mathbb B_{dr,X_{\mathbb C_p}},F) \\
\underline{\mathbb Z_p}_{X_{\mathbb C_p}}
\ar[u]^{\iota'_{X_{\mathbb C_p}^{pet}}:=l\mapsto l.1}
\ar[rrr]^{\iota_{X_{\mathbb C_p}^{pet}}:=(I,0)} & \, & \, &
(\Omega^{\bullet}_{X_{\mathbb C_p},\log,\mathcal O}\otimes\mathbb Z_p,F_b)
\ar[u]_{OL_X\otimes I:=(w\otimes\lambda_n)_{n\in\mathbb N\mapsto (w\otimes\lambda_n)_{n\in\mathbb N}}}}.
\end{equation*}
Consider an integral model $X^{\mathcal O}_{\hat k_{\sigma_p}}\in\Sch^{int}/O_{\hat k_{\sigma_p}}$ 
of $X_{\hat k_{\sigma_p}}$. If $X^{\mathcal O}_{\hat k_{\sigma_p}}$ has good reduction modulo $p$,
we have (see e.g. \cite{ChinoisCrys}) the embedding in $C((X^{\mathcal O}_{\hat k_{\sigma_p}})^{Falt})$
\begin{equation*}
\alpha(X):\mathbb B_{cris, X_{\hat k_{\sigma_p}}}\hookrightarrow 
a_{\bullet*}\Omega^{\bullet}_{X^{\mathcal O,\bullet}_{\hat k_{\sigma_p}}}
\otimes_{O_{X^{\mathcal O}}}O\mathbb B_{cris,X^{\bullet}_{\hat k_{\sigma_p}}}
\end{equation*}
which is a filtered quasi-isomorphism compatible with 
the action of $Gal(\mathbb C_p/\hat k_{\sigma_p})$ and the action of the Frobenius $\phi_p$,
where $(X^{\mathcal O}_{\hat k_{\sigma_p}})^{Falt}$ denote the Falting site.

\begin{defi}\label{deltakX}
For $k$ a field of finite type over $\mathbb Q$ and $X\in\SmVar(k)$, 
we denote $\delta(k,X)^0\subset\Spec(O_k)$ the finite set where $O_k\subset k$ is the ring of integers 
such that if $p\in\mathbb N\backslash a_{O_k}(\delta(k,X)^0)$ is a prime number, $k$ is unramified at $p$ and 
there exists an integral model $X^{\mathcal O}_{\hat k_{\sigma_p}}\in\Sch^{int}/O_{\hat k_{\sigma_p}}$ 
of $X_{\hat k_{\sigma_p}}$ with good reduction modulo $p$ for all embeddings $\sigma_p:k\hookrightarrow\mathbb C_p$,
$\hat k_{\sigma_p}\subset\mathbb C_p$ being the $p$-adic completion of $k$ with respect to $\sigma_p$.
Here $a_{O_k}:\Spec(O_k)\to\Spec(\mathbb Z)$ is the canonical map given by the inclusion $\mathbb Z\subset O_k$.
To see that $\delta(k,X)^0$ is a finite set, considering $k=k_0(x_1,\cdots,x_d)(x_{d+1})$
where $x_1,\cdots x_d$ are algebraicaly independent and $k_0$ is a number field
one can take an integral model $Y\to O_{k_0}[x_1,\cdots,x_{d+1}]/f_{d+1}$ of $X$,
a desingularization of $Y'\to Y$ of $Y$ and then see that $\delta(k,X)$ is contained in the discriminant of the
family $Y'\to Y\to O_{k_0}[x_1,\cdots,x_{d+1}]/f_{d+1}$.
\end{defi}

Let $k$ be a field of finite type over $\mathbb Q$. Let $X\in\PSmVar(k)$ be a smooth projective variety.
Let $p\in\mathbb N\backslash a_{O_k}(\delta(k,X)^0)$ be a prime number. Consider an embedding $\sigma_p:k\hookrightarrow\mathbb C_p$.
Then $X_{\hat k_{\sigma_p}}$ has good reduction modulo $p$ and let 
$X_{\hat k_{\sigma_p}}^{\mathcal O}\in\PSch/O_{\hat k_{\sigma_p}}$ be a smooth model, 
i.e. $X_{\hat k_{\sigma_p}}^{\mathcal O}\otimes_{O_{\hat k_{\sigma_p}}}\hat k_{\sigma_p}=X_{\hat k_{\sigma_p}}$
and $X_{\hat k_{\sigma_p}}^{\mathcal O}$ is smooth with smooth special fiber.
The main result of \cite{ChinoisCrys} say in this case that the embedding in $C((X^{\mathcal O}_{\hat k_{\sigma_p}})^{Falt})$
\begin{equation*}
\alpha(X):\mathbb B_{cris, X_{\hat k_{\sigma_p}}}\hookrightarrow 
a_{\bullet*}\Omega^{\bullet}_{X^{\mathcal O,\bullet}_{\hat k_{\sigma_p}}}
\otimes_{O_{X^{\mathcal O}}}O\mathbb B_{cris,X^{\bullet}_{\hat k_{\sigma_p}}}
\end{equation*}
induces a filtered quasi-isomorphism compatible 
for each $j\in\mathbb Z$, a filtered isomorphism of filtered abelian groups
\begin{eqnarray*}
H^jR\alpha(X):H_{et}^j(X_{\mathbb C_p},\mathbb Z_p)\otimes_{\mathbb Z_p}\mathbb B_{cris,\hat k_{\sigma_p}}
\xrightarrow{H^jT(a_X,\mathbb B_{cris})^{-1}}
H_{et}^j((X)^{Falt})(\mathbb B_{cris, X_{\hat k_{\sigma_p}}}) \\
\xrightarrow{H^jR\Gamma(X_{\hat k_{\sigma_p}}^{\mathcal O},\alpha(X))}
H^j_{DR}(X_{\hat k_{\sigma_p}})\otimes_{\hat k_{\sigma_p}}\mathbb B_{cris,\hat k_{\sigma_p}}
\end{eqnarray*}
compatible with the action of $G_p:=Gal(\mathbb C_p/\hat k_{\sigma_p})$ and of the Frobenius $\phi_p$.

\begin{defi}\label{walpha}
Let $k$ be a field of finite type over $\mathbb Q$. 
Denote $\bar k$ the algebraic closure of $k$ and $G=Gal(\bar k/k)$ the absolute Galois group of $k$.
Let $X\in\PSmVar(k)$. Let $p\in\mathbb N\backslash a_{O_k}(\delta(k,X)^0)$ be a prime number. 
Consider an embedding $\sigma_p:k\hookrightarrow\mathbb C_p$.
Denote $\hat k_{\sigma_p}\subset\mathbb C_p$ the $p$-adic completion of $k$ with respect to $\sigma_p$.
For $\alpha\in H^j_{et}(X_{\mathbb C_p},\mathbb Z_p)$, we consider
\begin{eqnarray*}
w(\alpha):=H^jR\alpha(X)(\alpha\otimes 1)\in
H^j_{DR}(X_{\hat k_{\sigma_p}})\otimes_{\hat k_{\sigma_p}}\mathbb B_{cris,\hat k_{\sigma_p}}.
\end{eqnarray*}
the associated de Rham class by the $p$-adic periods. We recall
\begin{eqnarray*}
H^jR\alpha(X):H_{et}^j(X_{\mathbb C_p},\mathbb Z_p)\otimes_{\mathbb Z_p}\mathbb B_{cris,\hat k_{\sigma_p}}
\xrightarrow{\sim}H^j_{DR}(X_{\hat k_{\sigma_p}})\otimes_{\hat k_{\sigma_p}}\mathbb B_{cris,\hat k_{\sigma_p}}
\end{eqnarray*}
is the canonical filtered isomorphism compatible with the action of $G_p:=Gal(\mathbb C_p/\hat k_{\sigma_p})$ 
and with the action of the Frobenius $\phi_p$.
\end{defi}

We have the following key proposition, we state and prove it
for smooth projective varieties, the case for smooth varieties is obtained in the same way
using a smooth compactification with normal crossing divisors. The projective case suffices for our purpose:

\begin{prop}\label{LatticeLog}
(projective case of \cite{B6} proposition 4(i)).
Let $k$ be a field of finite type over $\mathbb Q$. Let $X\in\PSmVar(k)$.
Let $p\in\mathbb N\backslash a_{O_k}(\delta(k,X)^0)$ be a prime number (see definition \ref{deltakX}). 
Consider an embedding $\sigma_p:k\hookrightarrow\mathbb C_p$.
Denote by $k\subset\hat k_{\sigma_p}\subset\mathbb C_p$ the $p$-adic completion of $k$ with respect to $\sigma_p$.
Consider $X_{\hat k_{\sigma_p}}^{\mathcal O}\in\PSch/O_{\hat k_{\sigma_p}}$ a smooth integral model of $X_{\hat k_{\sigma_p}}$,
i.e. $X_{\hat k_{\sigma_p}}^{\mathcal O}\otimes_{O_{\hat k_{\sigma_p}}}\hat k_{\sigma_p}=X_{\hat k_{\sigma_p}}$
and $X_{\hat k_{\sigma_p}}^{\mathcal O}$ is smooth with smooth special fiber. Denote $G_p:=Gal(\mathbb C_p/\hat k_{\sigma_p})$
Let $j\in\mathbb Z$. We have, see definition \ref{wlogdef}(iii), 
\begin{eqnarray*}
H^jR\alpha(X)(H^j_{et}(X_{\mathbb C_p},\mathbb Z_p)^{G_p})=c^{*-1}H^jOL_{\hat X_{\hat k_{\sigma_p}}}
(\mathbb H^j_{pet}(X_{\hat k_{\sigma_p}},\Omega^{\bullet}_{\hat X^{(p)}_{\hat k_{\sigma_p}},\log,\mathcal O})) 
\subset H^j_{DR}(X_{\hat k_{\sigma_p}})\otimes_{\hat k_{\sigma_p}}\mathbb B_{cris,\hat k_{\sigma_p}},
\end{eqnarray*}
where we recall $c:\hat X_{\hat k_{\sigma_p}}^{\mathcal O}\to X_{\hat k_{\sigma_p}}$ is 
the morphism in $\RTop$ given by the completion of $X_{\hat k_{\sigma_p}}^{\mathcal O}$ with respect to $(p)$, and 
\begin{equation*}
c^*:H^j_{DR}(X_{\hat k_{\sigma_p}})\xrightarrow{\sim}H^j_{DR}(\hat X^{\mathcal O}_{\hat k_{\sigma_p}})\otimes_{O_{\hat k_{\sigma}}}\hat k_{\sigma} 
\end{equation*}
is an isomorphism by formal GAGA for coherent sheaves on $X_{\hat k_{\sigma_p}}^{\mathcal O}$ and the Hodge De Rham spectral sequence.
\end{prop}

\begin{proof} 
Consider $c:\hat X^{\mathcal O}_{\hat k_{\sigma_p}}\to X_{\hat k_{\sigma_p}}^{\mathcal O}$ 
the morphism in $\RTop$ which is the formal completion along the ideal $(p)$.
Take a Zariski or etale cover $r=(r_i:X^{\mathcal O}_i\to X^{\mathcal O})_{1\leq i\leq r}$ such that for each $i$ there exists an etale map
$e_i:X^{\mathcal O}_i\to\mathbb G_m^{\mathcal O,d_{X_i}}\subset\mathbb A_{O_{\hat k_{\sigma_p}}}^{d_{X_i}}$
with $e_i:X^{\mathcal O}_i\to e_i(X^{\mathcal O}_i)$ finite etale. 
Then, by \cite{ChinoisCrys}, we have for each $i$ explicit lifts of Frobenius
\begin{equation*}
\phi^{i,\mathcal O}:X^{\mathcal O}_{i,\hat k_{\sigma_p}}\to X^{\mathcal O}_{i,\hat k_{\sigma_p}},
\phi^i_n:=\phi_i/p^n:X^{\mathcal O}_{i,\hat k_{\sigma_p}}/p^n\to X^{\mathcal O}_{i,\hat k_{\sigma_p}}/p^n 
\end{equation*}
of the Frobenius $\phi_1^i:X^{\mathcal O}_{i,\hat k_{\sigma_p}}/p\to X^{\mathcal O}_{i,\hat k_{\sigma_p}}/p$.
In particular for $n'>n$ the following diagram commutes
\begin{equation*}
\xymatrix{0\ar[r] & O_{X_{i,\hat k_{\sigma_p}}^{\mathcal O}/p^{n'-n}}\ar[r]^{p^n\cdot} & 
O_{X_{i,\hat k_{\sigma_p}}^{\mathcal O}/p^{n'}}\ar[r]^{/p^{n'-n}} &
O_{X_{i,\hat k_{\sigma_p}}^{\mathcal O}/p^n}\ar[r] & 0 \\
0\ar[r] & O_{X_{i,\hat k_{\sigma_p}}^{\mathcal O}/p^{n'-n}}\ar[r]^{p^n\cdot}\ar[u]^{\phi^i_{n'-n}} & 
O_{X_{i,\hat k_{\sigma_p}}^{\mathcal O}/p^{n'}}\ar[r]^{/p^{n'-n}}\ar[u]^{\phi^i_{n'}} &
O_{X_{i,\hat k_{\sigma_p}}^{\mathcal O}/p^n}\ar[u]^{\phi^i_n}\ar[r] & 0}
\end{equation*}
and such that the action of $\phi^i_n$ on $\Omega^{\bullet}_{X_{i,\hat k_{\sigma_p}}^{\mathcal O}/p^n}$ 
is a morphism of complex, i.e. commutes with the differentials. 
Moreover by Dieudonne Cartier (see \cite{Illusie}), we have a canonical morphism in $\Fun(\mathbb N,\PSh(X^{\mathcal O}_{\hat k_{\sigma_p}}))$
\begin{equation*}
JW:=(JW_n:O_{X^{\mathcal O}_{\hat k_{\sigma_p}}/p^n}\xrightarrow{JW_n}W_nO(X^{\mathcal O}_{\hat k_{\sigma_p}}/p))_{n\in\mathbb N}
\end{equation*}
of sheaves of rings whose restriction to $X^{\mathcal O}_{i,\hat k_{\sigma_p}}$ commutes with $\phi^i$, 
that is $\phi^i_n\circ JW_n=JW_n\circ\phi^i_n$ for each $n\in\mathbb N$.
On the other hand, by \cite{Illusie}, we have using proposition \ref{Ilprop}(i) an action of the Frobenius on 
$H^j_{DR}(X_{\hat k_{\sigma_p}})=H^j_{DR}(X^{\mathcal O}_{\hat k_{\sigma_p}})\otimes_{O_{\hat k_{\sigma_p}}}\hat k_{\sigma_p}$ by
\begin{eqnarray*}
\phi:H^j_{DR}(X^{\mathcal O}_{\hat k_{\sigma_p}})\xrightarrow{c^*}H^j_{DR}(\hat X^{\mathcal O}_{\hat k_{\sigma_p}})\xrightarrow{H^jIW}
\mathbb H_{pet}^j(X_{\hat k_{\sigma_p}},W\Omega^{\bullet}_{X^{\mathcal O}_{\hat k_{\sigma_p}}})
\xrightarrow{\phi_{W(X_{\hat k_{\sigma_p}})}^*} \\
\mathbb H_{pet}^j(X_{\hat k_{\sigma_p}},W\Omega^{\bullet}_{X^{\mathcal O}_{\hat k_{\sigma_p}}})
\xrightarrow{H^jIW^{-1}}H^j_{DR}(\hat X^{\mathcal O}_{\hat k_{\sigma_p}})\xrightarrow{c^{*,-1}}H^j_{DR}(X^{\mathcal O}_{\hat k_{\sigma_p}}),
\end{eqnarray*}
where $IW=(IW_n)_{n\in\mathbb N}$ the morphism of complexes in $C(X^{\mathcal O,pet}_{\hat k_{\sigma_p}})$
\begin{equation*}
IW:=(IW_n)_{n\in\mathbb N}:\Omega^{\bullet}_{\hat X^{\mathcal O}_{\hat k_{\sigma_p}}}\xrightarrow{\Omega(JW)}
\Omega^{\bullet}_{W(X^{\mathcal O}_{\hat k_{\sigma_p}}/p)}\to W\Omega^{\bullet}_{X^{\mathcal O}_{\hat k_{\sigma_p}}/p}
\end{equation*}
induced by the morphism $JW$.
We then have the following commutative diagram in $C(X^{\mathcal O,pet}_{\hat k_{\sigma_p}})$, 
where $R:=[1,\ldots,r]$ and $X_R:=X_1\times_X\cdots\times_X X_r$,
\begin{equation}\label{Frob}
\xymatrix{W\Omega^{\bullet}_{X^{\mathcal O}_{\hat k_{\sigma_p}}}\ar[rr]^{r_i^*} & \, &
\oplus_{i=1}^rr_{i*}W\Omega^{\bullet}_{X^{\mathcal O}_{i,\hat k_{\sigma_p}}}\ar[rr]^{r_I^*} & \, & 
r_{R*}W\Omega^{\bullet}_{X^{\mathcal O}_{R,\hat k_{\sigma_p}}} \\
W\Omega^{\bullet}_{X^{\mathcal O}_{\hat k_{\sigma_p}}}\ar[rr]^{r_i^*}\ar[u]^{I-\phi} & \, &
\oplus_{i=1}^rr_{i*}W\Omega^{\bullet}_{X^{\mathcal O}_{i,\hat k_{\sigma_p}}}\ar[rr]^{r_I^*}\ar[u]^{I-\phi^i} & \, &
r_{R*}W\Omega^{\bullet}_{X^{\mathcal O}_{R,\hat k_{\sigma_p}}}\ar[u]^{I-\phi^R} \\
\Omega^{\bullet}_{\hat X^{(p)}_{\hat k_{\sigma_p}},\log,\mathcal O}
\ar[rr]^{r_i^*}\ar[u]^{r^*OL_{\hat X^{\mathcal O}_{\hat k_{\sigma_p}}}\circ IW} & \, &
\oplus_{i=1}^rr_{i*}\Omega^{\bullet}_{\hat X^{(p)}_{i,\hat k_{\sigma_p}},\log,\mathcal O}
\ar[rr]^{r_I^*}\ar[u]^{r^*OL_{\hat X_{i,\hat k_{\sigma_p}}^{\mathcal O,(p)}}\circ IW} & \, &
r_{R*}\Omega^{\bullet}_{\hat X^{(p)}_{R,\hat k_{\sigma_p}},\log,\mathcal O}
\ar[u]^{r^*OL_{\hat X^{\mathcal O,(p)}_{R,\hat k_{\sigma_p}}}\circ IW}}
\end{equation}
and $(I-\phi)\circ r^*OL_{\hat X^{\mathcal O,(p)}_{\hat k_{\sigma_p}}}\circ IW=0$. 
In particular, we get factorization in $C(X^{\mathcal O,pet}_{\hat k_{\sigma_p}})$ and a map in $C(X^{pet}_{\hat k_{\sigma_p}})$ after taking the functor $r^*$
\begin{equation}\label{Wlog}
IW\circ OL_{\hat X^{\mathcal O}_{\hat k_{\sigma_p}}}:\Omega^{\bullet}_{\hat X^{(p)}_{\hat k_{\sigma_p}},\log,\mathcal O}\xrightarrow{IW_{\log}} 
W\Omega^{\bullet}_{X^{\mathcal O}_{\hat k_{\sigma_p}},\log}:= 
\ker(I-\phi:W\Omega^{\bullet}_{X^{\mathcal O}_{\hat k_{\sigma_p}}}\to W\Omega^{\bullet}_{X^{\mathcal O}_{\hat k_{\sigma_p}}})
\xrightarrow{WOL_{\hat X^{\mathcal O}_{\hat k_{\sigma_p}}},\subset}W\Omega^{\bullet}_{X^{\mathcal O}_{\hat k_{\sigma_p}}}.
\end{equation}
By \cite{ChinoisCrys} and proposition \ref{Ilprop}(iii),
$\alpha\in H^j_{et}(X_{\mathbb C_p},\mathbb Z_p)$ is such that $w(\alpha)\in H^j_{DR}(X_{\hat k_{\sigma_p}})$
if and only if  
\begin{eqnarray*}
w(\alpha)\in\ker(I-\phi:H^j_{DR}(X^{\mathcal O}_{\hat k_{\sigma_p}})\to H^j_{DR}(X^{\mathcal O}_{\hat k_{\sigma_p}})) \\
=c^*(H^jIW)^{-1}\Im(H^jWOL_{\hat X^{\mathcal O}_{\hat k_{\sigma_p}}}:\mathbb H_{pet}^j(X^{\mathcal O}_{\hat k_{\sigma_p}},W\Omega^{\bullet}_{X^{\mathcal O}_{\hat k_{\sigma_p}},\log})
\to\mathbb H_{pet}^j(X^{\mathcal O}_{\hat k_{\sigma_p}},W\Omega^{\bullet}_{X^{\mathcal O}_{\hat k_{\sigma_p}}})).
\end{eqnarray*}  
By proposition \ref{Ilprop}(ii)  
the map $IW_{\log}$ (\ref{Wlog}) is an isomorphism after pro-etale sheaftification, in particular it induces isomorphisms for each $j\in\mathbb Z$
\begin{eqnarray*}
OL_{\hat X^{\mathcal O}_{\hat k_{\sigma_p}}}(H^jIW_{\log}):
OL_{\hat X^{\mathcal O}_{\hat k_{\sigma_p}}}(\mathbb H_{pet}^j(X_{\hat k_{\sigma_p}},\Omega^{\bullet}_{\hat X^{(p)}_{\hat k_{\sigma_p}},\log,\mathcal O}))\xrightarrow{\sim} \\
WOL_{\hat X^{\mathcal O}_{\hat k_{\sigma_p}}}(\mathbb H_{pet}^j(X_{\hat k_{\sigma_p}},W\Omega^{\bullet}_{X^{\mathcal O}_{\hat k_{\sigma_p}},\log}))= 
(H^jIW)^*c^*\ker(I-\phi:H^j_{DR}(X^{\mathcal O}_{\hat k_{\sigma_p}})\to H^j_{DR}(X^{\mathcal O}_{\hat k_{\sigma_p}}))
\end{eqnarray*}
This proves the proposition.
\end{proof}

We will use the following lemma
\begin{lem}\label{Aut}
\begin{itemize}
\item[(i)] Let $(K,v)$ a field of characteristic zero endowed with a valuation $v$ (e.g. a non archimedian field). 
Let $(L,v_1,\cdots,v_s)/(K,v)$ where $L/K$ is a finite type extension and $v_1,\cdots,v_s$ is a finite set of valuations of $L$ lifting of $v$ geometrically.
Then there exists a finite extension $L'/L$ such that $Aut_K(L')$ acts transitively on lifts $(v'_1,\cdots,v'_s)$
\item[(ii)] Let $K$ be a $p$-adic field. 
Let $Z^{\mathcal O}\in\Sch^{int}/O_K$, $Z^{\mathcal O}=\cup_{i=1}^rZ^{\mathcal O}_i$ where $(Z^{\mathcal O}_i)_{1\leq i\leq r}$ are the irreducible components of $Z^{\mathcal O}$.
For each $1\leq i\leq r$, $Z_i^{\mathcal O}/p=\cup_{j=1}^{s_i}((Z^{\mathcal O}_i)/p)_j$ 
where $((Z^{\mathcal O}_i/p))_{1\leq j\leq s_i}$ are the irreducible components of $Z_i^{\mathcal O}/p$.
Then there exists an etale morphism  
$r=(r_i):Z^{'\mathcal O}:=\sqcup_{i=1}^rZ_i^{'\mathcal O}\to Z^{\mathcal O}$ in $\Sch^{int}/O_K$ with $r_i:Z_i^{'\mathcal O}\to Z_i^{\mathcal O}$ dominant,
such that $Aut_{O_K}(Z^{'\mathcal O})$ acts transitively on $((Z_i^{'\mathcal O}/p)_j)_{1\leq j\leq s_i}$ for each $1\leq i\leq r$.
\item[(iii)]Let $K$ be a $p$-adic field. Let $X^{\mathcal O}\in\Sch^{int}/O_K$ irreducible and 
$Z^{\mathcal O}\subset X^{\mathcal O}$ be a closed subset of integral models over $O_K$.
Denote by $(Z^{\mathcal O}_i)_{1\leq i\leq r}$ the irreducible components of $Z^{\mathcal O}$.
For each $1\leq i\leq r$, $Z_i^{\mathcal O}/p=\cup_{j=1}^{s_i}((Z^{\mathcal O}_i)/p)_j$ 
where $((Z^{\mathcal O}_i/p)_j)_{1\leq j\leq s_i}$ are the irreducible components of $Z_i^{\mathcal O}/p$.
Then there exists an etale morphism $r_e:X_e^{\mathcal O}\to X^{\mathcal O}$ in $\Sch^{int}/O_K$ with $r_e(X^{\mathcal O}_e)\cap Z_i\neq\emptyset$ 
such that $Aut_{O_K}(X_e^{\mathcal O},Z^{\mathcal O})$ acts transitively on 
$((Z^{\mathcal O}_i/p)_j)_{1\leq j\leq s_i}$, where $Aut_{O_K}(X_e^{\mathcal O},Z^{\mathcal O})$ denote the automorphism $g$ of $X_e^{\mathcal O}$ over $O_K$
such that $g(r_e^{-1}(Z^{\mathcal O}))=r_e^{-1}(Z^{\mathcal O})$.
\end{itemize}
\end{lem}

\begin{proof}
\noindent(i):Follows from the fact that $(L,v_1,\cdots,v_s)$ are the lift of $(K(e_1,\cdots,e_d),v)$ where $v$ is the corresponding geometric valuation, 
where $L/K(e_1,\cdots,e_d)$ is finite with $e_1,\cdots,e_d$ a transcendence basis of $L/K$.

\noindent(ii):Follows from (i) applied to $(K,v)=(K,v_p)$ where $v_p$ is the $p$-adic valuation, 
and $(L,v_1,\cdots,v_s)=(K(Z_i),((Z^{\mathcal O}_i)/p)_1,\cdots,((Z^{\mathcal O}_i)/p)_{s_i})$ for each $1\leq i\leq r$.

\noindent(iii): Follows from (ii) and the lifting property of hensel semi-local rings. 
Indeed there exists a finite subgroup $M\subset Aut_{O_K}(Z^{'\mathcal O})$ which acts transitively on $((Z_i^{'\mathcal O}/p)_j)_{1\leq j\leq s_i}$ for each $1\leq i\leq r$.

\end{proof}

Proposition \ref{LatticeLog} together with proposition \ref{GAGAlog} and lemma \ref{Aut} implies the Tate conjecture for 
smooth projective varieties over fields of finite type over $\mathbb Q$ :

\begin{thm}\label{Tate}
Let $k$ be a field of finite type over $\mathbb Q$. Denote $O_k\subset k$ its ring of integers.
Let $X\in\PSmVar(k)$. Let $p\in\mathbb N\backslash a_{O_k}(\delta(k,X))^0$ be a prime number, 
where $\delta(k,X)^0\subset\Spec(O_k)$ is the finite set given in definition \ref{deltakX} and $a_{O_k}:\Spec(O_k)\to\Spec(\mathbb Z)$ is the canonical map.
Then the Tate conjecture holds for $X$. That is for $d\in\mathbb Z$, the cycle class map
\begin{equation*}
\mathcal Z^d(X)\otimes\mathbb Q_p\to H^{2d}_{et}(X_{\bar k},\mathbb Q_p)(d)^G, \; Z\mapsto[Z]
\end{equation*}
is surjective, where $G:=Gal(\bar k/k)$.
\end{thm}

\begin{proof}
Up to split $X$ by its connected components we may assume that $X$ is connected, hence irreducible, of dimension $d_X$.
Consider an embedding $\sigma_p:k\hookrightarrow\mathbb C_p$.
Then $k\subset\bar k\subset\mathbb C_p$ and $k\subset\hat k_{\sigma_p}\subset\mathbb C_p$,
where $\hat k_{\sigma_p}$ is the $p$-adic field which is the completion of $k$ with respect the $p$ adic norm given by $\sigma_p$. 
Consider $X_{\hat k_{\sigma_p}}^{\mathcal O}\in\PSch^{sm}/O_{\hat k_{\sigma_p}}$ a smooth model of $X_{\hat k_{\sigma_p}}$,
i.e. $X_{\hat k_{\sigma_p}}^{\mathcal O}\otimes_{O_{\hat k_{\sigma_p}}}\hat k_{\sigma_p}=X_{\hat k_{\sigma_p}}$
and $X_{\hat k_{\sigma_p}}^{\mathcal O}$ is smooth connected with smooth special fiber.
Consider $c:\hat X_{\hat k_{\sigma_p}}^{\mathcal O}\to X_{\hat k_{\sigma_p}}^{\mathcal O}$ the morphism of ringed spaces 
which is the formal completion of $X_{\hat k_{\sigma_p}}^{\mathcal O}$ along the ideal generated by $p$.
Denote by $G_p:=Gal(\mathbb C_p/\hat k_{\sigma_p})\subset G:=Gal(\bar k/k)$ the local Galois subgroup.
Let $\alpha\in H_{et}^{2d}(X_{\bar k},\mathbb Z_p)(d)^G$. 
Using definition \ref{walpha}, by proposition \ref{LatticeLog}  
\begin{equation*}
w(\alpha)\in H^{2d}OL_{\hat X^{(p)}}(\mathbb H_{pet}^{2d}(X_{\hat k_{\sigma_p}},\Omega^{\bullet}_{\hat X^{(p)}_{\hat k_{\sigma_p}},\log,\mathcal O}))
\subset H^{2d}_{DR}(\hat X^{\mathcal O}_{\hat k_{\sigma_p}})\otimes_{O_{\hat k_{\sigma_p}}}\hat k_{\sigma_p}=H^{2d}_{DR}(X_{\hat k_{\sigma_p}}).
\end{equation*}
Considering the decomposition
\begin{equation*}
\mathbb H_{pet}^{2d}(X_{\hat k_{\sigma_p}},\Omega^{\bullet}_{\hat X^{(p)}_{\hat k_{\sigma_p}},\log,\mathcal O})
=\oplus_{l=0}^{2d}H_{pet}^{2d-l}(X_{\hat k_{\sigma_p}},\Omega^{l}_{\hat X^{(p)}_{\hat k_{\sigma_p}},\log,\mathcal O}),
\end{equation*}
by proposition \ref{GAGAlog} (ii) and (iii) there exists $Z^{\mathcal O}=\cup_{i=1}^rZ_i^{\mathcal O}\subset X^{\mathcal O}_{\hat k_{\sigma_p}}$ a zariski closed subset 
whose irreducible components $(Z_i^{\mathcal O})_{1\leq i\leq r}$ 
are of codimension $d$ and (proper) surjective over the integral ring $O_{\hat k_{\sigma_p}}$ of $\hat k_{\sigma_p}$ such that
\begin{equation*}
w(\alpha)\in\Im(\Omega(\gamma^{\vee}_{Z}):H^{2d}_{DR,\hat Z^{\mathcal O}}(\hat X^{\mathcal O}_{\hat k_{\sigma_p}})\to 
H^{2d}_{DR}(\hat X^{\mathcal O}_{\hat k_{\sigma_p}})\otimes_{O_{\hat k_{\sigma_p}}}\hat k_{\sigma_p}=H_{DR}^{2d}(X_{\hat k_{\sigma_p}})). 
\end{equation*}
Note that the $Z_i^{\mathcal O}$ have bad reduction in general and that the number of irreducible components of $Z^{\mathcal O}/p$ is greater then the
number of irreducible components of $Z^{\mathcal O}$ in general. We get
\begin{equation*}
w(\alpha)=\sum_{i=1}^r\sum_{j=1}^{s_i}n_{ij}[(Z_i^{\mathcal O}/p)_j]\in 
H^{2d}_{DR}(\hat X^{\mathcal O}_{\hat k_{\sigma_p}})\otimes_{O_{\hat k_{\sigma_p}}}\hat k_{\sigma_p}=H_{DR}^{2d}(X_{\hat k_{\sigma_p}}), 
\end{equation*}
where $n_{ij}\in\mathbb Q_p$ and $((Z_i^{\mathcal O}/p)_j)_{1\leq j\leq s_i}$ are the irreducible components of $Z_i^{\mathcal O}/p$, 
which are all of codimension $d$ in $X^{\mathcal O}_{\hat k_{\sigma_p}}/p$ since the reduction mod $p$ consists on one equation. 
Since $Z^{\mathcal O}_i$ has bad reduction mod $p$ in general, in particular $Z_i^{\mathcal O}/p$ may be reducible while $Z_i^{\mathcal O}$ is irreducible.
By lemma \ref{Aut}(iii), there exists an etale morphism $r_e:X^{\mathcal O}_e\to X^{\mathcal O}_{\hat k_{\sigma_p}}$ in $\Sch^{int}/O_{\hat k_{\sigma_p}}$ 
such that for each $1\leq i\leq r$ we have $r_e(X^{\mathcal O}_e)\cap Z_i\neq\emptyset$ and $Aut_{O_{\hat k_{\sigma_p}}}(X^{\mathcal O}_e,Z^{\mathcal O})$ acts transitively
on $(T_{ij}:=r_e^{-1}((Z_i^{\mathcal O}/p)_j))_{1\leq j\leq s_i}$. Recall $Aut_{O_{\hat k_{\sigma_p}}}(X^{\mathcal O}_e,Z^{\mathcal O})$ denote the automorphisms $g$ of 
$X^{\mathcal O}_e$ over $O_{\hat k_{\sigma_p}}$ such that $g(r_e^{-1}(Z^{\mathcal O}))=r_e^{-1}(Z^{\mathcal O})$.
Take a compactification $\bar r_e:\bar X^{\mathcal O}_e\to X^{\mathcal O}_{\hat k_{\sigma_p}}$ of $r_e$ which is a finite surjective morphism of integral models.
Consider then the finite surjective morphism in $\Sch^{int}/O_{\hat k_{\sigma_p}}$ which is the composite
\begin{equation*}
r':X^{'\mathcal O}\xrightarrow{\epsilon^{\mathcal O}}\bar X^{\mathcal O}_e\xrightarrow{\bar r_e}X^{\mathcal O}_{\hat k_{\sigma_p}},
\end{equation*}
where $\epsilon:(X',E)\to(\bar X_e,\mathcal R)$ is a desingularization of the pair $(\bar X_e,\mathcal R)$ of projective varieties over $\hat k_{\sigma_p}$,
where $\mathcal R$ is the ramification locus of $\bar r_e\otimes_{O_{\hat k_{\sigma_p}}}\hat k_{\sigma_p}$.
We have then, considering $Z^{\mathcal O}/p=\cup_{i=1}^r\cup_{j=1}^{s_i}(Z_i^{\mathcal O}/p)_j$,
\begin{equation*}
Z^{'\mathcal O}:=r^{'-1}(Z^{\mathcal O}), \; Z^{'\mathcal O}/p=\cup_{i=1}^r\cup_{j=1}^{s_i}T_{ij}, \; T_{ij}:=r^{'-1}((Z_i^{\mathcal O}/p)_j).
\end{equation*}
Note that $X'$ has bad reduction in general. In particular $X^{'\mathcal O}$ is NOT smooth. 
Consider the factorization $r':X^{'\mathcal O}\xrightarrow{i}\mathbb P^N\times X^{\mathcal O}\xrightarrow{p}X^{\mathcal O}$ 
where $i$ is the graph closed embedding and $p$ is the projection. 
Denote by $C^{\mathcal O}:=C_{X^{'\mathcal O}/\mathbb P^N\times X^{\mathcal O}}\xrightarrow{p'}X^{'\mathcal O}$ the normal cone of $i$.
Since $X'$ is smooth $C:=C^{\mathcal O}\otimes_{O_{\hat k_{\sigma_p}}}\hat k_{\sigma_p}=N_{X'/\mathbb P^N\times X}\xrightarrow{p'}X'$ is a vector bundle.
We denote again by $c:\widehat{\mathbb P^N\times X^{\mathcal O}}\to \mathbb P^N\times X^{\mathcal O}$, $c:\hat C^{\mathcal O}\to C^{\mathcal O}$ the formal completion maps.
In particular $C^{\mathcal O}$ is not smooth but $C$ is smooth. 
Let $k\subset k'\subset\mathbb C_p$ be a subfield of finite type over $\mathbb Q$ over which 
$(Z_i\subset X_{\hat k_{\sigma_p}})_{1\leq i\leq r}$ and $r':X'\to X_{\hat k_{\sigma_p}}$ are defined.
Take an (algebraic) embedding of field $k'(X')\hookrightarrow\mathbb C_p\simeq\mathbb C$. Then 
\begin{equation*}
A:=Aut_{O_{\hat k_{\sigma_p}}}(X^{'\mathcal O},Z^{\mathcal O})\subset Aut_{k'}(k'(X'))\subset Aut_{k'}(\mathbb C_p)\subset G=Aut_k(\mathbb C_p)
\end{equation*}
Note that $\hat k_{\sigma_p}(X')$ is NOT a p-adic field and $\hat k_{\sigma_p}(X')\hookrightarrow\mathbb C_p$.
We have then the following commutative diagram
\begin{equation*}
\xymatrix{H_{et}^{2d}(C_{\bar k},\mathbb Q_p)(d)^{G_p}\ar[rr]^{\alpha\mapsto w(\alpha)} & \, &
H_{DR}^{2d}(C)\ar[r]^{c^*} & H_{DR}^{2d}(\hat C^{\mathcal O})\otimes_{O_{\hat k_{\sigma_p}}}\hat k_{\sigma_p}  \\
H_{et}^{2d}(X_{\bar k},\mathbb Q_p)(d)^{G_p}\ar[rr]^{\alpha\mapsto w(\alpha)}\ar[u]^{p^{'*}r^{'*}} & \, &
H_{DR}^{2d}(X_{\hat k_{\sigma_p}})\ar[r]^{c^*}\ar[u]^{p^{'*}r^{'*}} & H_{DR}^{2d}(\hat X^{\mathcal O})\otimes_{O_{\hat k_{\sigma_p}}}\hat k_{\sigma_p}
\ar[u]^{p^{'*}r^{'*}}}
\end{equation*}
whose arrows of the upper row commute with the action of $Aut_{O_{\hat k_{\sigma_p}}}(C^{\mathcal O})\subset G=Aut_k(\mathbb C_p)$, in particular with the action of
\begin{equation*}
A:=Aut_{O_{\hat k_{\sigma_p}}}(X^{'\mathcal O},Z^{\mathcal O})\subset Aut_{O_{\hat k_{\sigma_p}}}(C^{\mathcal O})\subset G=Aut_k(\mathbb C_p). 
\end{equation*}
The first arrows of the rows are the maps of the p-adic De Rham comparison theorem for $X_{\mathbb C_p}$ and $C_{\mathbb C_p}$ respectively ($C$ and $X$ are smooth).
Since $\alpha\in H_{et}^{2d}(X_{\bar k},\mathbb Q_p)(d)^G$, we have 
$p^{'*}r^{'*}\alpha\in H_{et}^{2d}(C_{\bar k},\mathbb Q_p) (d)^A$ for each $1\leq i\leq r$.
We have then
\begin{itemize}
\item $p^{'*}r^{'*}w(\alpha)=c^*w(p^{'*}r^{'*}\alpha)\in(H_{DR}^{2d}(\hat C^{\mathcal O})\otimes_{O_{\hat k_{\sigma_p}}}\hat k_{\sigma_p})^A$
\item and
\begin{equation*}
p^{'*}r^{'*}w(\alpha)=\sum_{i=1}^r\sum_{j=1}^{s_i}n_{ij}[T_{ij}]\in
\Im(H_{DR,\widehat{p^{'-1}(Z^{'\mathcal O})}}^{2d}(\hat C^{\mathcal O})\to H_{DR}^{2d}(\hat C^{\mathcal O})\otimes_{O_{\hat k_{\sigma_p}}}\hat k_{\sigma_p}), 
\end{equation*}
since 
\begin{equation*}
w(\alpha)=\sum_{i=1}^r\sum_{j=1}^{s_i}n_{ij}[(Z_i^{\mathcal O}/p)_j]\in 
H^{2d}_{DR}(\hat X^{\mathcal O}_{\hat k_{\sigma_p}})\otimes_{O_{\hat k_{\sigma_p}}}\hat k_{\sigma_p}=H_{DR}^{2d}(X_{\hat k_{\sigma_p}}), 
\end{equation*}
\end{itemize}
Take for each $1\leq i\leq r$, $1\leq s'_i\leq s_i$ such that
\begin{equation*}
([(Z_i^{\mathcal O}/p)_j])_{1\leq i\leq r,1\leq j\leq s'_i}\in H^{2d}_{DR}(\hat X^{\mathcal O}_{\hat k_{\sigma_p}})
\otimes_{O_{\hat k_{\sigma_p}}}\hat k_{\sigma_p}=H_{DR}^{2d}(X_{\hat k_{\sigma_p}})
\end{equation*}
are linearly independent. Hence 
\begin{equation*}
([T_{ij}])_{1\leq i\leq r,1\leq j\leq s'_i}\in H_{DR}^{2d}(\hat C^{\mathcal O})\otimes_{O_{\hat k_{\sigma_p}}}\hat k_{\sigma_p}
\end{equation*}
are linearly independent since
\begin{eqnarray*}
r'_*([T_{ij}])=[(Z_i^{\mathcal O}/p)_j], \; 
r'_*:H_{DR,\hat Z^{'\mathcal O}}^{2d}(\hat C^{\mathcal O})\to H_{DR,\hat Z^{\mathcal O}}^{2d}(\hat X^{\mathcal O})\to H_{DR}^{2d}(\hat X^{\mathcal O}).
\end{eqnarray*}
We get, since $A$ acts transitively on $(T_{ij}:=r^{'-1}((Z_i^{\mathcal O}/p)_j))_{1\leq j\leq s_i}$ for each $1\leq i\leq r$,
\begin{equation*}
n_{ij}=n_i, \; \mbox{for each} \; 1\leq i\leq r \; \mbox{and all} \; 1\leq j\leq s_i.
\end{equation*}
This gives 
\begin{eqnarray*}
w(\alpha)=\sum_{i=1}^rn_i[Z_i^{\mathcal O}/p]=\sum_{i=1}^rn_i[Z_i^{\mathcal O}]\in 
H^{2d}_{DR}(\hat X^{\mathcal O}_{\hat k_{\sigma_p}})\otimes_{O_{\hat k_{\sigma_p}}}\hat k_{\sigma_p}=H_{DR}^{2d}(X_{\hat k_{\sigma_p}}), 
\end{eqnarray*}
that is
\begin{equation*}
w(\alpha)=[Z]\in H_{DR}^{2d}(X_{\hat k_{\sigma_p}}), \; Z:=\sum_{i=1}^rn_i[Z_i]\in\mathcal Z^d(X_{\hat k_{\sigma_p}})\otimes\mathbb Q_p.
\end{equation*} 
Note that the local Galois group $G_p:=Gal(\mathbb C_p/\hat k_{\sigma_p})\subset G$ 
fixes the components $(T_{ij})_{1\leq j\leq s_i}$ of $r^{'-1}(Z^{\mathcal O}_i/p)$ for each $1\leq i\leq r$,
since classes of formal algebraic cycles are $G_p$ invariant. Here is where we need that the class $\alpha$ is invariant by the absolute Galois group $G$ of $k$,
not just invariant by $G_p$. 
Since the Hilbert schemes of $X$ are defined over $k$, we get 
\begin{equation*}
w(\alpha)=[Z]=[Z^N]\in H_{DR}^{2d}(X_{\hat k_{\sigma_p}}), \; Z^N\in\mathcal Z^d(X_{\bar k})\otimes\mathbb Q_p.
\end{equation*} 
By the $p$-adic crystalline or de Rham comparison isomorphism, we get $\alpha=[Z^N]\in H_{et}^{2d}(X_{\bar k},\mathbb Z_p)$.
Indeed, we have the corresponding commutative diagram of abelian groups whose rows are exact, 
where $j:X\backslash|Z^N|\hookrightarrow X$ denote the open embedding,
\begin{equation*}
\xymatrix{H_{et,Z^N}^{2d}(X_{\bar k},\mathbb Z_p)\otimes_{\mathbb Z_p}\mathbb B_{dr}
\ar[rrr]^{B_{dr,X}(\gamma^{\vee}_{Z^N})}\ar[d]^{Ra_{X*}\Gamma_{Z'}\alpha(X)} & \, & \, &
H_{et}^{2d}(X_{\bar k},\mathbb Z_p)\otimes_{\mathbb Z_p}\mathbb B_{dr}\ar[r]^{j^*}\ar[d]^{R\alpha(X)} &
H_{et}^{2d}((X\backslash Z^N)_{\bar k},\mathbb Z_p)\otimes_{\mathbb Z_p}\mathbb B_{dr}\ar[d]^{R\alpha(X\backslash Z)} \\
H^{2d}_{DR,Z^N_{\hat k_{\sigma_p}}}(X_{\hat k_{\sigma_p}})\otimes_{\hat k_{\sigma_p}}\mathbb B_{dr}
\ar[rrr]^{DR(X)(OB_{dr,X})(\gamma^{\vee}_{Z^N})} & \, & \, &
H^{2d}_{DR}(X_{\hat k_{\sigma_p}})\otimes_{\hat k_{\sigma_p}}\mathbb B_{dr}\ar[r]^{j^*} &
H^{2d}_{DR}((X\backslash Z^N)_{\hat k_{\sigma_p}})\otimes_{\hat k_{\sigma_p}}\mathbb B_{dr}}.
\end{equation*}
Since $[Z^N]$ is $G$ invariant,
\begin{equation*}
Z^{N'}:=1/\#(gZ^N,g\in G)\sum_{g\in G/G_{Z^N}}gZ^N \in\mathcal Z^d(X)\otimes\mathbb Q_p
\end{equation*}
satisfy $[Z^{N'}]=[Z^N]=\alpha\in H_{et}^{2d}(X_{\bar k},\mathbb Z_p)$.
\end{proof}

\begin{rem}\label{RasEx}
Let $K$ be a $p$-adic field. Denote $O_K\subset K$ its ring of integers and $G_p:Gal(\mathbb C_p/K)$.
In the example of \cite{Liedtke}, we have two ellpitic curves $E$ and $E'$ over $K$ 
and smooth integral models $E^{\mathcal O}$ and $E^{'\mathcal O}$ over $O_K$ of $E$ and $E'$ respectively such that $E^{\mathcal O}/p=E^{'\mathcal O}/p$ 
and $E$ have CM but $E$ is without CM. Hence $E$ and $E'$ are not isogenius.
They consider a canonical class 
\begin{equation*}
\alpha=I_{E^{\mathcal O}/p}\neq 0\in\Hom_{\Mod(G_p)}(V_p(E),V_p(E'))\subset H^2(E_{\mathbb C_p}\times E'_{\mathbb C_p},\mathbb Q_p)^{G_p}
\end{equation*}
given by the identity of $E^{\mathcal O}/p$ and the splitting of $V_p(E)$.
Take $T^{\mathcal O}\subset E^{\mathcal O}\times_{O_K} E^{'\mathcal O}$ irreducible of codimension one such that
\begin{equation*}
T^{\mathcal O}\cap(E^{\mathcal O}\times_{O_K}E^{'\mathcal O})/p=\Delta(E^{\mathcal O}/p)\cup L^p
\end{equation*}
where 
\begin{equation*}
\Delta(E^{\mathcal O}/p)\subset(E^{\mathcal O}\times_{O_K}E^{'\mathcal O})/p=E^{\mathcal O}/p\times_{O_K/p}E^{\mathcal O}/p
\end{equation*}
is the diagonal. Denote $T:=T^{\mathcal O}\cap(E\times E')$
We have 
\begin{equation*}
\alpha=[\Delta(E^{\mathcal O}/p)]\subset\Im(H^2_{DR,\hat T^{\mathcal O}}(\widehat{E^{\mathcal O}\times_{O_K}E^{'\mathcal O}})\to 
H^2_{DR,\hat T^{\mathcal O}}(\widehat{E^{\mathcal O}\times_{O_K}E^{'\mathcal O}})\otimes_{O_K}K=H^2_{DR}(E\times E')
\end{equation*}
that is $\alpha$ is the class of the formal cycle 
\begin{equation*}
\Delta(E^{\mathcal O}/p)\in\mathcal Z^1((E^{\mathcal O}\times_{O_K}E^{'\mathcal O})/p)
\end{equation*}
but is NOT the class of an algebraic cycle. Also note that if $k$ is a field of finite type over $\mathbb Q$ over which $E$ and $E'$ are defined and
$f:E_k^{\mathcal O_k}\times_{O_k}E_k^{'\mathcal O_k}\to\Spec(O_k)$ is an integral model over $O_k$, with $E=E_k\otimes_kK$ and $E'=E'_k\otimes_k K$,
\begin{itemize}
\item $p\in\delta(k,E_k\times E_k')$, where $\delta(k,E_k\times E_k')\backslash\delta(k,E_k\times E_k')^0$ is the set of $p$ which has non generic Neron-Severi group,
which known to be a finite set for divisor (unknown in codimension greater then one),  
\item for $l\in\Spec(O_k)\backslash\delta(k,E_k\times E_k')$, $E_k^{\mathcal O_k}/l$ and $E_k^{'\mathcal O_k}/l$ are not isogenius,
\item the subspaces $H^2(E_{\bar k}\times E'_{\bar k},\mathbb Q_p)^G\subset H^2(E_{\bar k}\times E'_{\bar k},\mathbb Q_p)$, 
and $H^2(E_{\bar k}\times E'_{\bar k},\mathbb Q_l)^G\subset H^2(E_{\bar k}\times E'_{\bar k},\mathbb Q_l)$, $l\neq p$ a prime number, 
where $G:=Gal(\bar k/k)$, are of rank two.
\end{itemize}
\end{rem}

It is well known that theorem \ref{Tate} implies the following :

\begin{cor}\label{TateCor}
Let $X\in\PSmVar(\mathbb C)$. 
\begin{itemize}
\item[(i)] The standard conjectures holds for $X$.
\item[(ii)] Let $k\subset\mathbb C$ be a subfield of finite type over $\mathbb Q$ over which $X$ is defined,
that is $X\simeq X_k\otimes_k\mathbb C$ in $\PSmVar(\mathbb C)$ with $X_k\in\PSmVar(k)$. 
For $\theta\in Aut(\mathbb C/\mathbb Q)$, we have the isomorphism $\theta:X\xrightarrow{\sim}X_{\theta}$ in $\Sch$.
Then $\theta\in Aut(\mathbb C/\mathbb Q)$ induces, for each $j\in\mathbb Z$, an isomorphism
\begin{equation*}
\theta:H_{\sing}^j(X^{an},\mathbb C)\xrightarrow{H^jev(X)^{-1}}H^j_{DR}(X)\xrightarrow{\theta^*}
H^j_{DR}(X_{\theta})\xrightarrow{H^jev(X_{\theta})}H_{\sing}^j(X^{an}_{\theta},\mathbb C).
\end{equation*}
Let $d\in\mathbb N$. Let $\alpha\in F^dH^{2d}(X^{an},\mathbb Q)$ where
$F^dH^{2d}(X^{an},\mathbb Q):=F^dH^{2d}_{DR}(X)\cap H^{2d}_{\sing}(X^{an},\mathbb Q)\subset H^{2d}(X^{an},\mathbb C)$.
If $\theta(\alpha)\in H^{2d}_{\sing}(X^{an}_{\theta},\mathbb Q)$ for all $\theta\in Aut(\mathbb C/k)$, 
then $\alpha=[Z]$ with $Z\in\mathcal Z^d(X)$.
\item[(ii)']Let $\alpha\in H^{2d}(X^{an},\mathbb Q)$.
If $\theta(\alpha)\in H^{2d}_{\sing}(X^{an}_{\theta},\mathbb Q)$ for all $\theta\in Aut(\mathbb C/k)$, 
then $\alpha=[Z]$ with $Z\in\mathcal Z^d(X)$. In particular $\alpha\in F^dH^{2d}(X^{an},\mathbb Q)$.
\end{itemize}
\end{cor}

\begin{proof} 
Standard. We consider, for $j\in\mathbb Z$, $k'\subset\mathbb C$ a subfield and $Y\in\Var(k')$, the canonical morphism
\begin{equation*}
T(Y):=T^j(Y):H^i_{\sing}(Y_{\mathbb C}^{an},\mathbb Q)\xrightarrow{(/p^n)_{n\in\mathbb N}}
H^j_{\sing}(Y_{\mathbb C}^{an},\mathbb Q_p)
\xrightarrow{\sim}H_{et}^j(Y_{\mathbb C},\mathbb Q_p)\xrightarrow{\sim}H_{et}^j(Y_{\bar k'},\mathbb Q_p).
\end{equation*}
For $Y\in\Var(k')$ and $\theta\in Aut(\mathbb C/k')$, 
we have the commutative diagrams in $C(Y\times\mathbb N)$
\begin{eqnarray*}
\xymatrix{\An_*E_{usu}(\mathbb Z)\ar[r]^{(/p^n)_{n\in\mathbb N}}\ar[d]^{\ad(\theta^*,\theta_*)(-)} & 
\An_*E_{usu}(\mathbb Z_p)\ar[d]^{\ad(\theta^*,\theta_*)(-)} & 
E_{et}\mathbb Z_{p,Y_{\mathbb C}}\ar[l]_{\ad(\An^*,\An_*)(\mathbb Z_p)}\ar[d]^{\ad(\theta^*,\theta_*)(-)} \\
\theta_*\theta^*\An_*E_{usu}(\mathbb Z)\ar[r]^{(/p^n)_{n\in\mathbb N}} & \theta_*\theta^*\An_*E_{usu}(\mathbb Z_p) & 
\theta_*\theta^*E_{et}\mathbb Z_{p,Y_{\mathbb C}}=E_{et}\mathbb Z_{p,Y_{\mathbb C,\theta}}
\ar[l]_{\theta_*\theta^*\ad(\An^*,\An_*)(\mathbb Z_p)}\ar[ld]^{\ad(\An_{\theta}^*,\An_{\theta*})(\mathbb Z_p)} \\
\An_{\theta*}E_{usu}(\mathbb Z)\ar[r]^{(/p^n)_{n\in\mathbb N}} & \An_{\theta*}E_{usu}(\mathbb Z_p) &} 
\end{eqnarray*}
and
\begin{eqnarray*}
\xymatrix{\An_*E_{usu}(\mathbb Z)
\ar[r]^{\iota_{2i\pi\mathbb Z/\mathbb C}(Y_{\mathbb C}^{an})}\ar[d]^{\ad(\theta^*,\theta_*)(-)} & 
\An_*E_{usu}(\Omega^{\bullet}_{Y_{\mathbb C}^{an}})\ar[d]^{\ad(\theta^*,\theta_*)(-)} & 
E_{et}\Omega^{\bullet}_{Y_{\mathbb C}}\ar[l]_{\Omega(\An)}\ar[d]^{\ad(\theta^*,\theta_*)(-)} \\
\theta_*\theta^*\An_*E_{usu}(\mathbb Z)\ar[r]^{\iota_{2i\pi\mathbb Z/\mathbb C}(Y_{\mathbb C}^{an})} & 
\theta_*\theta^*\An_*E_{usu}(\Omega^{\bullet}_{Y_{\mathbb C}^{an}}) & 
\theta_*\theta^*E_{et}\Omega^{\bullet}_{Y_{\mathbb C}}=E_{et}\Omega^{\bullet}_{Y_{\mathbb C,\theta}}
\ar[l]_{\theta_*\theta^*\Omega(\An)}\ar[ld]^{\Omega(\An_{\theta})} \\
\An_{\theta*}E_{usu}(\mathbb Z)\ar[r]^{\iota_{2i\pi\mathbb Z/\mathbb C}(Y_{\mathbb C,\theta}^{an})} &
 \An_{\theta*}E_{usu}(\Omega^{\bullet}_{Y^{an}_{\mathbb C,\theta}}) &}
\end{eqnarray*}
where $\An:Y_{\mathbb C}^{an}\to Y$ and $\An_{\theta}:Y_{\mathbb C,\theta}^{an}\to Y$ are the analytical functors, 
usu denote the usual complex topology and $E_{usu}$ is the canonical flasque resolution (see section 2.1).
Hence for $Y\in\SmVar(k')$ and $\alpha\in H^j_{\sing}(Y^{an}_{\mathbb C},\mathbb Q)$ such that 
$\theta(\alpha)\in H^j_{\sing}(Y^{an}_{\mathbb C,\theta},\mathbb Q)$ for some $\theta\in Aut(\mathbb C/k')$, where
\begin{equation*}
\theta:H_{\sing}^j(Y_{\mathbb C}^{an},\mathbb C)\xrightarrow{H^jev(Y)^{-1}}H^j_{DR}(Y_{\mathbb C})\xrightarrow{\theta^*}
H^j_{DR}(Y_{\mathbb C,\theta})\xrightarrow{H^jev(Y_{\mathbb C,\theta})}H_{\sing}^j(Y^{an}_{\mathbb C,\theta},\mathbb C),
\end{equation*}
we have
\begin{equation*}
\theta^*T(Y)(\alpha)=T(Y_{\theta})(\theta(\alpha))\in H^j_{et}(Y_{\bar k'},\mathbb Q_p).
\end{equation*}

\noindent(i): Let $k\subset\mathbb C$ be a subfield of finite type over $\mathbb Q$ such that $X$ is defined.
Denote $G=Gal(\bar k/k)$ the Galois group. Denote $d_X=\dim(X)$ and for short $X^{an}=X$.
Let $L^j\in H^{2d_X-2j}_{\sing}(X\times X,\mathbb Q)$ 
the class inducing the cup product with of the intersection of $j$ hyperplane sections on $H^{d_X-j}_{\sing}(X,\mathbb Q)$
and zero on $H^i_{\sing}(X,\mathbb Q)$, $i\neq d_X-j$.
Let $\Lambda^j\in H^{2d_X+2j}_{\sing}(X\times X,\mathbb Q)$ inducing the inverse of $L_j$ on 
$H^{d_X+j}_{\sing}(X,\mathbb Q)$ and zero on $H^i_{\sing}(X,\mathbb Q)$, $i\neq d_X+j$.
Consider an isomorphism $\sigma_p:\mathbb C\xrightarrow{\sim}\mathbb C_p$ such that 
$\hat k_{\sigma_p}\subset\mathbb C_p$ is unramified and $X$ has good reduction.
We have $T(X\times X)(L^j)\in H^{2d_X-2j}_{et}((X\times X)_{\bar k},\mathbb Q_p)^G$
since it is the class of an algebraic cycle. Hence, 
\begin{equation*}
T(X\times X)(\Lambda^j)\in H^{2d_X+2j}_{et}((X\times X)_{\bar k},\mathbb Q_p)(d_X+j)^G
\end{equation*}
since the inverse of a Galois invariant morphism is Galois invariant.
Hence, by theorem \ref{Tate}, 
\begin{equation*}
\Lambda^j=[Z], \; \mbox{with} \; Z\in\mathcal Z^{d_X+j}(X\times X)\otimes\mathbb Q.
\end{equation*}

\noindent(ii) and(ii)': As $\alpha\in H^{2d}\Gamma(X_{\bullet},\Omega_X)$, where $X=\cup_{i=1}^sX_i$ is an open affine cover,
there exists a subfield $k\subset k'\subset\mathbb C$ of finite type over $\mathbb Q$ over which $\alpha$ is defined,
that is $\alpha\in H^{2d}_{DR}(X_{k'})\subset H^{2d}_{DR}(X)$, 
where $X_{k'}:=X_k\otimes_kk'\in\PSmVar(k')$ satisfy $X_{k'}\otimes_{k'}\mathbb C\simeq X$ in $\PSmVar(\mathbb C)$.
Let $p\in\mathbb N\backslash a_{O_{k'}}(\delta(k',X_{k'}))^0$ be a prime number. Consider 
\begin{equation*}
T(X_{k'})(\alpha)\in H^{2d}(X_{\bar k'},\mathbb Q_p). 
\end{equation*}
If $\theta(\alpha)\in H^{2d}_{\sing}(X^{an}_{\theta},\mathbb Q)$ for all $\theta\in Aut(\mathbb C/k')$, we get
\begin{equation*}
T(X_{k'})(\alpha)=1/\#(Aut(\mathbb C/k')\alpha)
\sum_{\theta\in Aut(\mathbb C/k')}\theta^{-1,*}T(X_{k',\theta})(\theta(\alpha))
\in H_{et}^{2d}(X_{\bar k'},\mathbb Q_p)=H_{et}^{2d}(X,\mathbb Q_p),
\end{equation*} 
hence, for $g\in Gal(\bar k'/k')$, we get
\begin{eqnarray*}
g\cdot T(X_{k'})(\alpha)&=&1/\#(Aut(\mathbb C/k')\alpha)
\sum_{\theta\in Aut(\mathbb C/k')}g\cdot\theta^{-1,*}T(X_{k',\theta})(\theta(\alpha)) \\
&=&1/\#(Aut(\mathbb C/k')\alpha)
\sum_{\theta\in Aut(\mathbb C/k')}\theta^{-1,*}T(X_{k',\theta})(\theta(\alpha))
=T(X_{k'})(\alpha)\in H_{et}^{2d}(X_{\bar k'},\mathbb Q_p),
\end{eqnarray*} 
that is, $T(X_{k'})(\alpha)\in H^{2d}(X_{\bar k'},\mathbb Q_p)(d)^G$, with $G:=Gal(\bar k'/k)$.
Hence, if $\theta(\alpha)\in H^{2d}_{\sing}(X^{an}_{\theta},\mathbb Q)$ for all $\theta\in Aut(\mathbb C/k)$, 
then in particular $\theta(\alpha)\in H^{2d}_{\sing}(X^{an}_{\theta},\mathbb Q)$ for all $\theta\in Aut(\mathbb C/k')$
and we get by theorem \ref{Tate}, 
\begin{equation*}
T(X)(\alpha)=T(X_{k'})(\alpha)=[Z], \; \mbox{with} \; Z\in\mathcal Z^d(X_{k'})\otimes\mathbb Q\subset\mathcal Z^d(X)\otimes\mathbb Q.
\end{equation*}
\end{proof}

We get the following. See \cite{Voisin} for conditions for families where the hypothesis of the theorem holds.

\begin{thm}
Let $X\in\PSmVar(\mathbb C)$, $X=V(I)\subset\mathbb P^N_{\mathbb C}$. 
Consider the canonical deformation 
$f:\mathcal X=V(\tilde I)\subset\mathbb P^N_{\mathbb Q}\times S\to S\subset\bar S$,
with $S\subset\bar S$ the open subset over which $f$ is smooth, where
$\bar S:=\bar{0_X}^{\mathbb Q}\subset\mathbb A^r_{\mathbb Q}$,
$\mathcal X\in\SmVar(\mathbb Q)$, and $X=\mathcal X_s$ with $s:=0_X\in S_{\mathbb C}$. 
Denote $(E^{2d}_{DR}(\mathcal X/S),F):=H^{2d}\int_f(O_{\mathcal X},F_b)\in\Vect_{\mathcal Dfil}(S)$ and
\begin{equation*}
HL^{d,2d}(\mathcal X/S):=
F^dE^{2d}_{DR}(\mathcal X_{\mathbb C}/S_{\mathbb C})\cap f_*ev(\mathcal X)^{-1}(R^{2d}f_*\mathbb Q_{\mathcal X_{\mathbb C}^{an}})
\subset E^{2d}_{DR}(\mathcal X_{\mathbb C}/S_{\mathbb C})
\end{equation*}
the locus of Hodge classes. Let $\lambda\in F^dH^{2d}(X^{an},\mathbb Q)$. 
If the irreducible components $W\subset HL^{d,2d}(\mathcal X/S)$ of the locus of Hodge classes such that 
$\lambda\in W$ are defined over $\bar{\mathbb Q}$ and if 
their Galois conjugates $\sigma(W)$ with $\sigma\in Gal(\bar{\mathbb Q}/\mathbb Q)$ are also components of $HL^{d,2d}(\mathcal X/S)$,
then $\lambda=[Z]\in H^{2d}(X^{an},\mathbb Q)$ with $Z\in\mathcal Z^d(X,\mathbb Q)$. 
\end{thm}

\begin{proof}
We have $\lambda\in\cup_{i=1}^sW_i$ where $W_i\subset HL^{d,2d}(\mathcal X/S)$ 
are the irreducible components passing through $\lambda$.
If the irreducible components $W\subset HL^{d,2d}(\mathcal X/S)$ of the locus of Hodge classes such that 
$\lambda\in W$ are defined over $\bar{\mathbb Q}$ and if 
their Galois conjugates $\sigma(W)$ with $\sigma\in Gal(\bar{\mathbb Q}/\mathbb Q)$ are also components of $HL^{d,2d}(\mathcal X/S)$, 
we get 
\begin{equation*}
\bar\lambda^{\mathbb Q}\subset\cup_{i=1}^s\cup_{\sigma}\sigma(W_i)\subset E^{2d}_{DR}(\mathcal X_{\mathbb C}/S_{\mathbb C}) 
\end{equation*}
since $\cup_{i=1}^s\cup_{\sigma}(W_i)$ is then defined over $\mathbb Q$. Since 
\begin{equation*}
\bar\lambda^{\mathbb Q}=\overline{\left\{\theta^*\lambda,\theta\in Aut(\mathbb C,\mathbb Q)\right\}}\subset E^{2d}_{DR}(\mathcal X_{\mathbb C}/S_{\mathbb C})
\end{equation*}
we get 
\begin{equation*}
\overline{\left\{\theta^*\lambda,\theta\in Aut(\mathbb C,\mathbb Q)\right\}}\subset HL^{d,2d}(\mathcal X/S).
\end{equation*}
In particular, for all $\theta\in Aut(\mathbb C,\mathbb Q)$, $\theta(\lambda)\in H^{2d}_{\sing}(X_{\theta}^{an},\mathbb Q)$.
Hence, by corollary \ref{TateCor}(ii) $\lambda=[Z]\in H^{2d}(X^{an},\mathbb Q)$ with $Z\in\mathcal Z^d(X,\mathbb Q)$.
\end{proof}


LAGA UMR CNRS 7539 \\
Universit\'e Paris 13, Sorbonne Paris Cit\'e, 99 av Jean-Baptiste Clement, \\
93430 Villetaneuse, France, \\
bouali@math.univ-paris13.fr


\begin{thebibliography}{1}

\bibitem{ChinoisCrys} F.Andreatta, A.Iovitta, 
\emph{Semi-stable sheaves and comparaison isomorphisms in the semi-stable case}, 
Rencontre de seminaire mathematique de l'universite de Padova, Tome 128, 2012, p.131-286.

\bibitem{BSch} B.Bhatt, P.Scholze, \emph{The pro-etale topology for schemes},Asterisque No 369, 2015, 99-201.

\bibitem{B5} J.Bouali, \emph{The De Rham, complex and $p$ adic Hodge 
realization functor for relative motives of algebraic varieties over a field of characteristic zero}, 
arxiv 2203.10875, preprint, 2022. 

\bibitem{B6} J.Bouali \emph{Complex vs etale Abel Jacobi map for higher Chow groups and 
algebraicity of the zero locus of etale normal functions}, arxiv 2211.15317, preprint, 2022.

\bibitem{CD} D.C.Cisinski, F.Deglise, \emph{Triangulated categories of mixed motived},
arxiv preprint 0912.2110, 2009-arxiv.org.

\bibitem{Liedtke}, O.Gregory, C.Liedtke, \emph{p-adic Tate conjectures and abeloid varieties},
Doc. Math. 24, 1879-1934, 2019.

\bibitem{Illusie} L.Illusie, \emph{Complexe de de Rham-Witt et cohomologie cristalline},
Annales scientifiques de l'ENS, tome 12, 1979, p.501-661.

\bibitem{LW} F.Lecomte, N.Wach, \emph{R\'ealisation de Hodge des motifs de Voevodsky}, 
manuscripta mathematica.141,663-697, Springer-Verlag, Berlin, 2012.

\bibitem{VoeMW} C.Mazza, V.Voevodsky, C.Weibel, \emph{Lecture Notes on Motivic Cohomology},
Clay mathematics monographs, Vol 2, AMS, Providence RI, 2006.

\bibitem{Raskind} W.Raskind,
\emph{A generalized Hodge-Tate conjecture for algebraic varieties with totally degenerate reduction over p-adic fields},
Tandon, R. (eds) Algebra and Number Theory, 2005.

\bibitem{Sch} P.Scholze, \emph{p-adic Hodge theory for rigid analytic varieties}, Forum of Mathematics, Pi,1,el, 2013.

\bibitem{Voisin} C.Voisin, Hodge loci and absolute Hodge classes, Compositio Math. 143, 945-958, 2007.

\end{thebibliography}
\end{document}